\documentclass[Vancouver,Times2COL]{arxivnjd}

\usepackage{amsmath,amsfonts,amssymb,latexsym,mathtools,amsthm}
\usepackage{graphicx}
\usepackage{enumitem}
\usepackage{tikz}
\usetikzlibrary{calc,patterns,decorations.pathmorphing,decorations.markings,arrows.meta,positioning,shapes.geometric,backgrounds}
\usepackage[usenames,dvipsnames]{xcolor}
\usepackage{cite}
\usepackage{booktabs}
\usepackage{multirow}
\usepackage{makecell}
\usepackage{float}
\usepackage{algorithm}
\usepackage{algpseudocode}
\algrenewcommand{\algorithmiccomment}[1]{\hfill$\triangleright$ #1}
\usepackage{BernsteinStyle}
\allowdisplaybreaks

\newtheorem{example}{Example}
\newcommand{\exmark}{\hfill{\Large$\diamond$}}
\providecommand{\Ts}{T_\rms}
\newcommand{\Iq}{{\mathcal I}_{\rm q}}   % pitch moment of inertia; script I avoids the identity $I_n$

\articletype{Research Article}%
\received{Date Month Year}
\revised{Date Month Year}
\accepted{Date Month Year}
\journal{Int. J. Robust Nonlinear Control}
\volume{00}
\copyyear{2026}
\startpage{1}
\begin{document}
	
	\title{Iterative State- and Control-Dependent Model Predictive Control: A Jacobian-Free Formulation for Constrained Nonlinear Systems}
	
	\author[1]{Mohammadreza Kamaldar}
	
	\authormark{Kamaldar, M.}
	\titlemark{Iterative State- and Control-Dependent Model Predictive Control: A Jacobian-Free Formulation}
	
	\address[1]{\orgdiv{Department of Mechanical, Aerospace, and Biomedical Engineering}, \orgname{University of South Alabama}, \orgaddress{\state{AL}, \country{USA}}}
	
	\corres{Corresponding author: Mohammadreza Kamaldar \email{mkamaldar@southalabama.edu}}
	
	\abstract[Abstract]{This paper presents an iterative model predictive control algorithm that stabilizes constrained nonlinear systems without evaluating a single plant derivative. By factoring the exact nonlinear dynamics into a pseudo-linear form using state- and control-dependent coefficients (SCDCs), we replace the standard nonconvex optimization with a sequence of constrained linear-quadratic programs. Refreezing the coefficient matrices along the previously predicted trajectory drives the iteration. Near the origin, we prove this sequence contracts to a unique fixed point. We explicitly bound the number of iterations required to reach any stopping tolerance, and we quantify the distance from the fixed point to a true Karush-Kuhn-Tucker point, showing this optimality gap vanishes quadratically as the state approaches the origin. Inflating the discrete algebraic Riccati equation generates terminal ingredients that guarantee recursive feasibility and asymptotic stability, even when the solver terminates early. We adapt the terminal penalty online, proving it remains uniformly bounded, and we secure output feedback through the block-observable canonical form, which extracts the exact system state directly from past inputs and outputs. Retaining the block-banded structure of the subproblem forces the computational cost to scale linearly with the horizon length $\ell$. This $O(\ell)$ complexity matches the iterative linear quadratic regulator (iLQR) but sharply undercuts the $O(\ell^3)$ scaling of dense sequential quadratic programming (SQP). Numerical studies on a saturated quadrotor, a nonholonomic integrator, and a nonminimum-phase plant illustrate the theoretical bounds and map how the algorithm compares with iLQR, SQP, and linear-parameter-varying MPC.}
	
	
	\jnlcitation{\cname{%
			\author{Kamaldar M.}}.
		\ctitle{Iterative State- and Control-Dependent Model Predictive Control: A Jacobian-Free Formulation for Constrained Nonlinear Systems.} \cjournal{\it Int. J. Robust and Nonlinear Control} \cvol{2026;00(00):1--22}.}
	
	\maketitle

	\section{Introduction}
	\label{sec:intro}
	
	Constraints on states and controls are a defining feature of most control applications, and model predictive control (MPC) addresses them directly by solving a constrained optimization problem at each sampling instant and applying the first element of the resulting sequence \cite{keerthi1988optimal,kwon2006receding,camacho2013model,mayne2000constrained,rawlings2017model}. Researchers and practitioners have adopted MPC across diverse domains, including aerospace systems \cite{eren2017model}, automotive powertrains \cite{cairano2007model}, process control \cite{qin2003survey}, and robotics \cite{kuhne2004model}. Comprehensive surveys of MPC theory and computation appear in \cite{mayne2014model,allgower2012nonlinear}.
	
	When governing dynamics are nonlinear, the receding-horizon optimization generally poses a nonconvex nonlinear program that controllers must solve within a single sampling interval. Direct solution methods discretize the continuous-time optimal control problem; prominent approaches include sequential quadratic programming (SQP) \cite{nocedal2006numerical,diehl2002real}, interior-point methods \cite{wachter2006implementation}, the real-time iteration scheme \cite{diehl2005nominal}, and structure-exploiting embedded solvers \cite{frison2020hpipm,verschueren2022acados}. Alternatively, indirect methods formulate a two-point boundary value problem using the Pontryagin minimum principle \cite{ohtsuka2004continuation}. Another major class of algorithms solves a sequence of local linear-quadratic approximations of the nonlinear optimal control problem. Differential dynamic programming \cite{jacobson1970differential} and the iterative linear quadratic regulator (iLQR) \cite{LiTodorov2004,TodorovLi2005,tomizukaiLQR} typify this approach. A related and computationally efficient technique casts the nonlinear dynamics into a linear-parameter-varying (LPV) model. This approach freezes the scheduling parameter at the current operating point for the duration of the prediction horizon \cite{cisneros2018lpv,morato2020model}.
	
	This paper develops an alternative approach to nonlinear MPC based on state- and control-dependent coefficients (SCDCs). An SCDC representation writes the dynamics map as a product of a state- and control-dependent matrix pair with the state and control, which is an algebraic identity rather than an approximation, so that no derivative of the dynamics is formed and no truncation error is incurred \cite{cloutier1996nonlinear}. The guarantees established for SCDC-based feedback are local \cite{tsiotras1996counterexample,mracek1998control}, and the technique has been used in a variety of applications \cite{findeisen2003state,mracek1998control,erdem2004design,ccimen2010systematic,weiss2012forward,prach2018output}. Previous research has also integrated SCDCs with receding-horizon optimization: Chang et al.\ \cite{chang2013constrained} employ a discrete-time state-dependent Riccati equation with a backward-propagating Riccati recursion, and Ru et al.\ \cite{ru2017nonlinear} freeze the coefficients at the current state over the horizon to control a quadcopter. The conceptual dual of SCDC-based control has recently been formalized for moving-horizon estimation \cite{kamaldarSCDMHE}.
	
	The primary contribution of this paper is iterative state- and control-dependent MPC (iSCD-MPC). At each control update, the algorithm solves a sequence of constrained linear-quadratic programs by refreezing the SCDC matrices along the trajectory predicted in the previous iteration. This mechanism offers distinct structural advantages over existing methods. Unlike iLQR and SQP, the algorithm evaluates zero plant Jacobians. Unlike LPV-MPC, the coefficient matrices vary along the predicted trajectory rather than remaining frozen at the current state. Updating the matrices along the prediction horizon is critical for controlling systems that lose stabilizability at the frozen point, such as kinematic models subject to Brockett's obstruction \cite{brockett1983asymptotic}.
	
	The specific contributions are as follows.
	First, we formulate the iSCD-MPC algorithm. We detail its block-banded quadratic-program (QP) structure, introduce a warm-starting procedure, and develop an online terminal-penalty adaptation that serves as the control-theoretic dual of the arrival-cost update in moving-horizon estimation \cite{Rao2003,Alessandri2008}.
	
	Second, we provide a rigorous theoretical analysis. We establish the existence and uniqueness of a fixed point for the SCDC iteration and derive a contraction estimate whose rate is proportional to the norm of the current state (Proposition~\ref{prop:contraction}). We explicitly bound the number of iterations required to meet a stopping tolerance (Proposition~\ref{prop:iteration_bound}) and prove the geometric convergence of the predicted cost (Proposition~\ref{prop:cost_decrease}). We then bound the distance between the fixed point and a Karush-Kuhn-Tucker (KKT) point of the underlying nonlinear program, demonstrating that this suboptimality gap vanishes quadratically at the origin (Proposition~\ref{prop:kkt_gap}). We show uniform boundedness of the adapted terminal penalty (Proposition~\ref{prop:Pk_bounded}), guarantee recursive feasibility (Theorem~\ref{thm:feasibility}), and establish asymptotic stability under both optimal (Theorem~\ref{thm:stability}) and early-terminated suboptimal (Theorem~\ref{thm:relaxed_lyapunov}) control sequences. The suboptimal analysis follows the framework of \cite{scokaert1999suboptimal}.
	
	Third, we extend the method to output-feedback control using the block-observable canonical form (BOCF) \cite{polderman1989state,islamPCAC}. This formulation, introduced preliminarily in \cite{kamaldar2023arxiv}, acts as a deadbeat observer and transfers the state-feedback stability guarantees to the output-feedback setting without requiring a separation argument.
	
	Fourth, we derive an analytical complexity comparison, demonstrating how the per-step cost of iSCD-MPC scales against iLQR, SQP, and LPV-MPC.
	
	Fifth, we present a numerical study on three distinct benchmarks designed around the theoretical results. We control a planar quadrotor subject to thrust and torque saturation to validate contraction, cost-convergence, and terminal-penalty bounds. We stabilize a nonholonomic integrator to demonstrate convergence where frozen-coefficient LPV-MPC structurally fails. Finally, we execute output-feedback control on a nonminimum-phase sampled-data triple integrator with asymmetric saturation to validate the BOCF architecture.
	
	A preliminary version of the algorithm appears in the preprint \cite{kamaldar2023arxiv}, which recent studies apply to six-degree-of-freedom aerospace guidance \cite{alhazmi2025nonlinear} and solve via root-finding methods \cite{alhazmi2025root}. The present paper formalizes the algorithm and supersedes that preprint: none of the theoretical results appears there, and the numerical study is new.
	
	The remainder of this paper is organized as follows. Section~\ref{sec:notation} presents notation. Section~\ref{sec:problem} formulates the receding-horizon control problem. Section~\ref{sec:algorithm} details the iSCD-MPC algorithm. Section~\ref{sec:OF} covers the BOCF-based output-feedback architecture. Section~\ref{sec:theory} provides the theoretical analysis, and Section~\ref{sec:complexity} compares computational complexity. Section~\ref{sec:examples} presents the numerical study, and Section~\ref{sec:conclusions} concludes the paper.
	
	\section{Notation}
	\label{sec:notation}
	
	Let $\BBR$, $\BBN$, and $\BBZ^+$ denote the sets of real numbers, nonnegative integers, and positive integers, respectively.
	Let $\|\cdot\|$ denote the Euclidean norm on $\BBR^n$ and the corresponding induced norm on $\BBR^{m\times n}$, that is, $\|M\|\triangleq \sup_{x\ne 0}\|Mx\|/\|x\|$.
	Let $x_{(i)}$ denote the $i$th component of $x\in\BBR^n$.
	The symmetric matrix $P\in\BBR^{n\times n}$ is positive semidefinite (respectively, positive definite) if all of its eigenvalues are nonnegative (respectively, positive), and we write $P\succeq 0$ (respectively, $P\succ 0$).
	For symmetric $P$, $\lambda_{\max}(P)$ and $\lambda_{\min}(P)$ denote the maximum and minimum eigenvalues of $P$.
	For $\SX,\SY\subseteq\BBR^n$, the Minkowski sum is $\SX\oplus\SY\triangleq \{x+y\colon x\in\SX,\ y\in\SY\}$, and the Pontryagin difference is $\SX\ominus\SY\triangleq \{x\in\BBR^n \colon x+y\in\SX \mbox{ for all } y\in\SY\}$.
	For a differentiable function $h\colon\BBR^n\times\BBR^m\to\BBR^q$, $\nabla_x h(x,u)\in\BBR^{q\times n}$ and $\nabla_u h(x,u)\in\BBR^{q\times m}$ denote the partial Jacobians of $h$ with respect to its first and second arguments, respectively.
	In Algorithm~\ref{alg:iscd_mpc}, the symbol $\gets$ denotes assignment.
	Finally, ${\rm int}\,\SX$ denotes the interior of $\SX$, $I_n$ denotes the $n\times n$ identity matrix, and $e_i\in\BBR^n$ denotes the $i$th column of $I_n$.
	
	\section{Problem Formulation}
	\label{sec:problem}
	
	Consider the discrete-time nonlinear system
	\begin{equation}
		x_{k+1}= f(x_k,u_k),
		\label{eq:xk1}
	\end{equation}
	where, for all $k\in\BBN$, $x_k\in\BBR^n$ is the state, $x_{0}\in\BBR^n$ is the initial condition, $u_k\in\BBR^m$ is the control, $u_0\in\BBR^m$ is the given initial control, and $f\colon \BBR^n\times \BBR^m\to\BBR^{n}$.
	Until Section~\ref{sec:OF}, we assume that the state $x_k$ is available for feedback; Section~\ref{sec:OF} treats output feedback.
	
	Let $\ell\ge2$ be the horizon length.
	For all $k\in\BBN$ and $j\in\{1,\ldots,\ell\}$, let $x_{k,j}$ denote the predicted state for step $k+j$ computed at step $k$ using
	\begin{equation}
		x_{k,j+1} = f(  x_{k,j},  u_{k,j} ),
		\label{eq:xfka}
	\end{equation}
	where $x_{k,0}\triangleq x_k$, $u_{k,0}\triangleq u_k$, and, for all $j\in\{1,\ldots,\ell-1\}$, $u_{k,j}$ is the predicted control for step $k+j$ computed at step $k$.
	Note that, for all $k\in\BBN$, \eqref{eq:xfka} with $j=0$ implies that
	\begin{equation}
		x_{k,1} = f( x_{k},u_{k} ),
		\label{eq:xk1_pred}
	\end{equation}
	which is determined by the current state $x_k$ and the current control $u_k$ and is thus independent of the predicted controls $u_{k,1},\ldots,u_{k,\ell-1}$.
	
	For all $k\in\BBN$, consider the cost $J_k\colon\BBR^{m(\ell-1)}\to[0,\infty)$ defined by
	\begin{align}
		J_k&(u_{k,1},\ldots,u_{k,\ell-1}) \triangleq  \tfrac{1}{2} x_{k,\ell}^\rmT P_k x_{k,\ell}\nn\\
		&\quad +\tfrac{1}{2} \sum_{j=1}^{\ell-1} \(x_{k,j}^\rmT Q_{k,j} x_{k,j} + u_{k,j}^\rmT R_{k,j} u_{k,j}\),
		\label{Jdefn}
	\end{align}
	where $x_{k,1},\ldots,x_{k,\ell}$ are given by \eqref{eq:xfka}; $P_k\in\BBR^{n\times n}$ is the positive-definite terminal weighting; and, for all $j\in\{1,\ldots,\ell-1\}$, $Q_{k,j}\in\BBR^{n\times n}$ is positive semidefinite and $R_{k,j}\in\BBR^{m\times m}$ is positive definite.
	
	At each step $k\in\BBN$, the objective is to determine a control sequence $u_{k,1},\ldots,u_{k,\ell-1}$ that minimizes \eqref{Jdefn} subject to the prediction dynamics \eqref{eq:xfka} and to state and control constraints.
	To express the constraints compactly, we stack the predicted states and controls in the vector
	\begin{equation*}
		v_k\triangleq \matl x_{k,1}^\rmT & \cdots & x_{k,\ell}^\rmT&u_{k,1}^\rmT&\cdots&u_{k,\ell-1}^\rmT\matr^\rmT \in\BBR^{n_v},
	\end{equation*}
	where $n_v\triangleq \ell(n+m)-m$, and we consider the constraints
	\begin{gather}
		\SA v_k \le b,\label{eq:ineqcon}\\
		\SA_{\rm eq} v_k = b_{\rm eq},\nonumber\\
		\underline v_{(s)}\le v_{k(s)}\le \overline v_{(s)},\quad s\in\{1,\ldots,n_v\},\label{eq:boxcon}
	\end{gather}
	where $\SA\in\BBR^{n_\rmc\times n_v}$, $b\in\BBR^{n_\rmc}$, $\SA_{\rm eq}\in\BBR^{n_\rme\times n_v}$, $b_{\rm eq} \in \BBR^{n_\rme}$, $\underline v,\overline v\in\BBR^{n_v}$, and $n_\rmc,n_\rme\in\BBN$.
	
	In accordance with receding-horizon control, only the first element of the computed sequence is applied, that is,
	\begin{equation}
		u_{k+1} = u_{k,1},
		\label{eq:uk1_applied}
	\end{equation}
	and $u_{k,2},\ldots,u_{k,\ell-1}$ are discarded.
	
	\begin{remark}\rm
		\label{rem:delay}
		The update \eqref{eq:uk1_applied} reflects a one-step computational delay: the control applied at step $k+1$ is computed during the interval between steps $k$ and $k+1$, using the state $x_k$ measured at step $k$ and the control $u_k$ applied at step $k$.
		Consequently, the first predicted state $x_{k,1}=f(x_k,u_k)$ is fixed, the earliest control that the optimization can influence is $u_{k,1}$, and \eqref{eq:xk1_pred} and \eqref{eq:uk1_applied} imply that $x_{k+1}=x_{k,1}$.
		This convention accounts for the computation time of the optimization in real-time implementation; the same convention is used in \cite{islamPCAC,kamaldarSCDMHE}.
	\end{remark}
	
	\section{\lowercase{i}SCD-MPC Algorithm}
	\label{sec:algorithm}
	
	This section reformulates \eqref{eq:xk1} using SCDCs and presents the iterative state- and control-dependent model predictive control (iSCD-MPC) algorithm; Section~\ref{sec:terminal_adaptation} presents an online adaptation of the terminal weighting $P_k$ in \eqref{Jdefn} based on a forward-propagating Riccati equation.
	Let $A\colon \BBR^n\times \BBR^m\to \BBR^{n\times n}$ and $B\colon \BBR^n\times \BBR^m \to \BBR^{n\times m}$ be such that, for all $x\in\BBR^n$ and $u\in\BBR^m$,
	\begin{equation}
		f(x,u) = A(x,u)\,x+B(x,u)\,u.
		\label{eq:fSCDC}
	\end{equation}
	Note that \eqref{eq:fSCDC} implies that $f(0,0)=0$, and thus the origin is an equilibrium of \eqref{eq:xk1}.
	If $n>1$, then the factorization \eqref{eq:fSCDC} is not unique \cite{ccimen2010systematic}, and thus the choice of $(A,B)$ is a design variable.

	Let $\rho\in\BBZ^+$ denote the maximum number of iterations, and let $i\in\{1,\ldots,\rho\}$ denote the iteration index.
	For all $k\in\BBN$, $j\in\{1,\hdots,\ell\}$, and $i\in\{1,\ldots,\rho\}$, let $x_{k,j|i}\in\BBR^n$ and $u_{k,j|i} \in\BBR^m$ denote the predicted state and predicted control for step $k+j$ computed at step $k$ and iteration $i$.
	The predicted states are generated by the nonlinear rollout
	\begin{equation}
		x_{k,j+1|i}= f(x_{k,j|i},u_{k,j|i}) = A_{k,j|i}\, x_{k,j|i}+B_{k,j|i}\, u_{k,j|i},
		\label{eq:xkj+1i}
	\end{equation}
	where
	\begin{equation}
		A_{k,j|i} \triangleq A(x_{k,j|i},u_{k,j|i}),\quad B_{k,j|i} \triangleq B(x_{k,j|i},u_{k,j|i}),
		\label{eq:AB_SCDC}
	\end{equation}
	and
	\begin{equation}
		x_{k,0|i}\triangleq x_k,\qquad u_{k,0|i} \triangleq u_k.
		\label{eq:init_cond}
	\end{equation}
	It follows from \eqref{eq:xk1_pred}--\eqref{eq:init_cond} that, for all $i\in\{1,\ldots,\rho\}$,
	\begin{equation*}
		x_{k,1|i} = f(x_k,u_k) = x_{k,1},
	\end{equation*}
	that is, the first predicted state is the same at every iteration.
	
	\subsection{Iterative QP}
	\label{sec:iterQP}
	
	For all $k\in\BBN$ and $i\in\{2,\hdots,\rho\}$, let $u_{k,1|i},\ldots,u_{k,\ell-1|i}$ be the solution of
	\begin{align}
		\min_{\mu_{1},\ldots,\mu_{\ell-1}} \Bigg(
		\tfrac{1}{2} \xi_{\ell}^\rmT P_k \xi_{\ell}+\tfrac{1}{2} &\sum_{j=1}^{\ell-1} \( \xi_{j}^\rmT Q_{k,j} \xi_{j} + \mu_{j}^\rmT R_{k,j} \mu_{j}\) \Bigg),
		\label{eq:QP_cost}
	\end{align}
	subject to
	\begin{gather}
		\xi_{1}= x_{k,1},\nonumber\\
		\xi_{j+1}= A_{k,j|i-1}\, \xi_{j}+B_{k,j|i-1}\, \mu_{j},\quad j\in\{1,\ldots,\ell-1\},\label{eq:QP_dyn}\\
		\SA \nu \le b,\qquad
		\SA_{\rm eq}  \nu = b_{\rm eq},\label{eq:QP_ineq}\\
		\underline v_{(s)}\le  \nu_{(s)}\le \overline v_{(s)},\quad s\in\{1,\ldots,n_v\},\label{eq:QP_box}
	\end{gather}
	where $\nu \triangleq  \matl \xi_{1}^\rmT& \cdots & \xi_{\ell}^\rmT&\mu_{1}^\rmT&\cdots&\mu_{\ell-1}^\rmT\matr^\rmT$.
	At iteration $i\ge 2$, the SCDC matrices $A_{k,j|i-1}$ and $B_{k,j|i-1}$ in \eqref{eq:QP_dyn} are frozen at the values obtained from the rollout \eqref{eq:xkj+1i} under the previous iterate, and thus \eqref{eq:QP_cost}--\eqref{eq:QP_box} is a constrained linear-quadratic program.
	Note that no Jacobian of $f$ is evaluated at any point of the algorithm; the linear prediction model \eqref{eq:QP_dyn} is obtained by refreezing the exact factorization \eqref{eq:fSCDC} along the previously predicted trajectory.
	
	Let $\varepsilon\in(0,\infty)$ denote the stopping tolerance, and define the stacked control vector
	\begin{equation*}
		U_{k|i} \triangleq \matl u_{k,1|i}^\rmT &\cdots &u_{k,\ell-1|i}^\rmT \matr^\rmT \in\BBR^{m(\ell-1)}.
	\end{equation*}
	The number of iterations performed at step $k$ is
	\begin{equation}
		\rho_k\triangleq \min\big\{\rho,\ \min\{i\in\{2,\ldots,\rho\}\colon \|U_{k|i}-U_{k|i-1}\|<\varepsilon\}\big\},
		\label{eq:rhokk}
	\end{equation}
	where the inner minimum is defined to be $\rho$ if the tolerance is not met for any $i\in\{2,\ldots,\rho\}$.
	In accordance with \eqref{eq:uk1_applied}, for all $k\in\BBN$, the applied control is
	\begin{equation}
		u_{k+1}=u_{k,1|\rho_k}.
		\label{eq:applied_control}
	\end{equation}
	
	\subsection{QP Block-Matrix Structure}
	\label{sec:QP_block}
	
	Define the stacked state vector
	\begin{equation*}
		X_{k|i}\triangleq \matl x_{k,1|i}^\rmT&\cdots& x_{k,\ell|i}^\rmT\matr^\rmT\in\BBR^{n\ell}.
	\end{equation*}
	For all $k\in\BBN$ and $i\in\{2,\ldots,\rho\}$, the constraint \eqref{eq:QP_dyn}, whose coefficients are frozen at iteration $i-1$, is equivalent to
	\begin{equation}
		X_{k|i} = \Phi_{k|i-1}\, x_{k,1} + \Gamma_{k|i-1}\, U_{k|i},
		\label{eq:X_block}
	\end{equation}
	where,  for all $i\in\{1,\ldots,\rho\}$,
	\begin{equation}
		\Phi_{k|i}\triangleq \matl I_n  \\ A_{k,1|i}\\\vdots\\ A_{k,\ell-1|i}\cdots A_{k,1|i}\matr\in\BBR^{n\ell\times n},
		\label{eq:Phidef}
	\end{equation}
	and $\Gamma_{k|i}\in\BBR^{n\ell\times m(\ell-1)}$ is the block-lower-triangular matrix whose $(r,c)$ block, for $r\in\{2,\ldots,\ell\}$ and $c\in\{1,\ldots,r-1\}$, is
	\begin{equation}
		\big[\Gamma_{k|i}\big]_{r,c} \triangleq
		\begin{cases}
			B_{k,r-1|i}, & c=r-1,\\
			A_{k,r-1|i}\cdots A_{k,c+1|i}\, B_{k,c|i}, & c<r-1,
		\end{cases}
		\label{eq:Gammadef}
	\end{equation}
	and whose remaining blocks are zero.
	For all $k\in\BBN$, define the block-diagonal weightings
	\begin{align*}
		\SQ_k&\triangleq {\rm diag}(Q_{k,1},\ldots,Q_{k,\ell-1},P_k)\in\BBR^{n\ell\times n\ell},\\
		\SR_k&\triangleq {\rm diag}(R_{k,1},\ldots,R_{k,\ell-1})\in\BBR^{m(\ell-1)\times m(\ell-1)}.
	\end{align*}
	Substituting \eqref{eq:X_block} into \eqref{eq:QP_cost} yields the reduced Hessian
	\begin{equation}
		\SH_{k|i} \triangleq \SR_k + \Gamma_{k|i-1}^\rmT \SQ_k\, \Gamma_{k|i-1}.
		\label{eq:Hessian}
	\end{equation}
	Since $\SR_k\succ0$ and $\SQ_k\succeq 0$, it follows from \eqref{eq:Hessian} that
	\begin{equation*}
		\SH_{k|i}\succeq \lambda_{\min}(\SR_k)\,I_{m(\ell-1)}\succ 0,
	\end{equation*}
	and thus \eqref{eq:QP_cost}--\eqref{eq:QP_box} is a strictly convex QP.
	In the case where the constraints \eqref{eq:QP_ineq} and \eqref{eq:QP_box} are absent or inactive, the QP has the explicit solution
	\begin{equation}
		U_{k|i} = -\SH_{k|i}^{-1}\Gamma_{k|i-1}^\rmT \SQ_k\,\Phi_{k|i-1}\, x_{k,1},
		\label{eq:Hmap}
	\end{equation}
	which defines the iteration map $\SM\colon \BBR^n\times\BBR^{m(\ell-1)}\to\BBR^{m(\ell-1)}$ by
	\begin{equation}
		U_{k|i} = \SM(x_{k,1},U_{k|i-1}).
		\label{eq:Hmap_def}
	\end{equation}
	The step index $k$ enters \eqref{eq:Hmap_def} only through the anchor $x_{k,1}$ and, in the case of time-varying weightings, through $\SQ_k$ and $\SR_k$; for the constant weightings used in Section~\ref{sec:theory}, $\SM$ is independent of $k$.
	The block-banded structure underlying \eqref{eq:Hessian} permits a Riccati-based factorization of each QP with computational cost $O(\ell(n+m)^3)$ \cite{rao1998application}; an analytical comparison with other nonlinear MPC methods is given in Section~\ref{sec:complexity}.
	
	\subsection{Terminal Control Law}
	\label{sec:terminal}
	
	The warm start of Section~\ref{sec:stopping} and the terminal ingredients of the stability analysis both use a terminal control law $\kappa_\rmf\colon\BBR^n\to\BBR^m$, which is constructed here.
	Define $A_0\triangleq A(0,0)$ and $B_0\isdef B(0,0)$. 
	Following the weight-inflation technique of \cite{chen1998quasi}, let $Q\succ0$, assume that  $(A_0,B_0)$ is stabilizable, and let $P\succ 0$ be the unique stabilizing solution of the inflated discrete algebraic Riccati equation
	\begin{equation}
		P = A_0^\rmT P A_0 - A_0^\rmT P B_0(2R+B_0^\rmT P B_0)^{-1}B_0^\rmT P A_0+2Q.
		\label{eq:DARE}
	\end{equation}
	Define
	\begin{gather}
		K_\rmf\triangleq -(2R+B_0^\rmT P B_0)^{-1}B_0^\rmT P A_0,\label{eq:Kf_def}\\
		\kappa_\rmf(x)\triangleq K_\rmf x,\label{eq:kappaf_def}\\
		A_{\rm cl}\triangleq A_0+B_0 K_\rmf.\label{eq:Acl_def}
	\end{gather}
	Note that using \eqref{eq:DARE}--\eqref{eq:Acl_def} and completion of squares implies 
	\begin{equation}
		A_{\rm cl}^\rmT P A_{\rm cl} - P =  -Q_\star,
		\label{eq:Riccati_identity}
	\end{equation}
	where $Q_\star\isdef2Q + 2K_\rmf^\rmT R K_\rmf \succ 0$ because $Q\succ 0$.
	
	\begin{remark}\rm
		\label{rem:inflation}
		Since the discrete algebraic Riccati equation is positively homogeneous in $(P,Q,R)$, the solution of \eqref{eq:DARE} is $P = 2\bar P$, where $\bar P$ is the stabilizing solution associated with the weights $(Q,R)$, and $-K_\rmf$, with $K_\rmf$ given by \eqref{eq:Kf_def}, equals the linear-quadratic-regulator gain associated with $(Q,R)$.
		Hence, the inflation doubles the terminal penalty without altering the terminal control law; the factor of two absorbs the nonlinear remainder in the proof of Proposition~\ref{prop:terminal_set}, as in the quasi-infinite-horizon construction of \cite{chen1998quasi}.
	\end{remark}
	
	The terminal set associated with $\kappa_\rmf$, together with the invariance and descent properties used in the stability analysis, is constructed in Section~\ref{sec:terminal_set}.
	
	\subsection{Warm Starting}
	\label{sec:stopping}
	
	For $k=0$, initialize
	\begin{equation}
		u_{0,j|1} \triangleq u_0,\quad j \in\{ 1,\ldots,\ell-1\}.
		\label{eq:u0_init}
	\end{equation}
	For all $k\ge1$ and $j \in\{ 1,\ldots,\ell-1\}$, define
	\begin{equation}
		u_{k,j|1} \triangleq \begin{cases}
			u_{k-1,j+1|\rho_{k-1}}, & j\in\{1,\ldots,\ell-2\},\\
			\kappa_\rmf\big(x_{k,\ell-1|1}\big), & j=\ell-1,
		\end{cases}
		\label{eq:warm_start}
	\end{equation}
	where $\kappa_\rmf$ is the terminal control law  constructed in Section~\ref{sec:terminal}, and $x_{k,\ell-1|1}$ is obtained from the rollout \eqref{eq:xkj+1i} under the shifted controls $u_{k,1|1},\ldots,u_{k,\ell-2|1}$.
	The warm start \eqref{eq:warm_start} shifts the previously computed control sequence forward by one step and appends the terminal control law; this construction mirrors the shifted candidate sequence used in the feasibility and stability analysis of Section~\ref{sec:theory} and is analogous to the warm start used in the real-time iteration scheme \cite{diehl2002real} and in moving-horizon estimation \cite{Alessandri2008,kamaldarSCDMHE}.
	
	\subsection{Online Terminal-Penalty Adaptation}
	\label{sec:terminal_adaptation}
	
	In moving-horizon estimation, the arrival cost compresses the discarded data into a quadratic penalty that is updated at each step by a Riccati recursion \cite{Rao2003,rawlings2017model}.
	As the control-theoretic dual, we update the terminal penalty $P_k$ in \eqref{Jdefn} online by the forward-propagating Riccati equation \cite{weiss2012forward,prach2018output}
	\begin{align}
		P_{k+1} &= \bar A_k^\rmT P_k \bar A_k\nn\\
		&\quad - \bar A_k^\rmT P_k \bar B_k (\bar B_k^\rmT P_k \bar B_k + R)^{-1} \bar B_k^\rmT P_k \bar A_k + Q,
		\label{eq:Pk_update}
	\end{align}
	where $P_0\succ 0$, $Q\succ 0$, $R\succ 0$, and
	\begin{equation}
		\bar A_k\triangleq A(x_k,u_k),\qquad \bar B_k\triangleq B(x_k,u_k)
		\label{eq:barAB}
	\end{equation}
	are the SCDC matrices at the current operating point.
	Note that, in the case where $(\bar A_k,\bar B_k)$ is constant, the sequence $(P_k)_{k=0}^\infty$ generated by \eqref{eq:Pk_update} converges to the stabilizing solution of the corresponding discrete algebraic Riccati equation \cite[Chap.~4]{Anderson1979}.
	Proposition~\ref{prop:Pk_bounded} shows that \eqref{eq:Pk_update} is uniformly bounded above and below along the closed-loop trajectory.
	Figure~\ref{fig:iscd_architecture} illustrates the two-loop execution architecture of iSCD-MPC, separating the outer receding-horizon update from the inner Jacobian-free optimization.
	Algorithm~\ref{alg:iscd_mpc} summarizes the complete method.
	
	\begin{algorithm}[t]
		\caption{iSCD-MPC with Warm Starting}
		\label{alg:iscd_mpc}
		\begin{algorithmic}[1]
			\Require horizon $\ell$, maximum iterations $\rho$, tolerance $\varepsilon$, weightings $Q_{k,j}$, $R_{k,j}$, $P_0$, terminal law $\kappa_\rmf$, initial control $u_0$.
			\State $u_{0,j|1}\gets u_0$ for $j\in\{1,\ldots,\ell-1\}$
			\For{$k = 0, 1, 2, \ldots$}
			\State Measure $x_k$ 
			\State $x_{k,1}\gets f(x_k,u_k)$
			\If{$k\ge 1$}
			\State $u_{k,j|1}\gets u_{k-1,j+1|\rho_{k-1}}$ for $j\in\{1,\ldots,\ell-2\}$
			\State Propagate $x_{k,\ell-1|1}$ by \eqref{eq:xkj+1i}
			\State$u_{k,\ell-1|1}\gets \kappa_\rmf(x_{k,\ell-1|1})$
			\EndIf
			\For{$i = 2,3,\ldots,\rho$}
			\State Roll out $x_{k,j|i-1}$ by \eqref{eq:xkj+1i}; evaluate \eqref{eq:AB_SCDC}
			\State Solve QP \eqref{eq:QP_cost}--\eqref{eq:QP_box} to obtain $U_{k|i}$
			\If{$\|U_{k|i}-U_{k|i-1}\|<\varepsilon$} \textbf{break}
			\EndIf
			\EndFor
			\State $\rho_k\gets i$; apply $u_{k+1}\gets u_{k,1|\rho_k}$
			\State Update $P_{k+1}$ by \eqref{eq:Pk_update}
			\EndFor
		\end{algorithmic}
	\end{algorithm}
	
	\begin{figure*}[t!]
		\centering
		\begin{tikzpicture}[
			>=Stealth,
			% --- Compact & Legible Functional Block Styles ---
			plantbox/.style={
				rectangle,
				draw=black!85,
				very thick,
				fill=black!8,
				rounded corners=3pt,
				minimum width=3.9cm,
				minimum height=1.6cm,
				align=center,
				font=\small
			},
			execbox/.style={
				rectangle,
				draw=teal!70!black,
				very thick,
				fill=teal!7,
				rounded corners=3pt,
				minimum width=3.9cm,
				minimum height=1.6cm,
				align=center,
				font=\small
			},
			warmbox/.style={
				rectangle,
				draw=orange!80!black,
				very thick,
				fill=orange!9,
				rounded corners=3pt,
				minimum width=4.1cm,
				minimum height=1.6cm,
				align=center,
				font=\small
			},
			riccatibox/.style={
				rectangle,
				draw=blue!40!black,
				very thick,
				fill=blue!8,
				rounded corners=3pt,
				minimum width=4.1cm,
				minimum height=1.6cm,
				align=center,
				font=\small
			},
			rolloutbox/.style={
				rectangle,
				draw=blue!40!black,
				very thick,
				fill=blue!8,
				rounded corners=3pt,
				minimum width=4.3cm,
				minimum height=1.6cm,
				align=center,
				font=\small
			},
			qpbox/.style={
				rectangle,
				draw=blue!40!black,
				very thick,
				fill=blue!12,
				rounded corners=3pt,
				minimum width=4.3cm,
				minimum height=1.6cm,
				align=center,
				font=\small
			},
			decision/.style={
				diamond,
				draw=red!80!black,
				very thick,
				fill=red!8,
				aspect=2.2,
				align=center,
				inner sep=1.5pt,
				font=\small
			},
			line/.style={draw, thick, ->},
			labeltext/.style={font=\small, text=black!85}
			]
			
			% ================= COLUMN 1: PLANT & EXECUTION (x = 0) =================
			\node[plantbox] (plant) at (0, 0) {
				\textbf{\normalsize Controlled Plant} \\[3pt]
				Dynamic Flow Eq.~\eqref{eq:xk1} \\
				BOCF Observer, Section~\ref{sec:OF}
			};
			
			\node[execbox] (apply) at (0, -4.6) {
				\textbf{\normalsize Control Execution} \\[3pt]
				Apply Update Eq.~\eqref{eq:applied_control} \\
				Discard Candidate Tail
			};
			
			% ================= COLUMN 2: WARM START & ADAPTATION (x = 5.3) =================
			\node[warmbox] (warm) at (5.3, 0) {
				\textbf{\normalsize Shifted Warm Start} \\[3pt]
				Candidate Shift Eq.~\eqref{eq:warm_start} \\
				Append Terminal Law Eq.~\eqref{eq:kappaf_def}
			};
			
			\node[riccatibox] (riccati) at (5.3, -2.3) {
				\textbf{\normalsize Terminal Adaptation} \\[3pt]
				Riccati Recursion Eq.~\eqref{eq:Pk_update} \\
				Operating Matrices Eq.~\eqref{eq:barAB}
			};
			
			% ================= COLUMN 3: INNER SCDC ITERATION (x = 11.2) =================
			\node[rolloutbox] (rollout) at (11.2, 0) {
				\textbf{\normalsize Pseudo-Linear Rollout} \\[3pt]
				SCDC Trajectory Eq.~\eqref{eq:xkj+1i} \\
				Factorization Map Eq.~\eqref{eq:AB_SCDC} \\[2pt]
				\bfseries\color{blue!40!black} (No Jacobians evaluated)
			};
			
			\node[qpbox] (qp) at (11.2, -2.3) {
				\textbf{\normalsize Block-Banded QP} \\[3pt]
				Reduced Hessian Eq.~\eqref{eq:Hessian} \\
				Iteration Map Eq.~\eqref{eq:Hmap_def}
			};
			
			\node[decision] (check) at (11.2, -4.6) {
				\textbf{\normalsize $\|U_{k|i} - U_{k|i-1}\| < \varepsilon$} \\
				or iteration cap $i = \rho$?
			};
			
			% ================= COMPACT SYMMETRICAL CONTAINER =================
			% Bounding box shifted right to x=8.5 so "Seed U_{k|1}" clears the dashed line.
			% Right edge extended to 14.7 for perfect symmetry and arrow clearance.
			\begin{pgfonlayer}{background}
				\draw[dashed, very thick, draw=blue!40!black, fill=blue!3, rounded corners=8pt]
				(8.5, 1.45) rectangle (14.7, -5.7);
				% Title anchored at the exact horizontal center of the box (x = 11.6)
				\node[font=\footnotesize\bfseries, text=blue!40!black, anchor=center] 
				at (11.6, 0.98) {Inner SCDC Contraction Loop ($i = 2, \ldots, \rho$)};
			\end{pgfonlayer}
			
			% ================= ORTHOGONAL ROUTING & INTERCONNECTIONS =================
			% Horizontal: Plant to Warm start
			\draw[line] (plant.east) -- node[above=2pt, labeltext] {\textbf{State} $x_k$} (warm.west);
			
			% Horizontal: Warm start to Rollout
			% pos=0.4 pulls the text leftward, keeping it safely away from the dashed boundary
			\draw[line] (warm.east) -- node[above=2pt, pos=0.4, labeltext] {\textbf{Seed} $U_{k|1}$} (rollout.west);
			
			% Vertical: Rollout down to QP
			\draw[line] (rollout.south) -- node[right=2pt, labeltext] {Refrozen Matrices} (qp.north);
			
			% Vertical: QP down to Check
			\draw[line] (qp.south) -- node[right=2pt, labeltext] {Iterate $U_{k|i}$} (check.north);
			
			% Horizontal: Check YES to Apply (Exactly Level at y = -4.6)
			\draw[line, blue!40!black] (check.west) -- node[above=2pt, font=\small\bfseries, text=blue!40!black] {Yes ($\rho_k = i$)} (apply.east);
			
			% Vertical: Apply up to Plant (Exactly Vertical at x = 0)
			\draw[line] (apply.north) -- node[left=2pt, labeltext] {\textbf{Control} $u_{k+1}$} (plant.south);
			
			% Orthogonal State-Tap to Online Riccati Adaptation
			\coordinate (tap) at (2.65, 0);
			\fill (tap) circle (1.8pt);
			\draw[line, dotted, very thick, blue!40!black] (tap) |- (riccati.west);
			\draw[line, dotted, very thick, blue!40!black] (riccati.north) -- node[right=2pt, font=\small, text=blue!40!black] {\textbf{Weight} $P_{k+1}$} (warm.south);
			
			% Orthogonal Right Retry Arrow
			% The label is positioned below the horizontal segment exiting the diamond
			\draw[line, red!80!black] (check.east) -- node[below=1pt, pos=0.4, font=\small\bfseries, text=red!85!black] {No ($i \gets i+1$)} ++(0.85, 0) |- (rollout.east);
			
		\end{tikzpicture}
		\caption{Algorithmic architecture of iSCD-MPC. The outer receding-horizon loop  advances the sampling step $k$, updates the SCDC terminal penalty $P_k$ via the forward Riccati recursion \eqref{eq:Pk_update}, and warm-starts the prediction by shifting the previous control sequence forward and appending the terminal control law \eqref{eq:warm_start}. The inner contraction loop  refreezes the exact pseudo-linear coefficients along the rollout of the previous iterate without evaluating plant Jacobians, converging locally at a rate proportional to the norm of the current state (Proposition~\ref{prop:contraction}).}
		\label{fig:iscd_architecture}
	\end{figure*}
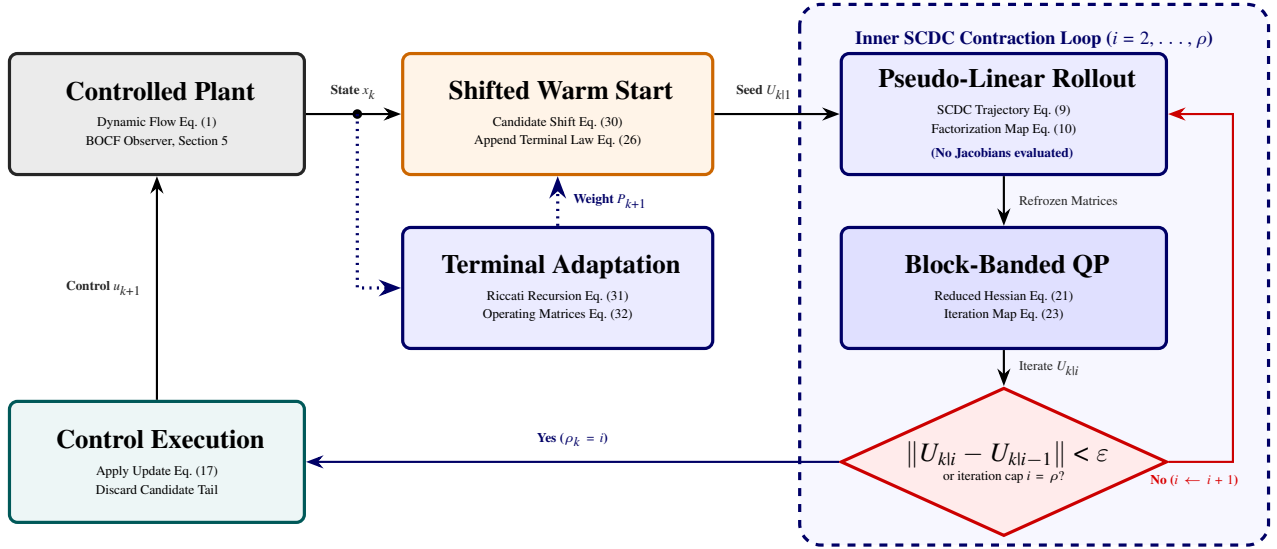
	
	\section{Output-Feedback Control}
	\label{sec:OF}
	
	State-feedback predictive control requires the full state $x_k$ at every step.
	For a general nonlinear measurement map, output feedback requires a nonlinear asymptotic observer, and the state-feedback guarantees are not preserved unless a separation principle applies.
	We avoid the general observer problem by restricting attention to systems that admit an exact input-output realization, for which the block-observable canonical form (BOCF) provides deadbeat state reconstruction from a finite history of inputs and outputs.
	
	Consider the discrete-time input-output system
	\begin{align}
		y_{k} &=-\sum_{\tau=1}^n F_{\tau,k}\, y_{k-\tau} + \sum_{\tau=1}^n G_{\tau,k}\, u_{k-\tau}, \label{eq:yyy}
	\end{align}
	where $y_k\in\BBR^p$ is the output, $u_k\in\BBR^m$ is the control, and, for all $\tau\in\{1,\ldots,n\}$ and $k\in\BBN$,
	\begin{align}
		F_{\tau,k} &\triangleq f_\tau(y_{k-1},\ldots,y_{k-n},u_{k-1},\ldots, u_{k-n}), \label{eq:Fk_def}\\
		G_{\tau,k} &\triangleq g_\tau(y_{k-1},\ldots,y_{k-n},u_{k-1},\ldots, u_{k-n}),\label{eq:Gk_def}
	\end{align}
	where $f_\tau\colon \BBR^{np}\times \BBR^{nm}\to\BBR^{p\times p}$ and $g_\tau\colon \BBR^{np}\times \BBR^{nm} \to\BBR^{p\times m}$.
	The BOCF realization of \eqref{eq:yyy}, which is of the SCDC form \eqref{eq:fSCDC} with coefficients determined by \eqref{eq:Fk_def} and \eqref{eq:Gk_def}, is given by \cite{polderman1989state,islamPCAC}
	\begin{align}
		x_{k+1} &= A(x_k,u_k)\,x_{k} + B(x_k,u_k)\, u_{k},\label{eq:xk1_OF}\\
		y_{k} & = C x_{k},\label{eq:yk_OF}
	\end{align}
	where
	\begin{gather}
		A(x_k,u_k) \triangleq \matl -F_{1,k+1} & I_p & 0 & \cdots & 0  \\
		-F_{2,k+1} & 0 & I_p &\cdots  & 0\\
		\vdots & \vdots & \vdots & \ddots & \vdots \\
		-F_{n-1,k+1} & 0 & 0 & \cdots & I_p\\
		-F_{n,k+1} & 0 & 0 & \cdots & 0
		\matr,   \nonumber    \\
		B(x_k,u_k) \triangleq \matl  G_{1,k+1}  \\
		G_{2,k+1}  \\
		\vdots   \\
		G_{n,k+1}
		\matr,\quad
		C \triangleq \matl I_p & 0 & \cdots & 0\matr,
		\label{eq:BC_BOCF}
	\end{gather}
	and $x_k\triangleq \matl x_{k(1)}^\rmT &  \cdots& x_{k(n)}^\rmT\matr^\rmT \in\BBR^{np}$, where
	\begin{equation}
		x_{k(1)} \triangleq y_k,
		\label{eq:eta1}
	\end{equation}
	and, for all $\tau\in\{ 2,\hdots,n\}$,
	\begin{equation}
		x_{k(\tau)} \triangleq \sum_{s=1}^{n-\tau+1}  \( -F_{\tau+s-1,k}\, y_{k-s}  +  G_{\tau+s-1,k}\, u_{k-s}\).
		\label{eq:eta2}
	\end{equation}
	It follows from \eqref{eq:eta1} and \eqref{eq:eta2} that the BOCF state $x_k$ is an explicit function of the $n$ most recent outputs and inputs, and thus \eqref{eq:xk1_OF}--\eqref{eq:yk_OF} acts as a deadbeat observer with zero reconstruction error.
	The reconstruction is exact because \eqref{eq:eta1} and \eqref{eq:eta2} invert the input-output model \eqref{eq:yyy} algebraically rather than asymptotically, and it therefore presumes that $y_k$ is measured exactly and that \eqref{eq:yyy} describes the plant exactly.
	Under measurement noise, or for measurement maps that admit no exact finite-order input-output realization, the deadbeat property is lost and an estimator is required; Section~\ref{sec:conclusions} identifies the state- and control-dependent moving-horizon estimator of \cite{kamaldarSCDMHE} as the natural counterpart of the present method in that setting.
	Since \eqref{eq:xk1_OF} is in the SCDC form \eqref{eq:fSCDC}, iSCD-MPC applies directly to the BOCF realization for output-feedback control of \eqref{eq:yyy}; the corresponding stability statement is given in Proposition~\ref{prop:bocf_stability}.
	
	\section{Theoretical Analysis}
	\label{sec:theory}
	
	Throughout this section, we consider the system \eqref{eq:xk1} controlled by iSCD-MPC given by \eqref{eq:fSCDC}--\eqref{eq:applied_control}.
	To streamline the exposition, the weightings are taken to be constant, that is, for all $k\in\BBN$ and $j\in\{1,\ldots,\ell-1\}$, $Q_{k,j}= Q\succeq 0$ and $R_{k,j}= R\succ 0$; the algorithm supports time- and horizon-varying weightings.
	Except in Proposition~\ref{prop:Pk_bounded}, which concerns the adaptation \eqref{eq:Pk_update}, the terminal penalty is the static matrix $P\succ 0$ constructed in Section~\ref{sec:terminal}.
	Furthermore, for the analysis, we assume that the constraints \eqref{eq:ineqcon}--\eqref{eq:boxcon} are equivalent to the decoupled constraints
	\begin{equation}
		x_{k,j}\in\SX,\ \ j\in\{1,\ldots,\ell\},\quad u_{k,j}\in\SU,\ \ j\in\{1,\ldots,\ell-1\},
		\label{eq:decoupled_con}
	\end{equation}
	where $\SX\subseteq\BBR^n$ and $\SU\subseteq\BBR^m$ are closed, $0\in{\rm int}\,\SX$, and $0\in{\rm int}\,\SU$; the algorithm supports the general polyhedral constraints \eqref{eq:ineqcon}--\eqref{eq:boxcon}.
	
	The following assumptions are used in this section.
	
	\begin{enumerate}[label=({A\arabic*}),itemindent=5pt]
		\item \label{A:lipschitz}
		$A$ and $B$ are continuously differentiable, and there exist  $L_A, L_B\in(0,\infty)$ such that, for all $x, \hat{x} \in \BBR^n$ and $u, \hat{u} \in \BBR^m$,
		\begin{align}
			\|A(x,u) - A(\hat{x},\hat{u})\| &\le L_A \(\|x - \hat{x}\| + \|u - \hat{u}\|\), \label{eq:LipA}\\
			\|B(x,u) - B(\hat{x},\hat{u})\| &\le L_B \(\|x - \hat{x}\| + \|u - \hat{u}\|\).\label{eq:LipB}
		\end{align}
		
		\item \label{A:bounded}
		There exist $\bar a, \bar b\in(0,\infty)$ such that
		\begin{equation}
			\sup_{x\in\BBR^n,\,u\in\BBR^m}\|A(x,u)\|\le \bar a,\quad \sup_{x\in\BBR^n,\,u\in\BBR^m}\|B(x,u)\|\le \bar b.\label{eq:AB_bound}
		\end{equation}
		
		\item \label{A:stabilizable}
		The pair $(A_0,B_0)\triangleq (A(0,0),B(0,0))$ is stabilizable, and the pair $(A_0,Q^{1/2})$ is detectable.
	\end{enumerate}
	
	Assumption \ref{A:lipschitz} imposes uniform Lipschitz continuity and smoothness of the SCDC matrices; \ref{A:bounded} bounds the pseudo-linear flow; and \ref{A:stabilizable} enables the construction of the terminal ingredients from the linearization at the origin.
	For notational convenience, \ref{A:lipschitz} and \ref{A:bounded} are stated on $\BBR^n\times\BBR^m$; both hold globally for the saturated-input models of Section~\ref{sec:examples}, whose coefficient maps are built from bounded functions.
	If \ref{A:lipschitz} and \ref{A:bounded} hold only on a compact set, then the results of this section hold verbatim provided that the predicted and closed-loop trajectories remain in that set.
	Note that \eqref{eq:fSCDC} and \ref{A:lipschitz} imply that $f$ is continuously differentiable with $\nabla_x f(0,0) = A_0$ and $\nabla_u f(0,0) = B_0$, so that $(A_0,B_0)$ is the Jacobian linearization of \eqref{eq:xk1} at the origin.
	
	The results so far concern a single control update.
	Establishing that the closed loop is well posed and convergent requires the terminal ingredients, to which we now turn.
	
	\subsection{Terminal Set}
	\label{sec:terminal_set}
	
	This subsection constructs the terminal set associated with the terminal control law \eqref{eq:kappaf_def} and establishes the invariance and descent properties used in the feasibility and stability analysis.
	
	\begin{proposition}\label{prop:terminal_set}\rm
		Consider the system \eqref{eq:xk1}, where for all $k\in\BBN$, $u_k=\kappa_\rmf(x_k)$ and the terminal control law $\kappa_\rmf$ is given by \eqref{eq:Kf_def} and \eqref{eq:kappaf_def}. Assume that \ref{A:lipschitz} and \ref{A:stabilizable} hold, $Q\succ 0$, $0\in{\rm int}\,\SX$, and $0\in{\rm int}\,\SU$. Define $V_\rmf\colon\BBR^n\to[0,\infty)$ by $V_\rmf(x)\triangleq x^\rmT P x$, where $P$ is given by \eqref{eq:DARE}.
		Then, there exists $\alpha\in(0,\infty)$ such that the sublevel set
		\begin{equation}
			\SX_\rmf\triangleq \{x\in\BBR^n\colon V_\rmf(x)\le \alpha\}\label{eq:Xf_def}
		\end{equation}
		satisfies the following statements:
		\begin{enumerate}[label=({\it \roman*}),itemindent=1pt]
			\item\label{ts:i} $\SX_\rmf\subseteq \SX$, and, for all $x\in\SX_\rmf$, $\kappa_\rmf(x)\in\SU$.
			\item\label{ts:ii} For all $x\in\SX_\rmf$, $f(x,\kappa_\rmf(x))\in\SX_\rmf$.
			\item\label{ts:iii} For all $x\in\SX_\rmf$,
			\begin{equation}
				V_\rmf(f(x,\kappa_\rmf(x)))-V_\rmf(x)\le - x^\rmT Q x - \kappa_\rmf(x)^\rmT R\, \kappa_\rmf(x).\label{eq:Vf_descent}
			\end{equation}
		\end{enumerate}
	\end{proposition}
	\begin{proof}
		The argument is the quasi-infinite-horizon construction of \cite{chen1998quasi}, and we indicate only the steps at which the factorization \eqref{eq:fSCDC} enters.
		Define $\phi(x)\triangleq f(x,\kappa_\rmf(x))-A_{\rm cl}x$, so that \eqref{eq:fSCDC}, \eqref{eq:kappaf_def}, and \eqref{eq:Acl_def} give
		\begin{equation}
			\phi(x)=\(A(x,\kappa_\rmf(x))-A_0\)x+\(B(x,\kappa_\rmf(x))-B_0\)K_\rmf x.\label{eq:phi_def}
		\end{equation}
		The factorization enters at exactly this point: \eqref{eq:phi_def} is an exact identity rather than a Taylor remainder, so \ref{A:lipschitz} bounds the deviation directly, giving $\|\phi(x)\|\le c_\phi\|x\|^2$, where $c_\phi\triangleq(L_A+L_B\|K_\rmf\|)(1+\|K_\rmf\|)$.
		Writing $f(x,\kappa_\rmf(x))=A_{\rm cl}x+\phi(x)$ and using \eqref{eq:Riccati_identity} yields
		\begin{align*}
			&V_\rmf(f(x,\kappa_\rmf(x)))-V_\rmf(x)\\
			&\qquad= -x^\rmT Q_\star x+2x^\rmT A_{\rm cl}^\rmT P\phi(x)+\phi(x)^\rmT P\phi(x).
		\end{align*}
		Let $r\in(0,\infty)$ satisfy $2\|A_{\rm cl}\|\,\|P\|\,c_\phi\,r+\|P\|c_\phi^2r^2\le\tfrac{1}{2}\lambda_{\min}(Q_\star)$.
		Then the quadratic bound on $\phi$ implies that the last two terms are dominated by $\tfrac{1}{2}\lambda_{\min}(Q_\star)\|x\|^2$ whenever $\|x\|\le r$, leaving $-\tfrac{1}{2}x^\rmT Q_\star x$, which, since $Q_\star=2Q+2K_\rmf^\rmT RK_\rmf$, is exactly the right-hand side of \eqref{eq:Vf_descent}.
		This is the role of the inflation in \eqref{eq:DARE} noted in Remark~\ref{rem:inflation}.
		Since $0\in{\rm int}\,\SX$ and $0\in{\rm int}\,\SU$, there exist $r_x,r_u\in(0,\infty)$ such that the balls of those radii are contained in $\SX$ and $\SU$, and setting
		\begin{equation}
			\alpha\triangleq \lambda_{\min}(P)\min\{r^2,\ r_x^2,\ r_u^2/\|K_\rmf\|^2\}\label{eq:alpha_def}
		\end{equation}
		gives $\|x\|\le\min\{r,r_x,r_u/\|K_\rmf\|\}$ for all $x\in\SX_\rmf$, which confirms \ref{ts:iii}.
		Statement \ref{ts:i} follows because $\|x\|\le r_x$ and $\|\kappa_\rmf(x)\|\le\|K_\rmf\|\,\|x\|\le r_u$ on $\SX_\rmf$, and \ref{ts:ii} follows because \eqref{eq:Vf_descent} with $Q\succ0$ and $R\succ0$ implies $V_\rmf(f(x,\kappa_\rmf(x)))\le V_\rmf(x)\le\alpha$.
	\end{proof}
	
	\subsection{Receding-Horizon Optimal Control Problem}
	\label{sec:ocp}
	
	For the accuracy, feasibility, and stability analysis, consider the receding-horizon optimal control problem with nonlinear prediction dynamics, that is, for each $k\in\BBN$,
	\begin{align}
		{\rm OCP}_k\colon\quad &\min_{u_{k,1},\ldots,u_{k,\ell-1}}\ J_k(u_{k,1},\ldots,u_{k,\ell-1})\nn\\
		&\mbox{subject to \eqref{eq:xfka}, \eqref{eq:decoupled_con}, and } x_{k,\ell}\in\SX_\rmf,
		\label{eq:OCP}
	\end{align}
	where $J_k$ is given by \eqref{Jdefn} with $P_k= P$ given by \eqref{eq:DARE}, and $\SX_\rmf$ is given by \eqref{eq:Xf_def}.
	Note that, for all $k\in\BBN$, ${\rm OCP}_k$ is the problem that the SCDC iteration of Section~\ref{sec:iterQP} approximates; Propositions~\ref{prop:contraction} and \ref{prop:kkt_gap} quantify the approximation, and Theorem~\ref{thm:relaxed_lyapunov} addresses the suboptimal case in which the applied sequence is not a minimizer of ${\rm OCP}_k$.
	Since \eqref{eq:xk1_pred} implies that, for all $k\in\BBN$, the data of ${\rm OCP}_k$ are determined by the pair $(x_k,u_k)$, define the extended state
	\begin{equation}
		z_k\triangleq \matl x_k\\ u_k\matr\in\BBR^{n+m},
		\label{eq:zk_def}
	\end{equation}
	and define the feasible set $\SF_\ell\subseteq\BBR^{n+m}$ of ${\rm OCP}_k$ by
	\begin{align}
		\SF_\ell\triangleq \Big\{ &\matl x^\rmT& u^\rmT\matr^\rmT\in\BBR^{n+m}\colon \mbox{There exists}\, u_1,\ldots,u_{\ell-1}\in\SU\nn\\
		& \mbox{ such that~}\mbox{with } x_1=f(x,u) \mbox{ and } x_{j+1}=f(x_j,u_j),\nn\\
		& x_\ell\in\SX_\rmf\mbox{ and }\mbox{for all } j\in\{1,\ldots,\ell-1\}, x_j\in\SX   \Big\}.
		\label{eq:Fl_def}
	\end{align}
	
	Both the warm start \eqref{eq:warm_start} and the feasibility and stability arguments below rest on the same construction, namely, discarding the first element of a control sequence and appending the terminal control law.
	It is convenient to name that construction once.
	For all $k\in\BBN$ and all $U=\matl u_{1}^\rmT&\cdots&u_{\ell-1}^\rmT\matr^\rmT$ that are feasible for ${\rm OCP}_k$, let $x_1,\ldots,x_\ell$ denote the predicted trajectory generated by \eqref{eq:xfka} from $x_1=x_{k,1}$ under $U$, and define the \emph{shifted candidate} $\SSS_k(U)\in\BBR^{m(\ell-1)}$ by
	\begin{equation}
		\SSS_k(U)\triangleq \matl u_{2}^\rmT&\cdots&u_{\ell-1}^\rmT&\kappa_\rmf(x_\ell)^\rmT\matr^\rmT.
		\label{eq:shift_op}
	\end{equation}
	Thus $\SSS_k(U)$ is a control sequence for ${\rm OCP}_{k+1}$, and the warm start \eqref{eq:warm_start} is precisely $U_{k+1|1}=\SSS_k(U_{k|\rho_k})$.
	
	\subsection{Contraction of the SCDC Iteration}
	\label{sec:contraction}
	
	This subsection analyzes the iteration map $\SM$ defined by \eqref{eq:Hmap_def} in the case where the constraints \eqref{eq:QP_ineq} and \eqref{eq:QP_box} are inactive, so that each iterate is given explicitly by \eqref{eq:Hmap}.
	Since $0\in{\rm int}\,\SX$ and $0\in{\rm int}\,\SU$, the constraints are inactive in a neighborhood of the origin, and thus the analysis applies to the constrained algorithm for all $\|x_k\|$ and $\|u_k\|$ in that neighborhood.
	Since the step index $k$ enters \eqref{eq:Hmap} only through the anchor $x_{k,1}$, we make that dependence explicit by writing the map as a function of an anchor $\xi\in\BBR^n$ and a control sequence $U$.
	Specifically, for all $\xi\in\BBR^n$ and $U\isdef\matl u_1^\rmT&\cdots&u_{\ell-1}^\rmT\matr^\rmT\in\BBR^{m(\ell-1)}$, let $x_1(\xi,U)\triangleq \xi$ and, for all $j\in\{1,\ldots,\ell-1\}$, $x_{j+1}(\xi,U)\triangleq f(x_j(\xi,U),u_j)$ denote the rollout \eqref{eq:xkj+1i} initialized at $\xi$ under $U$; let $\Phi(\xi,U)$, $\Gamma(\xi,U)$, and $\SH(\xi,U)$ denote \eqref{eq:Phidef}, \eqref{eq:Gammadef}, and \eqref{eq:Hessian} evaluated with $A_j = A(x_j(\xi,U),u_j)$ and $B_j = B(x_j(\xi,U),u_j)$; and write
	\begin{equation}
		\SM(\xi,U) = -\SH(\xi,U)^{-1}\Gamma(\xi,U)^\rmT \SQ\,\Phi(\xi,U)\, \xi,
		\label{eq:M_explicit}
	\end{equation}
	where $\SQ\triangleq{\rm diag}(Q,\ldots,Q,P)$ and $\SR\triangleq{\rm diag}(R,\ldots,R)$.
	For all $k\in\BBN$, the iterates \eqref{eq:Hmap_def} of iSCD-MPC are generated by $U_{k|i}=\SM(x_{k,1},U_{k|i-1})$.
	
	Let  $\delta_U\in(0,\infty)$, and define
	\begin{equation}
		\SB\triangleq\{U\in\BBR^{m(\ell-1)}\colon \|U\|\le \delta_U\}.
	\end{equation} 
	
	\begin{proposition}\label{prop:contraction}\rm
		Consider the system \eqref{eq:xk1} controlled by iSCD-MPC \eqref{eq:fSCDC}--\eqref{eq:applied_control}, where the iteration map $\SM$ is given by \eqref{eq:M_explicit}. 
		Assume that \ref{A:lipschitz} and \ref{A:bounded} hold. 	Let  $k\in\BBN$, and  assume that there exist $\gamma\in(0,1)$, $c_0\in[1,\infty)$, and $c_\SM\in(0,\infty)$ such that
		\begin{equation}
			\|x_{k,1}\|\le \min\left\{\frac{\delta_U}{c_0},\ \frac{\gamma}{c_\SM}\right\}.
			\label{eq:xk1_small}
		\end{equation}
		Then, the following statements hold:
		\begin{enumerate}[label=({\it \roman*}),itemindent=1pt]
			\item\label{ct:i} For all $U\in\SB$, $\SM(x_{k,1},U)\in\SB$. In addition, for all $U,\hat U\in\SB$,
			\begin{equation}
				\|\SM(x_{k,1},U)-\SM(x_{k,1},\hat U)\|\le \gamma \|U-\hat U\|.
				\label{eq:contraction_ineq}
			\end{equation}
			\item\label{ct:ii} $\SM$ has a unique fixed point $U_{k,\star}\in\SB$ satisfying $\SM(x_{k,1}, U_{k,\star}) = U_{k,\star}$.
			\item\label{ct:iii} For all $U_{k|1}\in\SB$ and all $i\in\BBZ^+$,
			\begin{equation}
				\|U_{k|i}-U_{k,\star}\|\le \gamma^{i-1}\|U_{k|1}-U_{k,\star}\|,
				\label{eq:geo_convergence}
			\end{equation}
			and thus, as $i\to\infty$, $U_{k|i}\to U_{k,\star}$.
		\end{enumerate}
	\end{proposition}
	\begin{proof}
		Let $k\in\BBN$, and throughout the proof, abbreviate $x_j(U)\triangleq x_j(x_{k,1},U)$, $\Phi(U)\triangleq\Phi(x_{k,1},U)$, $\Gamma(U)\triangleq\Gamma(x_{k,1},U)$, $\SH(U)\triangleq\SH(x_{k,1},U)$, and $\SM(U)\triangleq\SM(x_{k,1},U)$.
		Since $\SQ\succeq0$ and, for all $U\in\SB$, \eqref{eq:Hessian} implies $\SH(U)\succeq\lambda_{\min}(\SR)I_{m(\ell-1)}$, it follows that
		\begin{equation}
			\|\SH(U)^{-1}\|\le \lambda_{\min}(\SR)^{-1}.\label{eq:Hinv_bound}
		\end{equation}
		Furthermore, since each block of $\Phi(U)$ and $\Gamma(U)$ is a product of at most $\ell$ matrices, each of which is bounded by $\max\{\bar a,\bar b,1\}$ from \ref{A:bounded}, there exist $\bar\phi,\bar g\in(0,\infty)$, depending only on $\bar a$, $\bar b$, and $\ell$, such that, for all $U\in\SB$,
		\begin{equation}
			\|\Phi(U)\|\le\bar\phi,\qquad \|\Gamma(U)\|\le\bar g.\label{eq:PhiGam_bound}
		\end{equation}
		It follows from \eqref{eq:M_explicit}, \eqref{eq:Hinv_bound}, and \eqref{eq:PhiGam_bound} that, for all $U\in\SB$,
		\begin{equation}
			\|\SM(U)\|\le c_0\|x_{k,1}\|,
			\label{eq:M_bound}
		\end{equation}
		where $c_0\triangleq \max\{1,\,\lambda_{\min}(\SR)^{-1}\bar g\,\|\SQ\|\,\bar\phi\}\ge 1$.
		Hence, \eqref{eq:xk1_small} and \eqref{eq:M_bound} imply that, for all $U\in\SB$, $\|\SM(U)\|\le \delta_U$, which proves $\SM(\SB)\subseteq\SB$.
		
		Next, let $U,\hat U\in\SB$, and write $x_j\triangleq x_j(U)$ and $\hat x_j\triangleq x_j(\hat U)$.
		It follows from \eqref{eq:fSCDC}, \ref{A:bounded}, and the rollout recursion that, for all $j\in\{1,\ldots,\ell\}$,
		\begin{equation}
			\|x_j\|\le \bar a^{\,j-1}\|x_{k,1}\|+\bar b\sum_{s=1}^{j-1}\bar a^{\,j-1-s}\|u_s\|\le c_1\(\|x_{k,1}\|+\delta_U\),
			\label{eq:rollout_bound}
		\end{equation}
		where $c_1\triangleq \max\{1,\bar a^{\,\ell-1}\}(1+\bar b\,\ell)$.
		Since \eqref{eq:xk1_small} and $c_0\ge1$ imply $\|x_{k,1}\|\le \delta_U$, \eqref{eq:rollout_bound} implies $\|x_j\|\le 2 c_1\delta_U$.
		Next, it follows from \eqref{eq:fSCDC}, \eqref{eq:LipA}, \eqref{eq:LipB}, \eqref{eq:AB_bound}, and \eqref{eq:rollout_bound} that, for all $j\in\{1,\ldots,\ell-1\}$,
		\begin{align}
			\|x_{j+1}-\hat x_{j+1}\|&\le \|A(x_j,u_j)\|\,\|x_j-\hat x_j\|\nn\\
			&\quad+\|A(x_j,u_j)-A(\hat x_j,\hat u_j)\|\,\|\hat x_j\|\nn\\
			&\quad+\|B(x_j,u_j)\|\,\|u_j-\hat u_j\|\nn\\
			&\quad+\|B(x_j,u_j)-B(\hat x_j,\hat u_j)\|\,\|\hat u_j\|\nn\\
			&\le \bar c_2 \(\|x_j-\hat x_j\|+\|u_j-\hat u_j\|\),
			\label{eq:rollout_lip}
		\end{align}
		where $\bar c_2\triangleq \bar a+\bar b+ (L_A+L_B)(2 c_1+1)\delta_U$.
		Since $x_1=\hat x_1 = x_{k,1}$, iterating \eqref{eq:rollout_lip} yields, for all $j\in\{1,\ldots,\ell\}$,
		\begin{equation}
			\|x_j-\hat x_j\|\le c_3\|U-\hat U\|,
			\label{eq:xj_lip}
		\end{equation}
		where $c_3\triangleq \ell\,\max\{1,\bar c_2^{\,\ell}\}$.
		Since each block of $\Phi(U)$ and $\Gamma(U)$ is a product of at most $\ell$ factors, each of which is bounded by $\max\{\bar a,\bar b,1\}$ and, by \eqref{eq:LipA}, \eqref{eq:LipB}, and \eqref{eq:xj_lip}, is Lipschitz in $U$ with constant $\max\{L_A,L_B\}(c_3+1)$, it follows that there exist $c_\Phi,c_\Gamma\in(0,\infty)$, depending only on $\bar a$, $\bar b$, $L_A$, $L_B$, $\ell$, and $\delta_U$, such that, for all $U,\hat U\in\SB$,
		\begin{gather}
			\|\Phi(U)-\Phi(\hat U)\|\le c_\Phi\|U-\hat U\|,\label{eq:Phi_lip}\\
			\|\Gamma(U)-\Gamma(\hat U)\|\le c_\Gamma\|U-\hat U\|.\label{eq:Gam_lip}
		\end{gather}
		To prove \eqref{eq:contraction_ineq}, write, for brevity, $\SH\triangleq\SH(U)$, $\hat\SH\triangleq\SH(\hat U)$, $\Gamma\triangleq\Gamma(U)$, $\hat\Gamma\triangleq\Gamma(\hat U)$, $\Phi\triangleq\Phi(U)$, and $\hat\Phi\triangleq\Phi(\hat U)$.
		It follows from \eqref{eq:M_explicit} that
		\begin{align}
			\SM(U)-\SM(\hat U)
			&= -\Big[\(\SH^{-1}-\hat\SH^{-1}\)\hat\Gamma^\rmT \SQ\,\hat\Phi\nn\\
			&\quad +\SH^{-1}\(\Gamma-\hat\Gamma\)^\rmT \SQ\,\hat\Phi\nn\\
			&\quad +\SH^{-1}\Gamma^\rmT \SQ\,\(\Phi-\hat\Phi\)\Big]\, x_{k,1}.\label{eq:M_diff}
		\end{align}
		Using the identity $\SH^{-1}-\hat\SH^{-1} = \SH^{-1}(\hat\SH-\SH)\hat\SH^{-1}$ together with \eqref{eq:Hessian}, \eqref{eq:Hinv_bound}, \eqref{eq:PhiGam_bound}, and \eqref{eq:Gam_lip} yields
		\begin{equation}
			\|\SH^{-1}-\hat\SH^{-1}\|\le 2\lambda_{\min}(\SR)^{-2}\bar g\,\|\SQ\|\, c_\Gamma\|U-\hat U\|.\label{eq:Hinv_lip}
		\end{equation}
		Applying the triangle inequality to \eqref{eq:M_diff} and substituting \eqref{eq:Hinv_bound}, \eqref{eq:PhiGam_bound}, \eqref{eq:Phi_lip}, \eqref{eq:Gam_lip}, and \eqref{eq:Hinv_lip} yields
		\begin{equation}
			\|\SM(U)-\SM(\hat U)\|\le c_\SM\|x_{k,1}\|\,\|U-\hat U\|,
			\label{eq:M_lip}
		\end{equation}
		where
		$c_\SM\triangleq \lambda_{\min}(\SR)^{-1}\|\SQ\|\big(2\lambda_{\min}(\SR)^{-1}\bar g^2\|\SQ\|\,\bar\phi\, c_\Gamma+\bar\phi\, c_\Gamma+\bar g\, c_\Phi\big)$.
		Hence, \eqref{eq:xk1_small} and \eqref{eq:M_lip} imply \eqref{eq:contraction_ineq}, which proves \ref{ct:i}.
		Statements \ref{ct:ii} and \ref{ct:iii} follow from \ref{ct:i} and the Banach fixed-point theorem \cite[Appendix~B.1]{khalil2002nonlinear} applied to $\SM(x_{k,1},\cdot)$ on the closed ball $\SB$.
	\end{proof}
	
	Proposition~\ref{prop:contraction} shows that the SCDC iteration is locally a contraction whose rate is proportional to $\|x_{k,1}\|$; hence, the contraction rate decreases as the origin is approached.
	The condition \eqref{eq:xk1_small} is stated in terms of the anchor $x_{k,1}$ because that is the quantity the iteration map \eqref{eq:M_explicit} actually depends on.
	For design purposes, however, the quantity available at step $k$ is the extended state $z_k$ given by \eqref{eq:zk_def}, and the following consequence of \ref{A:bounded} restates the condition in those terms.
	
	\begin{corollary}\label{cor:contraction_zk}\rm
		Assume that the hypotheses of Proposition~\ref{prop:contraction} hold, let $k\in\BBN$, and define $L_\rmf\triangleq\sqrt2\max\{\bar a,\bar b\}$.
		If
		\begin{equation}
			\|z_k\|\le \frac{1}{L_\rmf}\min\left\{\frac{\delta_U}{c_0},\ \frac{\gamma}{c_\SM}\right\},
			\label{eq:zk_small}
		\end{equation}
		then \eqref{eq:xk1_small} holds, and thus statements \ref{ct:i}--\ref{ct:iii} of Proposition~\ref{prop:contraction} hold.
	\end{corollary}
	\begin{proof}
		It follows from \eqref{eq:xk1_pred}, \eqref{eq:fSCDC}, and \ref{A:bounded} that
		\begin{align*}
			\|x_{k,1}\| &= \|A(x_k,u_k)x_k+B(x_k,u_k)u_k\|\\
			&\le \bar a\|x_k\|+\bar b\|u_k\|\le L_\rmf\|z_k\|,
		\end{align*}
		where the last inequality follows from the Cauchy-Schwarz inequality applied to $\matl \bar a&\bar b\matr$ and $\matl \|x_k\|&\|u_k\|\matr^\rmT$.
		Hence \eqref{eq:zk_small} implies \eqref{eq:xk1_small}.
	\end{proof}

	The next result translates \eqref{eq:geo_convergence} into an explicit bound on the number of iterations $\rho_k$ defined by \eqref{eq:rhokk}, which is the quantity that determines the per-step cost in Section~\ref{sec:complexity}.
	
	\begin{proposition}\label{prop:iteration_bound}\rm
		Consider the system \eqref{eq:xk1} controlled by iSCD-MPC \eqref{eq:fSCDC}--\eqref{eq:applied_control} with stopping criterion \eqref{eq:rhokk}. Assume  that all the hypotheses of Proposition~\ref{prop:contraction} hold. Let $k\in\BBN$, and assume that $U_{k|1}\in\SB$ and $U_{k|2}\ne U_{k|1}$.
		Then, the termination iteration $\rho_k$ given by \eqref{eq:rhokk} satisfies
		\begin{equation}
			\rho_k \le \min\left\{\rho,\ 2+\max\left\{0,\ \left\lfloor \tfrac{\ln(\varepsilon/\|U_{k|2}-U_{k|1}\|)}{\ln\gamma}\right\rfloor + 1\right\}\right\}.
			\label{eq:iter_bound}
		\end{equation}
	\end{proposition}
	\begin{proof}
		It follows from \eqref{eq:Hmap_def} and \eqref{eq:contraction_ineq} that, for all $i\ge 2$,
		\begin{equation*}
			\|U_{k|i}-U_{k|i-1}\|\le \gamma^{i-2}\|U_{k|2}-U_{k|1}\|.
		\end{equation*}
		Hence, $\|U_{k|i}-U_{k|i-1}\|<\varepsilon$ holds for all $i\ge 2$ satisfying $\gamma^{i-2}\|U_{k|2}-U_{k|1}\|<\varepsilon$, that is, for all $i\ge 2$ such that $i > 2+\ln\big(\varepsilon/\|U_{k|2}-U_{k|1}\|\big)/\ln\gamma$, where the direction of the inequality reverses because $\ln\gamma<0$.
		Since the smallest integer strictly greater than a real number $x$ is $\lfloor x\rfloor+1$, combining this threshold with the inner minimum in \eqref{eq:rhokk} and the outer minimum with $\rho$ yields \eqref{eq:iter_bound}.
	\end{proof}
	
	\begin{remark}\rm
		\label{rem:contraction_practice}
		Proposition~\ref{prop:contraction} and Corollary~\ref{cor:contraction_zk} are usable for tuning even though $c_\SM$ is not available in closed form, because \eqref{eq:M_lip} shows that the contraction rate observed at step $k$ is $\gamma\approx c_\SM\|x_{k,1}\|$.
		A single cold-start run therefore identifies $c_\SM$: fitting the geometric decay of $\|U_{k|i}-U_{k|i-1}\|$ at $k=0$ yields $\hat\gamma_0$, and $\hat c_\SM\triangleq\hat\gamma_0/\|x_{0,1}\|$ estimates the constant.
		On the quadrotor of Example~\ref{ex:quadrotor}, $\hat\gamma_0=0.354$ with $\|x_{0,1}\|=1.18$, so that $\hat c_\SM\approx0.30$.
		With $\hat c_\SM$ in hand, \eqref{eq:zk_small} delimits the extended states from which a prescribed rate $\gamma$ is expected, and \eqref{eq:iter_bound} converts that rate into an iteration budget, which is how $\rho$ and $\varepsilon$ may be chosen for a given sampling period.
		The two constants trade off against each other: enlarging $\delta_U$ widens the ball $\SB$ on which the analysis is valid, but increases $\bar c_2$ and hence $c_3$ and $c_\SM$, so that the right-hand side of \eqref{eq:zk_small} need not increase.
	\end{remark}
	
	\begin{remark}\rm
		\label{rem:warmstart}
		Numerical simulations suggest that convergence is achieved for a wider range of initializations than the ball $\SB$ centered at the origin, and, away from the origin, the observed contraction rate is often close to, but below, one, as illustrated in Section~\ref{sec:examples}.
		Along converging closed-loop trajectories, the warm start \eqref{eq:warm_start} reduces $\|U_{k|2}-U_{k|1}\|$ relative to the cold start \eqref{eq:u0_init}, so that the bound \eqref{eq:iter_bound} on $\rho_k$ decreases after the initial steps; this behavior parallels the real-time iteration scheme \cite{diehl2002real} and warm-started iLQR \cite{TodorovLi2005}.
	\end{remark}
	
	The next result shows that the cost of the nonlinear rollout converges geometrically along the iterates.
	For all $U\isdef\matl u_1^\rmT&\cdots&u_{\ell-1}^\rmT\matr^\rmT\in\BBR^{m(\ell-1)}$, define the rollout cost $\bar J_k\colon\BBR^{m(\ell-1)}\to[0,\infty)$ by
	\begin{align}
		\bar J_k(U)&\triangleq \tfrac{1}{2}x_\ell(x_{k,1},U)^\rmT P\, x_\ell(x_{k,1},U)\nn\\
		&\quad+\tfrac{1}{2}\sum_{j=1}^{\ell-1}\(x_j(x_{k,1},U)^\rmT Q\, x_j(x_{k,1},U)+u_j^\rmT R\, u_j\),
		\label{eq:Jbar_def}
	\end{align}
	which is the cost \eqref{Jdefn} evaluated along the nonlinear rollout \eqref{eq:xkj+1i} under $U$.
	
	\begin{proposition}\label{prop:cost_decrease}\rm
		Consider the system \eqref{eq:xk1} controlled by iSCD-MPC \eqref{eq:fSCDC}--\eqref{eq:applied_control}. Assume that all the hypotheses of Proposition~\ref{prop:contraction} hold. Let $k\in\BBN$, and assume that $U_{k|1}\in\SB$.
		Then, there exists $L_J\in(0,\infty)$ such that, for all $i\in\BBZ^+$,
		\begin{equation}
			\big|\bar J_k(U_{k|i})-\bar J_k(U_{k,\star})\big|\le L_J\,\gamma^{i-1}\|U_{k|1}-U_{k,\star}\|.
			\label{eq:cost_mono}
		\end{equation}
		In addition, as $i\to\infty$, $\bar J_k(U_{k|i})\to \bar J_k(U_{k,\star})$.
	\end{proposition}
	\begin{proof}
		Let $U,\hat U\in\SB$, and abbreviate $x_j(U)\triangleq x_j(x_{k,1},U)$ as in the proof of Proposition~\ref{prop:contraction}.
		It follows from \eqref{eq:rollout_bound} and \eqref{eq:xj_lip} that, for all $j\in\{1,\ldots,\ell\}$, $\|x_j(U)\|\le 2c_1\delta_U$ and $\|x_j(U)-x_j(\hat U)\|\le c_3\|U-\hat U\|$.
		Applying the identity $a^\rmT W a - \hat a^\rmT W\hat a = (a+\hat a)^\rmT W(a-\hat a)$, which holds for all symmetric $W$, to each term of \eqref{eq:Jbar_def} yields
		\begin{align}
			|\bar J_k(U)-\bar J_k(\hat U)|&\le  L_J\|U-\hat U\|,\label{eq:Jk_lip}
		\end{align}
		where 
		\begin{equation}
			L_J\triangleq 2c_1c_3\delta_U\big(\|P\|+(\ell-1)\|Q\|\big)+\delta_U\|R\|.
		\end{equation}
		Setting $\hat U = U_{k,\star}$ and $U = U_{k|i}$ in \eqref{eq:Jk_lip} and substituting \eqref{eq:geo_convergence} yields \eqref{eq:cost_mono}. Finally, part \ref{ct:iii} of Proposition~\ref{prop:contraction} confirms that as $i\to\infty$, $\bar J_k(U_{k|i})\to \bar J_k(U_{k,\star})$ as $i\to\infty$.
	\end{proof}
	
	Propositions~\ref{prop:contraction}--\ref{prop:cost_decrease} describe the iteration as a numerical process: it converges, at a known rate, to a known limit.
	They say nothing about what that limit is, and the fixed point is not in general a stationary point of ${\rm OCP}_k$.
	The next subsection quantifies the discrepancy.
	
	\subsection{First-Order Accuracy of the Fixed Point}
	\label{sec:kkt}
	
	The fixed point $U_{k,\star}$ of the SCDC iteration is, in general, not a stationary point of the underlying nonlinear program, because the frozen-coefficient QP ignores the dependence of the SCDC matrices on the decision variables.
	This subsection quantifies the resulting optimality gap.
	The first result bounds the difference between the true Jacobian of $f$ and the pseudo-Jacobian $\matl A & B\matr$ furnished by the SCDC factorization.
	
	\begin{lemma}\label{lem:jacobian_gap}\rm
		Consider the SCDC factorization \eqref{eq:fSCDC}, and suppose that \ref{A:lipschitz} holds.
		Then, $f$ is continuously differentiable, and, for all $x\in\BBR^n$ and $u\in\BBR^m$, the Jacobian $J_f(x,u)\triangleq\matl \nabla_x f(x,u)& \nabla_u f(x,u)\matr$ satisfies
		\begin{equation}
			\big\|J_f(x,u)-\matl A(x,u)& B(x,u)\matr\big\|\le \beta\(\|x\|+\|u\|\),
			\label{eq:Jac_gap}
		\end{equation}
		where $\beta\triangleq \sqrt{2}\max\{L_A,L_B\}$.
	\end{lemma}
	\begin{proof}
		Continuous differentiability of $f$ follows from \eqref{eq:fSCDC} and \ref{A:lipschitz}.
		Let $h\in\BBR^n$.
		It follows from \eqref{eq:fSCDC} and the product rule that
		\begin{align}
			\nabla_x f(x,u)\,h &= A(x,u)h + \big(D_x A(x,u)[h]\big)x\nn\\
			&\quad+\big(D_x B(x,u)[h]\big)u,\label{eq:dfx}
		\end{align}
		where $D_x A(x,u)[h]\in\BBR^{n\times n}$ and $D_x B(x,u)[h]\in\BBR^{n\times m}$ denote the directional derivatives of $A$ and $B$ at $(x,u)$ in the direction $(h,0)$.
		It follows from \eqref{eq:LipA} that $\|D_x A(x,u)[h]\|\le L_A\|h\|$, and, similarly, \eqref{eq:LipB} implies $\|D_x B(x,u)[h]\|\le L_B\|h\|$.
		Hence, \eqref{eq:dfx} implies
		\begin{equation}
			\|\nabla_x f(x,u)-A(x,u)\|\le L_A\|x\|+L_B\|u\|.\label{eq:dfx_bound}
		\end{equation}
		The same argument applied to directions of the form $(0,h_u)$, where $h_u\in\BBR^m$, yields
		\begin{equation}
			\|\nabla_u f(x,u)-B(x,u)\|\le L_A\|x\|+L_B\|u\|.\label{eq:dfu_bound}
		\end{equation}
		Combining \eqref{eq:dfx_bound} and \eqref{eq:dfu_bound} with the block-matrix norm inequality $\|\matl M_1& M_2\matr\|\le \sqrt{\|M_1\|^2+\|M_2\|^2}$ and the bound $L_A\|x\|+L_B\|u\|\le \max\{L_A,L_B\}(\|x\|+\|u\|)$ yields \eqref{eq:Jac_gap}.
	\end{proof}
	
	Consider the nonlinear program underlying ${\rm OCP}_k$ defined in Section~\ref{sec:ocp}, namely, minimization of \eqref{Jdefn} subject to the nonlinear prediction dynamics \eqref{eq:xfka}, and define the Lagrangian $\SL\colon\BBR^{n\ell}\times\BBR^{m(\ell-1)}\times\BBR^{n(\ell-1)}\to\BBR$ by
	\begin{align}
		\SL(X,U,\Lambda)&\triangleq \tfrac{1}{2}x_\ell^\rmT P x_\ell+\tfrac{1}{2}\sum_{j=1}^{\ell-1}\(x_j^\rmT Q x_j+u_j^\rmT R u_j\)\nn\\
		&\quad +\sum_{j=1}^{\ell-1}\lambda_{j+1}^\rmT\(f(x_j,u_j)-x_{j+1}\),
		\label{eq:Lagrangian}
	\end{align}
	where $X\triangleq\matl x_1^\rmT&\cdots&x_\ell^\rmT\matr^\rmT$, $U\triangleq\matl u_1^\rmT&\cdots&u_{\ell-1}^\rmT\matr^\rmT$, $\Lambda\triangleq\matl \lambda_2^\rmT&\cdots&\lambda_\ell^\rmT\matr^\rmT$, and $x_1 = x_{k,1}$ is fixed.
	
	\begin{proposition}\label{prop:kkt_gap}\rm
		Consider the system \eqref{eq:xk1} controlled by iSCD-MPC \eqref{eq:fSCDC}--\eqref{eq:applied_control}. Assume that the hypotheses of Proposition~\ref{prop:contraction} hold. Let $k\in\BBN$, and let $U_{k,\star}=\matl u_{1,\star}^\rmT&\cdots&u_{\ell-1,\star}^\rmT\matr^\rmT$ be the fixed point given by Proposition~\ref{prop:contraction}.
		Furthermore, let $x_{1,\star},\ldots,x_{\ell,\star}$ denote the rollout \eqref{eq:xkj+1i} under $U_{k,\star}$, define $A_{j,\star}\triangleq A(x_{j,\star},u_{j,\star})$ and $B_{j,\star}\triangleq B(x_{j,\star},u_{j,\star})$, and define the costates $\lambda_{\ell,\star}\triangleq P x_{\ell,\star}$ and, for all $j\in\{\ell-1,\ldots,2\}$,
		\begin{equation}
			\lambda_{j,\star}\triangleq Q x_{j,\star}+A_{j,\star}^\rmT\lambda_{j+1,\star}.
			\label{eq:costate}
		\end{equation}
		Then,
		\begin{align}
			&\big\|\nabla_{(X,U)}\SL(X_{k,\star},U_{k,\star},\Lambda_{k,\star})\big\|\nn\\
			&\qquad\le \beta\sum_{j=1}^{\ell-1}\(\|x_{j,\star}\|+\|u_{j,\star}\|\)\|\lambda_{j+1,\star}\|,
			\label{eq:KKTgap}
		\end{align}
		where $X_{k,\star}\triangleq\matl x_{2,\star}^\rmT&\cdots&x_{\ell,\star}^\rmT\matr^\rmT$ and $\Lambda_{k,\star}\triangleq\matl \lambda_{2,\star}^\rmT&\cdots&\lambda_{\ell,\star}^\rmT\matr^\rmT$ stack the free predicted states and costates, respectively, and $\beta$ is given by Lemma~\ref{lem:jacobian_gap}.
	\end{proposition}
	\begin{proof}
		At the fixed point, the QP \eqref{eq:QP_cost}--\eqref{eq:QP_dyn} with coefficients frozen at $\{(A_{j,\star},B_{j,\star})\}_{j=1}^{\ell-1}$ is solved by $U_{k,\star}$, and, since \eqref{eq:fSCDC} implies $A_{j,\star}x_{j,\star}+B_{j,\star}u_{j,\star}=f(x_{j,\star},u_{j,\star})$, the QP state trajectory coincides with the rollout $x_{1,\star},\ldots,x_{\ell,\star}$.
		Define the frozen Lagrangian $\SL_{\rm f}$ as \eqref{eq:Lagrangian} with $f(x_j,u_j)$ replaced by $A_{j,\star}x_j+B_{j,\star}u_j$.
		Since the initial state $x_1=x_{k,1}$ is fixed, the first-order optimality conditions for the free decision variables of the frozen QP are
		\begin{gather*}
			\nabla_{x_j}\SL_{\rm f} = Q x_{j,\star}+A_{j,\star}^\rmT\lambda_{j+1,\star}-\lambda_{j,\star}=0,\ j\in\{2,\ldots,\ell{-}1\},\\
			\nabla_{x_\ell}\SL_{\rm f} = P x_{\ell,\star}-\lambda_{\ell,\star}=0,\\
			\nabla_{u_j}\SL_{\rm f} = R u_{j,\star}+B_{j,\star}^\rmT\lambda_{j+1,\star}=0,\ j\in\{1,\ldots,\ell{-}1\},
		\end{gather*}
		which hold with the costates \eqref{eq:costate} by the optimality of $U_{k,\star}$ for the frozen QP; hence, $\nabla_{(X,U)}\SL_{\rm f}(X_{k,\star},U_{k,\star},\Lambda_{k,\star})=0$.
		Next, it follows from \eqref{eq:Lagrangian} that $\SL$ and $\SL_{\rm f}$ differ only in the dynamic-constraint terms.
		For $j\in\{2,\ldots,\ell-1\}$, where both $x_j$ and $u_j$ are free, the gradient difference is
		\begin{align*}
			&\nabla_{(x_j,u_j)}\big(\SL-\SL_{\rm f}\big)(X_{k,\star},U_{k,\star},\Lambda_{k,\star})\nn\\
			&\qquad= \Big(J_f(x_{j,\star},u_{j,\star})-\matl A_{j,\star}& B_{j,\star}\matr\Big)^\rmT \lambda_{j+1,\star},
		\end{align*}
		whose norm is bounded by $\beta(\|x_{j,\star}\|+\|u_{j,\star}\|)\|\lambda_{j+1,\star}\|$ by Lemma~\ref{lem:jacobian_gap}.
		At $j=1$, since $x_1$ is fixed, only $u_1$ is differentiated, yielding $\nabla_{u_1}(\SL-\SL_{\rm f})=(\nabla_u f(x_{1,\star},u_{1,\star})-B_{1,\star})^\rmT \lambda_{2,\star}$, whose norm is likewise bounded by $\beta(\|x_{1,\star}\|+\|u_{1,\star}\|)\|\lambda_{2,\star}\|$ by Lemma~\ref{lem:jacobian_gap}.
		Finally, at $j=\ell$, $\nabla_{x_\ell}(\SL-\SL_{\rm f})=0$ because the dynamic constraints do not evaluate $f$ at step $\ell$.
		Summing these bounds over $j\in\{1,\ldots,\ell-1\}$ and applying the triangle inequality yields \eqref{eq:KKTgap}.
	\end{proof}
	
	\begin{remark}\rm
		\label{rem:kkt_scaling}
		Since \eqref{eq:rollout_bound} and \eqref{eq:costate} imply that, on $\SB$, $\|x_{j,\star}\|$, $\|u_{j,\star}\|$, and $\|\lambda_{j,\star}\|$ are each bounded by a constant times $\|x_{k,1}\|$, the bound \eqref{eq:KKTgap} is $O(\|x_{k,1}\|^2)$; hence, the first-order optimality residual of the SCDC fixed point vanishes quadratically at the origin, where local stabilization demands the highest accuracy.
		Away from the origin, \eqref{eq:KKTgap} provides a computable certificate of first-order accuracy in terms of the Lipschitz constants of the chosen factorization $(A,B)$; among the nonunique factorizations \eqref{eq:fSCDC}, those with smaller $L_A$ and $L_B$ yield smaller optimality gaps.
	\end{remark}
	
	\begin{remark}\rm
		\label{rem:factorization_invariance}
		Since the factorization \eqref{eq:fSCDC} is not unique when $n>1$, it is useful to record which objects in the analysis depend on that choice and which do not.
		Let $(A,B)$ and $(\hat A,\hat B)$ be continuously differentiable factorizations of the same $f$.
		Evaluating \eqref{eq:dfx} at $(x,u)=(0,0)$ gives $\nabla_x f(0,0)=A(0,0)$ and $\nabla_u f(0,0)=B(0,0)$, and likewise for $(\hat A,\hat B)$, so that $A(0,0)=\hat A(0,0)$ and $B(0,0)=\hat B(0,0)$.
		The pair $(A_0,B_0)$ is therefore the Jacobian linearization of \eqref{eq:xk1} at the origin and is common to all such factorizations, and hence so are the solution $P$ of \eqref{eq:DARE}, the gain $K_\rmf$ given by \eqref{eq:Kf_def}, the matrix $A_{\rm cl}$ given by \eqref{eq:Acl_def}, and the function $V_\rmf$.
		The quantities that do depend on the factorization are those involving the Lipschitz constants $L_A$ and $L_B$ of \ref{A:lipschitz}, namely the constant $c_\SM$ of Proposition~\ref{prop:contraction}, the bound $\beta$ of Lemma~\ref{lem:jacobian_gap} and hence the residual \eqref{eq:KKTgap}, and the level $\alpha$ in \eqref{eq:alpha_def}, which decreases as $L_A$ and $L_B$ increase through the constant $c_\phi$ in \eqref{eq:phi_def}.
		The terminal control law and the terminal penalty are thus fixed by the plant, whereas the size of the terminal set and the sharpness of the contraction and optimality-gap estimates are affected by the choice of $(A,B)$.
	\end{remark}
	
	\subsection{Boundedness of the Online Terminal Penalty}
	\label{sec:Pk_bound}
	
	The next result shows that the adapted terminal penalty \eqref{eq:Pk_update} is uniformly bounded above and below along arbitrary trajectories.
	The upper bound uses the following uniform stabilizability assumption on the sequence \eqref{eq:barAB}.
	
	\begin{enumerate}[label=({A\arabic*}),itemindent=5pt]
		\setcounter{enumi}{3}
		\item \label{A:unif_stab}
		There exist $\bar\kappa\in(0,\infty)$, $c_\rme\in[1,\infty)$, $\lambda_\rme\in(0,1)$, and a sequence $\{K_k\}_{k=0}^\infty\subset\BBR^{m\times n}$ such that, for all $k\in\BBN$, $\|K_k\|\le\bar\kappa$ and, for all $k> j\ge0$,
		\begin{equation}
			\big\|(\bar A_{k-1}+\bar B_{k-1}K_{k-1})\cdots(\bar A_{j}+\bar B_{j}K_{j})\big\|\le c_\rme\,\lambda_\rme^{\,k-j}.
			\label{eq:unif_stab}
		\end{equation}
	\end{enumerate}
	
	Assumption \ref{A:unif_stab} states that the time-varying pair $(\bar A_k,\bar B_k)$ evaluated along the closed-loop trajectory is uniformly stabilizable \cite{anderson1981detectability}; it holds, in particular, whenever the trajectory remains in a compact set on which a continuous stabilizing gain schedule exists.
	
	\begin{proposition}\label{prop:Pk_bounded}\rm
		Consider the system \eqref{eq:xk1} controlled by iSCD-MPC \eqref{eq:fSCDC}--\eqref{eq:applied_control} with terminal-penalty adaptation \eqref{eq:Pk_update}, where $P_0\succ0$, $Q\succ 0$, and $R\succ0$, and suppose that \ref{A:bounded} and \ref{A:unif_stab} hold.
		Then, for all $k\in\BBZ^+$,
		\begin{equation}
			\underline p\, I_n\preceq P_k\preceq \bar p\, I_n,
			\label{eq:Pk_bounds}
		\end{equation}
		where $\underline p\triangleq \lambda_{\min}(Q)$ and $\bar p\triangleq c_\rme^2\lambda_\rme^{2}\|P_0\|+c_\rme^2\dfrac{\|Q\|+\bar\kappa^2\|R\|}{1-\lambda_\rme^2}$.
	\end{proposition}
	\begin{proof}
		To prove the lower bound, we use induction.
		The base case holds because $P_0\succ0$.
		Assume that $P_k\succ0$.
		Applying the matrix inversion lemma \cite[Cor.~3.9.8]{bernstein2018scalar} to \eqref{eq:Pk_update} yields
		\begin{equation*}
			P_{k+1} = Q + \bar A_k^\rmT \(P_k^{-1}+\bar B_k R^{-1}\bar B_k^\rmT\)^{-1}\bar A_k.
		\end{equation*}
		Since, in addition, $Q\succ0$ and $\(P_k^{-1}+\bar B_k R^{-1}\bar B_k^\rmT\)^{-1}\succ 0$, it follows that
		\begin{equation*}
			P_{k+1}\succeq Q\succeq \lambda_{\min}(Q)I_n\succ0,
		\end{equation*}
		which completes the induction and confirms the lower bound in \eqref{eq:Pk_bounds}.
		
		To prove the upper bound, note that completing the square implies that, for all $K\in\BBR^{m\times n}$ and $S\succeq0$,
		\begin{align}
			&Q+K^\rmT R K+(\bar A_k+\bar B_kK)^\rmT S(\bar A_k+\bar B_kK)\nn\\
			&\quad= Q+\bar A_k^\rmT S\bar A_k -\bar A_k^\rmT S\bar B_k(R+\bar B_k^\rmT S\bar B_k)^{-1}\bar B_k^\rmT S\bar A_k\nn\\
			&\qquad +(K-K_{S})^\rmT (R+\bar B_k^\rmT S\bar B_k)(K-K_{S}),
			\label{eq:completion}
		\end{align}
		where $K_{S}\triangleq -(R+\bar B_k^\rmT S\bar B_k)^{-1}\bar B_k^\rmT S\bar A_k$.
		Setting $S=P_k$ and $K=K_k$ in \eqref{eq:completion} and using \eqref{eq:Pk_update} and $(K_k-K_{P_k})^\rmT (R+\bar B_k^\rmT P_k\bar B_k)(K_k-K_{P_k})\succeq0$ yields
		\begin{equation}
			P_{k+1}\preceq Q+K_k^\rmT R K_k+\bar E_k^\rmT P_k\bar E_k,
			\label{eq:Pk_ub_recursion}
		\end{equation}
		where $\bar E_k\triangleq \bar A_k+\bar B_kK_k$.
		Iterating \eqref{eq:Pk_ub_recursion} from $0$ to $k$ yields
		\begin{equation}
			P_{k}\preceq \Psi_{k,0}^\rmT P_0 \Psi_{k,0}+\sum_{j=0}^{k-1}\Psi_{k,j+1}^\rmT\(Q+K_j^\rmT R K_j\)\Psi_{k,j+1},\label{eq:Pk_ub_sum}
		\end{equation}
		where, for all $k>j$, $\Psi_{k,j}\triangleq \bar E_{k-1}\cdots\bar E_j$ and $\Psi_{j,j}\triangleq I_n$.
		Substituting \eqref{eq:unif_stab} and the bound $\|K_j\|\le\bar\kappa$ into \eqref{eq:Pk_ub_sum} yields, for all $k\in\BBZ^+$,
		\begin{equation*}
			\|P_k\|\le c_\rme^2\lambda_\rme^{2k}\|P_0\|+c_\rme^2\(\|Q\|+\bar\kappa^2\|R\|\)\sum_{j=0}^{k-1}\lambda_\rme^{2(k-j-1)}\le \bar p,
		\end{equation*}
		which confirms the upper bound in \eqref{eq:Pk_bounds}.
	\end{proof}
	
	\subsection{Recursive Feasibility}
	\label{sec:feasibility}
	
	\begin{theorem}\label{thm:feasibility}\rm
		Consider the system \eqref{eq:xk1} controlled by receding-horizon control \eqref{eq:uk1_applied}, where, at each step $k$, the applied sequence is a minimizer of ${\rm OCP}_k$ given by \eqref{eq:OCP}. Assume that the hypotheses of Proposition~\ref{prop:terminal_set} hold, and let $z_0\in\SF_\ell$. Then, for all $k\in\BBN$, $z_k\in\SF_\ell$.
	\end{theorem}
	\begin{proof}
		We proceed by induction on $k$.
		The base case $k=0$ holds by hypothesis. Let $k\ge1$, assume that $z_k\in\SF_\ell$, and let $u_{k,1,\star},\ldots,u_{k,\ell-1,\star}$ denote a minimizer of ${\rm OCP}_k$ with predicted trajectory $x_{k,1,\star},\ldots,x_{k,\ell,\star}$ satisfying \eqref{eq:xfka}, \eqref{eq:decoupled_con}, and $x_{k,\ell,\star}\in\SX_\rmf$, where, from \eqref{eq:xk1_pred}, $x_{k,1,\star} = x_{k,1} = f(x_k,u_k)$.
		It follows from \eqref{eq:xk1}, \eqref{eq:xk1_pred}, and \eqref{eq:uk1_applied} that
		\begin{equation}
			x_{k+1} = f(x_k,u_k)=x_{k,1},\qquad u_{k+1}=u_{k,1,\star},\label{eq:xk1_eq_x1star}
		\end{equation}
		and thus \eqref{eq:xfka} implies
		\begin{equation}
			x_{k+1,1}=f(x_{k+1},u_{k+1})=f(x_{k,1},u_{k,1,\star})=x_{k,2,\star}.\label{eq:shift_anchor}
		\end{equation}
		Let $U_{k,\star}\triangleq\matl u_{k,1,\star}^\rmT&\cdots&u_{k,\ell-1,\star}^\rmT\matr^\rmT$, and consider the shifted candidate $\SSS_k(U_{k,\star})=\matl u_{k+1,1,\rmc}^{\rmT}&\cdots&u_{k+1,\ell-1,\rmc}^{\rmT}\matr^\rmT$ given by \eqref{eq:shift_op}, whose elements are
		\begin{equation}
			u_{k+1,j,\rmc} =
			\begin{cases}
				u_{k,j+1,\star}, & j\in\{1,\ldots,\ell-2\},\\
				\kappa_\rmf(x_{k,\ell,\star}), & j=\ell-1.
			\end{cases}
			\label{eq:candidate_U}
		\end{equation}
		It follows from \eqref{eq:xfka}, \eqref{eq:shift_anchor}, and \eqref{eq:candidate_U} that the candidate predicted states satisfy
		\begin{equation}
			x_{k+1,j,\rmc}=x_{k,j+1,\star},\quad j\in\{1,\ldots,\ell-1\},\label{eq:tilde_x_match}
		\end{equation}
		which, by the feasibility hypothesis at step $k$, satisfy \eqref{eq:decoupled_con}.
		At the terminal step, \eqref{eq:xfka}, \eqref{eq:candidate_U}, \eqref{eq:tilde_x_match}, and \ref{ts:ii} of Proposition~\ref{prop:terminal_set} imply
		\begin{equation}
			x_{k+1,\ell,\rmc} = f(x_{k,\ell,\star},\kappa_\rmf(x_{k,\ell,\star}))\in\SX_\rmf.\label{eq:tilde_x_ell_in_Xf}
		\end{equation}
		Since, in addition, \ref{ts:i} of Proposition~\ref{prop:terminal_set} implies $\kappa_\rmf(x_{k,\ell,\star})\in\SU$ and $x_{k+1,\ell,\rmc}\in\SX_\rmf\subseteq\SX$, it follows from \eqref{eq:decoupled_con} and \eqref{eq:tilde_x_match}--\eqref{eq:tilde_x_ell_in_Xf} that the candidate \eqref{eq:candidate_U} is feasible for ${\rm OCP}_{k+1}$, and thus $z_{k+1}\in\SF_\ell$.
	\end{proof}
	
	Recursive feasibility guarantees that the optimization can be solved at every step, but not that the closed loop approaches the origin.
	The value function supplies the missing Lyapunov argument.
	
	\subsection{Asymptotic Stability under State Feedback}
	\label{sec:stability}
	
	Under the one-step-delay convention of Remark~\ref{rem:delay}, the closed-loop dynamics evolve on the extended state \eqref{eq:zk_def} according to
	\begin{equation}
		z_{k+1}=\matl f(x_k,u_k)\\ \kappa_{\rm mpc}(z_k)\matr,
		\label{eq:closed_loop}
	\end{equation}
	where $\kappa_{\rm mpc}(z_k)\triangleq u_{k,1,\star}$ denotes the first element of the applied sequence.
	Note that \eqref{eq:fSCDC} implies $f(0,0)=0$, and hence $z=0$ is an equilibrium of \eqref{eq:closed_loop} provided that $\kappa_{\rm mpc}(0)=0$, which holds since $U=0$ is the unique minimizer of ${\rm OCP}_k$ at $z_k=0$.
	
	\begin{theorem}\label{thm:stability}\rm
		Consider the system \eqref{eq:xk1} controlled by receding-horizon control \eqref{eq:uk1_applied}, where, at each step $k\in\BBN$, the applied sequence is a minimizer of ${\rm OCP}_k$ given by \eqref{eq:OCP}. Assume that \ref{A:lipschitz}--\ref{A:stabilizable} and the hypotheses of Proposition~\ref{prop:terminal_set} hold.
		Then, the equilibrium $z=0$ of the closed-loop dynamics \eqref{eq:closed_loop} is asymptotically stable, and, for all $z_0\in\SF_\ell$, $\lim_{k\to\infty}z_k= 0$.
	\end{theorem}
	\begin{proof}
		For all $k\in\BBN$ and $z_k\in\SF_\ell$, define the value function $V_\star\colon\SF_\ell\to[0,\infty)$ by
		\begin{equation*}
			V_\star(z_k)\triangleq J_k(u_{k,1,\star},\ldots,u_{k,\ell-1,\star}),
		\end{equation*}
		where $u_{k,1,\star},\ldots,u_{k,\ell-1,\star}$ denote a minimizer of ${\rm OCP}_k$.
		
		For all $k\in\BBN$ such that $z_k\in\SF_\ell$, consider the candidate sequence \eqref{eq:candidate_U} for ${\rm OCP}_{k+1}$, which is feasible by the proof of Theorem~\ref{thm:feasibility}.
		It follows from \eqref{Jdefn}, \eqref{eq:xk1_eq_x1star}, \eqref{eq:candidate_U}, and \eqref{eq:tilde_x_match} that
		\begin{align}
			&J_{k+1}(u_{k+1,1,\rmc},\ldots,u_{k+1,\ell-1,\rmc}) - V_\star(z_k)\nn\\
			&\quad= -\tfrac{1}{2}x_{k,1}^\rmT Q\, x_{k,1} - \tfrac{1}{2}u_{k,1,\star}^\rmT R\, u_{k,1,\star} + \Delta_\rmf,
			\label{eq:costdiff}
		\end{align}
		where
		\begin{align}
			\Delta_\rmf &\triangleq \tfrac{1}{2}\Big(V_\rmf\big(f(x_{k,\ell,\star},\kappa_\rmf(x_{k,\ell,\star}))\big) - V_\rmf(x_{k,\ell,\star})\nn\\
			&\qquad + x_{k,\ell,\star}^\rmT Q\, x_{k,\ell,\star} + \kappa_\rmf(x_{k,\ell,\star})^\rmT R\, \kappa_\rmf(x_{k,\ell,\star})\Big).\label{eq:Deltaf_def}
		\end{align}
		Since $x_{k,\ell,\star}\in\SX_\rmf$, it follows from \ref{ts:iii} of Proposition~\ref{prop:terminal_set}, whose terminal ingredients $P\succ0$ and $\kappa_\rmf$ exist by \ref{A:lipschitz} and \ref{A:stabilizable}, that
		\begin{equation}
			\Delta_\rmf\le 0.
			\label{eq:Deltaf_neg}
		\end{equation}
		By the optimality of the minimizer of ${\rm OCP}_{k+1}$,
		\begin{equation}
			V_\star(z_{k+1})\le J_{k+1}(u_{k+1,1,\rmc},\ldots,u_{k+1,\ell-1,\rmc}),\label{eq:opt_bound}
		\end{equation}
		and thus, since \eqref{eq:xk1_eq_x1star} implies $x_{k,1}=x_{k+1}$ and $u_{k,1,\star}=u_{k+1}$, combining \eqref{eq:costdiff}--\eqref{eq:opt_bound} yields
		\begin{align}
			V_\star(z_{k+1}) - V_\star(z_k) &\le -\tfrac{1}{2}x_{k+1}^\rmT Q\, x_{k+1} - \tfrac{1}{2}u_{k+1}^\rmT R\, u_{k+1}\nn\\
			&\le -c\,\|z_{k+1}\|^2,
			\label{eq:Vdecrease}
		\end{align}
		where, since $Q\succ0$ and $R\succ0$ by hypothesis, $c\triangleq \tfrac{1}{2}\min\{\lambda_{\min}(Q),\lambda_{\min}(R)\}>0$.
		
		Let $z_0\in\SF_\ell$.
		Since \eqref{eq:Vdecrease} implies that $\{V_\star(z_k)\}_{k=0}^\infty$ is nonincreasing and bounded below by zero, summing \eqref{eq:Vdecrease} from $0$ to $N-1$ yields, for all $N\in\BBZ^+$,
		\begin{equation}
			0\le c\sum_{k=1}^{N}\|z_{k}\|^2\le V_\star(z_0)-V_\star(z_N)\le V_\star(z_0),\label{eq:summability}
		\end{equation}
		where the lower and upper bounds imply that $\sum_{k=1}^{\infty}\|z_k\|^2$ exists, and thus $\lim_{k\to\infty}z_k=0$.
		
		To establish a local upper bound on $V_\star$, it follows from \eqref{eq:fSCDC} and \ref{A:bounded} that, for all $k\in\BBN$ and all $z_k=\matl x_k^\rmT& u_k^\rmT\matr^\rmT\in\BBR^{n+m}$,
		\begin{equation}
			\|f(x_k,u_k)\|\le \bar a\|x_k\|+\bar b\|u_k\|\le L_f \|z_k\|,\label{eq:f_lip0}
		\end{equation}
		where $L_f \triangleq \sqrt{2}\max\{\bar a,\bar b\}$.
		Define $\delta_1\triangleq \sqrt{\alpha/\lambda_{\max}(P)}/L_f$, and let $k\in\BBN$ be such that $\|z_k\|\le\delta_1$.
		Then, \eqref{eq:f_lip0} implies $V_\rmf(f(x_k,u_k))\le\lambda_{\max}(P)L_f^2\|z_k\|^2\le\alpha$, and thus \eqref{eq:Xf_def} implies $f(x_k,u_k)\in\SX_\rmf$.
		Consider the candidate sequence for ${\rm OCP}_k$ defined by $u_{k,j} = \kappa_\rmf(x_{k,j})$ for $j\in\{1,\ldots,\ell-1\}$, where $x_{k,1}=f(x_k,u_k)$ and $x_{k,j+1}=f(x_{k,j},\kappa_\rmf(x_{k,j}))$.
		By \ref{ts:i} and \ref{ts:ii} of Proposition~\ref{prop:terminal_set}, this candidate is feasible, and summing \eqref{eq:Vf_descent} along the candidate trajectory yields
		\begin{align*}
			&\tfrac{1}{2}\sum_{j=1}^{\ell-1}\(x_{k,j}^\rmT Qx_{k,j}+u_{k,j}^\rmT Ru_{k,j}\)+\tfrac{1}{2}V_\rmf(x_{k,\ell})\le \tfrac{1}{2}V_\rmf(x_{k,1}),
		\end{align*}
		and thus, for all $k\in\BBN$ satisfying $\|z_k\|\le\delta_1$,
		\begin{equation}
			V_\star(z_k)\le \tfrac{1}{2}V_\rmf(f(x_k,u_k))\le  c_2\|z_k\|^2,\label{eq:V_upper}
		\end{equation}
		where $c_2\triangleq \tfrac{1}{2}\lambda_{\max}(P)L_f^2$.
		
		Let $\epsilon\in(0,\infty)$, and define $\delta\triangleq\min\{\delta_1,\ \epsilon,\ \epsilon\sqrt{c/c_2}\}$.
		Let $\|z_0\|\le\delta$.
		Since $f(x_0,u_0)\in\SX_\rmf$, the terminal candidate above shows $z_0\in\SF_\ell$.
		It follows from \eqref{eq:Vdecrease}, \eqref{eq:summability}, and \eqref{eq:V_upper} that, for all $k\in\BBZ^+$,
		\begin{equation*}
			c\|z_k\|^2\le V_\star(z_{k-1})\le V_\star(z_0)\le c_2\|z_0\|^2\le c_2\,\delta^2,
		\end{equation*}
		and thus $\|z_k\|\le \delta\sqrt{c_2/c}\le\epsilon$; furthermore, $\|z_0\|\le\delta\le\epsilon$.
		Hence, the equilibrium $z=0$ is Lyapunov stable, which, combined with $\lim_{k\to\infty}z_k=0$, confirms asymptotic stability.
	\end{proof}
	
	\begin{remark}\rm
		\label{rem:semidefQ}
		If $Q$ is positive semidefinite but not positive definite, then \eqref{eq:Vdecrease} yields $u_{k}\to0$ and $Q^{1/2}x_{k}\to0$, and convergence $x_k\to 0$ can be recovered from the detectability of $(A_0,Q^{1/2})$ in \ref{A:stabilizable}; see \cite[Theorem~2.24]{rawlings2017model}.
		For clarity, Theorem~\ref{thm:stability} is stated with $Q\succ 0$, which is also required by Proposition~\ref{prop:terminal_set}.
	\end{remark}
	
	\begin{remark}\rm
		\label{rem:terminal_equality}
		Theorems~\ref{thm:feasibility} and \ref{thm:stability} also hold with the terminal equality constraint $x_{k,\ell}=0$ in place of $x_{k,\ell}\in\SX_\rmf$, with $\kappa_\rmf\equiv0$ and $V_\rmf\equiv0$ in \eqref{eq:candidate_U} and \eqref{eq:Deltaf_def}: since $f(0,0)=0$ by \eqref{eq:fSCDC}, the appended candidate control keeps the terminal state at the origin, and $\Delta_\rmf=0$ in \eqref{eq:costdiff}.
		This variant, which dates to \cite{keerthi1988optimal}, does not require \ref{A:stabilizable} and is the appropriate formulation for systems whose linearization at the origin is not stabilizable, such as those subject to Brockett's obstruction \cite{brockett1983asymptotic}; it is used in Section~\ref{sec:ex_brockett}, where the terminal equality constraint is implemented through its exact-penalty relaxation \cite{rawlings2017model}.
	\end{remark}
	
	Theorem~\ref{thm:stability} presumes that the applied sequence minimizes ${\rm OCP}_k$, which iSCD-MPC does not deliver: the stopping criterion \eqref{eq:rhokk} terminates the iteration after finitely many steps, and Proposition~\ref{prop:kkt_gap} bounds the residual that remains.
	The following result removes that presumption.
	
	\subsection{Suboptimal Stability under Early Termination}
	\label{sec:subopt}
	
	In practice, for all $k\in\BBN$, the applied sequence is $U_{k|\rho_k}$, which is generally not a minimizer of ${\rm OCP}_k$.
	Following the framework of \cite{scokaert1999suboptimal}, the next result shows that stability is retained provided that the applied sequence is feasible and improves upon the shifted candidate.
	For all $k\in\BBZ^+$, the shifted candidate constructed from the applied sequence at step $k-1$ is $\SSS_{k-1}(U_{k-1|\rho_{k-1}})$ given by \eqref{eq:shift_op}, which by \eqref{eq:warm_start} coincides with the warm start $U_{k|1}$.
	This construction coincides exactly with the warm-start initialization \eqref{eq:warm_start}.
	
	\begin{theorem}\label{thm:relaxed_lyapunov}\rm
		Consider the system \eqref{eq:xk1} controlled by iSCD-MPC \eqref{eq:fSCDC}--\eqref{eq:applied_control} with stopping criterion \eqref{eq:rhokk} and warm start \eqref{eq:warm_start}. Assume that \ref{A:lipschitz}--\ref{A:stabilizable} and the hypotheses of Proposition~\ref{prop:terminal_set} hold with $Q\succ0$.
		Furthermore, assume that, for all $k\in\BBN$, the applied sequence $U_{k|\rho_k}$ is feasible for ${\rm OCP}_k$ given by \eqref{eq:OCP} and, for all $k\in\BBZ^+$, satisfies the descent condition
		\begin{equation}
			\bar J_k(U_{k|\rho_k})\le \bar J_k\big(\SSS_{k-1}(U_{k-1|\rho_{k-1}})\big),\label{eq:subopt_descent}
		\end{equation}
		where $\bar J_k$ is given by \eqref{eq:Jbar_def} with $P$ given by \eqref{eq:DARE}.
		Then, for all $z_0\in\SF_\ell$ such that $U_{0|\rho_0}$ is feasible for ${\rm OCP}_0$, $\lim_{k\to\infty}z_k= 0$.
		If, in addition, there exists $c_V\in(0,\infty)$ such that, for all $z_0$ in a neighborhood of the origin, $\bar J_0(U_{0|\rho_0})\le c_V\|z_0\|^2$, then the equilibrium $z=0$ of the closed-loop dynamics is asymptotically stable.
	\end{theorem}
	\begin{proof}
		For all $k\in\BBN$ and $z_k\in\SF_\ell$ such that $U_{k|\rho_k}$ is feasible for ${\rm OCP}_k$, define $V\colon\SF_\ell\to[0,\infty)$ by
		\begin{equation*}
			V(z_k)\triangleq \bar J_k(U_{k|\rho_k}).
		\end{equation*}
		Let $k\in\BBN$, and suppose $U_{k|\rho_k}$ is feasible for ${\rm OCP}_k$.
		By the proof of Theorem~\ref{thm:feasibility}, the shifted candidate $\SSS_k(U_{k|\rho_k})$ is feasible for ${\rm OCP}_{k+1}$.
		Repeating the computation \eqref{eq:costdiff}--\eqref{eq:Deltaf_neg} with $U_{k|\rho_k}$ in place of the minimizer yields
		\begin{equation}
			\bar J_{k+1}\big(\SSS_k(U_{k|\rho_k})\big)\le V(z_k) - \tfrac{1}{2}x_{k+1}^\rmT Q\, x_{k+1} - \tfrac{1}{2}u_{k+1}^\rmT R\, u_{k+1},\label{eq:Jk1_bound}
		\end{equation}
		where $x_{k+1}=x_{k,1}$ and $u_{k+1}=u_{k,1|\rho_k}$, and where $\Delta_\rmf\le 0$ holds by \ref{ts:iii} of Proposition~\ref{prop:terminal_set}, whose terminal ingredients exist by \ref{A:lipschitz} and \ref{A:stabilizable}.
		It follows from \eqref{eq:subopt_descent} at step $k+1$ and \eqref{eq:Jk1_bound} that
		\begin{equation}
			V(z_{k+1})-V(z_k)\le -c\|z_{k+1}\|^2,\label{eq:V_subopt_decrease}
		\end{equation}
		where, since $Q\succ0$ and $R\succ0$ by hypothesis, $c\triangleq \tfrac{1}{2}\min\{\lambda_{\min}(Q),\lambda_{\min}(R)\}>0$.
		
		Let $z_0\in\SF_\ell$ be such that $U_{0|\rho_0}$ is feasible for ${\rm OCP}_0$.
		Since \eqref{eq:V_subopt_decrease} implies that $\{V(z_k)\}_{k=0}^\infty$ is nonincreasing and bounded below by zero, summing \eqref{eq:V_subopt_decrease} from $0$ to $N-1$ yields, for all $N\in\BBZ^+$,
		\begin{equation}
			0\le c\sum_{k=1}^{N}\|z_{k}\|^2\le V(z_0)-V(z_N)\le V(z_0),\label{eq:summability_subopt}
		\end{equation}
		where the lower and upper bounds imply that $\sum_{k=1}^{\infty}\|z_k\|^2$ exists, which implies $\lim_{k\to\infty}z_k=0$.
		
		Finally, assume there exists $c_V\in(0,\infty)$ such that, for all $z_0$ in a neighborhood of the origin, $V(z_0)\le c_V\|z_0\|^2$.
		Let $\epsilon\in(0,\infty)$, and choose $\delta\in(0,\infty)$ such that $\|z_0\|\le\delta$ implies $V(z_0)\le c_V\|z_0\|^2$ and $\delta\le \epsilon\sqrt{c/c_V}$.
		It follows from \eqref{eq:V_subopt_decrease} and \eqref{eq:summability_subopt} that, for all $k\in\BBZ^+$,
		\begin{equation*}
			c\|z_k\|^2\le V(z_{k-1})\le V(z_0)\le c_V\|z_0\|^2\le c_V\,\delta^2,
		\end{equation*}
		and thus $\|z_k\|\le \delta\sqrt{c_V/c}\le\epsilon$; furthermore, $\|z_0\|\le\delta\le\epsilon$.
		Hence, the equilibrium $z=0$ is Lyapunov stable, which, combined with $\lim_{k\to\infty}z_k=0$, confirms asymptotic stability.
	\end{proof}
	
	\begin{remark}\rm
		\label{rem:descent_practice}
		The descent condition \eqref{eq:subopt_descent} can be enforced directly: since the warm start \eqref{eq:warm_start} equals $\SSS_{k-1}(U_{k-1|\rho_{k-1}})$, the algorithm evaluates $\bar J_k$ at that sequence when $i=1$ and accepts $U_{k|\rho_k}$ only if \eqref{eq:subopt_descent} holds, applying the warm start otherwise; this safeguard preserves \eqref{eq:V_subopt_decrease} regardless of the number of iterations performed.
		Feasibility of $U_{k|\rho_k}$ with respect to the nonlinear dynamics holds exactly at the fixed point, where the frozen-coefficient and nonlinear predicted trajectories coincide (see the proof of Proposition~\ref{prop:kkt_gap}); under early termination, the constraint mismatch is bounded by a constant times $\|U_{k|\rho_k}-U_{k,\star}\|\le\gamma^{\rho_k-1}\|U_{k|1}-U_{k,\star}\|$ by \eqref{eq:xj_lip} and \eqref{eq:geo_convergence}, and can be absorbed by constraint tightening.
		
		The same safeguard supplies the quadratic bound assumed in Theorem~\ref{thm:relaxed_lyapunov}, provided it is applied at $k=0$ as well.
		Note first that the descent condition \eqref{eq:subopt_descent} is imposed for $k\in\BBZ^+$ and that the initialization at $k=0$ is the cold start \eqref{eq:u0_init}, which carries no bound of the required form; the bound must therefore be secured separately at the initial step.
		Let $\delta_1$ be as in the proof of Theorem~\ref{thm:stability}, and let $\|z_0\|\le\delta_1$, so that $f(x_0,u_0)\in\SX_\rmf$.
		The terminal-law sequence defined by $u_{0,j}\triangleq\kappa_\rmf(x_{0,j})$ for $j\in\{1,\ldots,\ell-1\}$, where $x_{0,1}=f(x_0,u_0)$ and $x_{0,j+1}=f(x_{0,j},\kappa_\rmf(x_{0,j}))$, is feasible for ${\rm OCP}_0$ by \ref{ts:i} and \ref{ts:ii} of Proposition~\ref{prop:terminal_set}, and summing \eqref{eq:Vf_descent} along it gives a rollout cost of at most $\tfrac{1}{2}V_\rmf(f(x_0,u_0))\le c_2\|z_0\|^2$, with $c_2$ as in \eqref{eq:V_upper}.
		Hence, if at $k=0$ the algorithm accepts $U_{0|\rho_0}$ only when its rollout cost does not exceed that of the terminal-law sequence, then $\bar J_0(U_{0|\rho_0})\le c_2\|z_0\|^2$, and the hypothesis of Theorem~\ref{thm:relaxed_lyapunov} holds with $c_V=c_2$.
		The quadratic bound is thus not an additional assumption on the closed loop but a property that the algorithm can enforce by the same comparison it already performs at every other step.
	\end{remark}
	
	\subsection{Domain of Attraction}
	\label{sec:doa}
	
	\begin{proposition}\label{prop:doa_expansion}\rm
		For all $\ell\in\{2,3,\ldots\}$, let $\SF_\ell$ denote the feasible set \eqref{eq:Fl_def} with horizon $\ell$, and assume that the hypotheses of Proposition~\ref{prop:terminal_set} hold.
		Then, for all $\ell_1,\ell_2\in\{2,3,\ldots\}$ satisfying $\ell_1<\ell_2$, $\SF_{\ell_1}\subseteq \SF_{\ell_2}$.
	\end{proposition}
	\begin{proof}
		The result is the standard horizon monotonicity of terminal-set formulations \cite[Prop.~2.16]{rawlings2017model}, and follows by appending the terminal control law.
		Let $z\in\SF_{\ell_1}$ with feasible sequence $u_1,\ldots,u_{\ell_1-1}$ and trajectory $x_1,\ldots,x_{\ell_1}$, and extend it by $u_j\triangleq\kappa_\rmf(x_j)$ for $j\in\{\ell_1,\ldots,\ell_2-1\}$.
		Since $x_{\ell_1}\in\SX_\rmf$, \ref{ts:ii} of Proposition~\ref{prop:terminal_set} implies by induction that $x_j\in\SX_\rmf\subseteq\SX$ for all $j\in\{\ell_1,\ldots,\ell_2\}$, and \ref{ts:i} implies $\kappa_\rmf(x_j)\in\SU$; the extended sequence is therefore feasible for horizon $\ell_2$, and thus $z\in\SF_{\ell_2}$.
	\end{proof}

	The results above concern the nominal system \eqref{eq:xk1}.
	We close the analysis by admitting bounded disturbances.
	
	\subsection{Robust Feasibility under Bounded Disturbances}
	\label{sec:robust}
	
	Consider the disturbed system
	\begin{equation}
		x_{k+1}=f(x_k,u_k)+w_k,\label{eq:disturbed}
	\end{equation}
	where, for all $k\in\BBN$, $w_k\in\SW\subset\BBR^n$ and $\SW$ is compact and convex with $0\in\SW$.
	Following tube-based MPC \cite{mayne2005robust,mayne2011tube}, let $\bar x_{k}$ and $\bar u_{k}$ denote a nominal trajectory and control satisfying $\bar x_{k+1}=f(\bar x_k,\bar u_k)$, and let the applied control be the ancillary feedback
	\begin{equation}
		u_k= \bar u_{k}+K_\rmf(x_k-\bar x_{k}),\label{eq:tube_control}
	\end{equation}
	where $K_\rmf$ is given by \eqref{eq:Kf_def}.
	It follows from \eqref{eq:disturbed}, \eqref{eq:tube_control}, and the nominal dynamics that the tracking error $e_k\triangleq x_k-\bar x_k$ evolves according to
	\begin{equation}
		e_{k+1} = f(\bar x_k+e_k,\bar u_k+K_\rmf e_k)-f(\bar x_k,\bar u_k)+w_k.
		\label{eq:e_dynamics}
	\end{equation}
	
	A set $\SZ\subset\BBR^n$ is \emph{robustly positively invariant} for \eqref{eq:e_dynamics} over a nominal operating domain $\SX_{\rm tight}\times\SU_{\rm tight}\subseteq\SX\times\SU$ if $e_k\in\SZ$ implies $e_{k+1}\in\SZ$ for all $\bar x_k\in\SX_{\rm tight}$, all $\bar u_k\in\SU_{\rm tight}$, and all $w_k\in\SW$; see \cite{rakovic2005invariant} for the linear case.
	
	To guarantee that such an invariant cross-section exists and to bound the linearization mismatch incurred by the nonlinear flow, we impose the following condition.
	
	\begin{enumerate}[label=({A\arabic*}),itemindent=5pt]
		\setcounter{enumi}{4}
		\item \label{A:tube}
		There exist compact convex sets $\SE,\SZ\subset\BBR^n$, with $0\in\SE$, such that, defining the tightened constraint sets
		\begin{equation}
			\SX_{\rm tight}\triangleq \SX\ominus\SZ,\qquad \SU_{\rm tight}\triangleq \SU\ominus (K_\rmf\SZ),\label{eq:tightened}
		\end{equation}
		the following statements hold:
		\begin{enumerate}[label=({\it \alph*}),itemindent=1pt]
			\item for all $\bar x\in\SX_{\rm tight}$, $\bar u\in\SU_{\rm tight}$, and $e\in\SZ$,
			\begin{equation}
				f(\bar x+e,\bar u+K_\rmf e)-f(\bar x,\bar u)\in \{A_{\rm cl}\, e\}\oplus\SE.\label{eq:mismatch}
			\end{equation}
			\item $A_{\rm cl}\SZ\oplus\SE\oplus\SW\subseteq\SZ$.\label{item:rpi}
		\end{enumerate}
	\end{enumerate}
	
	Condition \ref{A:tube}(a) bounds the deviation of the nonlinear error dynamics from the linear error dynamics generated by $A_{\rm cl}$, whereas \ref{A:tube}(b) absorbs that mismatch and the disturbance into $\SZ$, establishing the robust positive invariance of $\SZ$.
	If $f$ is linear, then \eqref{eq:mismatch} holds with $\SE=\{0\}$, and \ref{A:tube}(b) holds with $\SZ$ equal to the minimal robust positively invariant set of $A_{\rm cl}$ and $\SW$ \cite{rakovic2005invariant}, which is compact since \eqref{eq:Riccati_identity} implies that $A_{\rm cl}$ is Schur.
	For nonlinear $f$, \ref{A:lipschitz} implies that the left-hand side of \eqref{eq:mismatch} deviates from $A_{\rm cl}e$ by a term of order $\|e\|(\|\bar x\|+\|\bar u\|+\|e\|)$, so that \ref{A:tube} is a smallness condition on $\SW$ and on the operating domain; see \cite{mayne2011tube} for constructions.
	
	\begin{proposition}\label{prop:robust_feasibility}\rm
		Consider the disturbed system \eqref{eq:disturbed} with control \eqref{eq:tube_control}, and assume that \ref{A:tube} holds, the nominal iSCD-MPC optimization with constraints \eqref{eq:tightened} is feasible at $k=0$, and $x_0-\bar x_{0}\in\SZ$.
		Then, for all $k\in\BBN$, $x_k\in\SX$ and $u_k\in\SU$.
	\end{proposition}
	\begin{proof}
		The argument is that of tube-based MPC \cite[Sec.~3]{mayne2011tube}, with \ref{A:tube} supplying the invariance that the nonlinear flow would otherwise not provide.
		We show by induction that $e_k\in\SZ$ for all $k\in\BBN$, the base case holding by hypothesis.
		By Theorem~\ref{thm:feasibility} applied to the nominal system with constraints \eqref{eq:tightened}, nominal feasibility at $k=0$ implies $\bar x_k\in\SX_{\rm tight}$ and $\bar u_k\in\SU_{\rm tight}$ for all $k\in\BBN$, so that \eqref{eq:e_dynamics}, \eqref{eq:mismatch}, and $w_k\in\SW$ give
		\begin{equation*}
			e_{k+1}\in \{A_{\rm cl}e_k\}\oplus\SE\oplus\SW\subseteq A_{\rm cl}\SZ\oplus\SE\oplus\SW\subseteq\SZ,
		\end{equation*}
		the last inclusion by \ref{A:tube}(b).
		The constraint satisfaction now follows from \eqref{eq:tightened}, since $x_k=\bar x_k+e_k\in(\SX\ominus\SZ)\oplus\SZ\subseteq\SX$ and, similarly, $u_k=\bar u_k+K_\rmf e_k\in(\SU\ominus(K_\rmf\SZ))\oplus(K_\rmf\SZ)\subseteq\SU$.
	\end{proof}
	
	All of the foregoing presumes that $x_k$ is measured.
	The BOCF realization of Section~\ref{sec:OF} removes that presumption without a separation argument, because its reconstruction error is identically zero.
	
	\subsection{Asymptotic Stability under Output Feedback}
	\label{sec:OF_stability}
	
	\begin{proposition}\label{prop:bocf_stability}\rm
		Consider the input-output system \eqref{eq:yyy} with BOCF realization \eqref{eq:xk1_OF}--\eqref{eq:eta2}, controlled by iSCD-MPC \eqref{eq:fSCDC}--\eqref{eq:applied_control} applied to the BOCF state $x_k$ given by \eqref{eq:eta1} and \eqref{eq:eta2}.
		If the hypotheses of Theorem~\ref{thm:stability} (respectively, Theorem~\ref{thm:relaxed_lyapunov}) hold for the BOCF realization, then, for all initial extended states $z_0\in\SF_\ell$ (respectively, all $z_0$ such that $U_{0|\rho_0}$ is feasible for ${\rm OCP}_0$), $\lim_{k\to\infty}y_k= 0$.
	\end{proposition}
	\begin{proof}
		Since the BOCF state \eqref{eq:eta1}--\eqref{eq:eta2} is an exact function of the past $n$ outputs and inputs, the state used by the controller equals the state of \eqref{eq:xk1_OF} with zero reconstruction error.
		Hence, Theorem~\ref{thm:stability} (respectively, Theorem~\ref{thm:relaxed_lyapunov}) applied to \eqref{eq:xk1_OF} yields $\lim_{k\to\infty}z_k= 0$ for $z_k=\matl x_k^\rmT&u_k^\rmT\matr^\rmT$, which implies $\lim_{k\to\infty}x_k= 0$.
		Since, in addition, \eqref{eq:BC_BOCF} implies $\|C\|=1$, it follows from \eqref{eq:yk_OF} that $\|y_k\|\to0$ as $k\to \infty$.
	\end{proof}
	
	\begin{remark}\rm
		Because the BOCF reconstruction error is identically zero, the guarantees of Theorems~\ref{thm:feasibility}--\ref{thm:relaxed_lyapunov} transfer to output feedback without a separation argument; this is in contrast to observer-based nonlinear MPC, where estimation error generally degrades the feasibility and stability margins.
	\end{remark}
	
	\section{Computational Complexity Analysis}
	\label{sec:complexity}
	
	This section compares the per-step computational cost of iSCD-MPC with that of iLQR, SQP, and LPV-MPC.
	We first fix the cost model, so that the comparison itself can be stated as a single result.
	
	Throughout, $n$, $m$, and $\ell$ are the state dimension, the control dimension, and the horizon length, and $\varepsilon\in(0,\infty)$ is the tolerance to which the inner optimization is solved.
	Costs are counted in floating-point operations under the following conventions, which are standard for interior-point implementations \cite{rao1998application,nocedal2006numerical}.
	
	\begin{enumerate}[label=({C\arabic*}),itemindent=5pt]
		\item \label{C:dense}
		Solving one quadratic program in $N$ decision variables by an interior-point method costs $O(N^3)$ when the Karush-Kuhn-Tucker matrix is treated as dense, and $O(\ell\,(n+m)^3)$ when that matrix is block banded with $\ell$ diagonal blocks of order $n+m$ and is factored by the corresponding Riccati recursion \cite{rao1998application}.
		In both cases the number of interior-point iterations is bounded by a constant independent of $\ell$, $n$, and $m$, and is absorbed into the stated cost.
		\item \label{C:ipm}
		The number of outer iterations required to reach tolerance $\varepsilon$, namely the SCDC iterations \eqref{eq:rhokk} for iSCD-MPC and the Newton-type iterations for iLQR and SQP, is $O(\log(1/\varepsilon))$.
		\item \label{C:jac}
		Let $C_f$ denote the cost of one evaluation of $f$.
		Evaluating $J_f$ by finite differences requires one evaluation of $f$ per column, hence $n+m$ evaluations, and costs $O((n+m)C_f)$ per stage.
		For the SCDC form \eqref{eq:fSCDC}, forming $A(x,u)x$ and $B(x,u)u$ gives $C_f=O(n^2+nm)=O(n(n+m))$, so that a finite-difference Jacobian costs $O(n(n+m)^2)$ per stage.
	\end{enumerate}
	
	For all $k\in\BBN$, let $C_{\rm iSCD}$, $C_{\rm iLQR}$, $C_{\rm SQP}$, and $C_{\rm LPV}$ denote the number of operations performed by the corresponding method at step $k$, and let $C_{\rm Jac}$ denote the operations spent on Jacobian evaluation and line search by iLQR at step $k$.
	The methods are those of Section~\ref{sec:setup}: iSCD-MPC solves the sequence of quadratic programs \eqref{eq:QP_cost}--\eqref{eq:QP_box} with the block-banded structure of \eqref{eq:Hessian} retained; iLQR \cite{TodorovLi2005} performs a Jacobian linearization, a backward Riccati pass, and a line-searched forward rollout per iteration; dense SQP \cite{nocedal2006numerical,diehl2002real} treats the Karush-Kuhn-Tucker system in the $N=O(\ell(n+m))$ state and control variables as dense, without exploiting block-banded sparsity, or equivalently condenses the states out at cubic horizon cost, per outer iteration; and LPV-MPC \cite{cisneros2018lpv,morato2020model} freezes $(A_{k,j|i},B_{k,j|i})\equiv(\bar A_k,\bar B_k)$ for all $j\in\{1,\ldots,\ell-1\}$ and solves one quadratic program per step.
	
	\begin{proposition}\label{prop:complexity}\rm
		Assume that \ref{C:dense}--\ref{C:jac} hold and that, for all $k\in\BBN$, the number of iSCD-MPC iterations satisfies the bound \eqref{eq:iter_bound} of Proposition~\ref{prop:iteration_bound}.
		Then, for all $k\in\BBN$, the following statements hold:
		\begin{enumerate}[label=({\it \roman*}),itemindent=1pt]
			\item\label{cx:iscd} $C_{\rm iSCD} = O\big(\ell(n+m)^3\log(1/\varepsilon)\big)$.
			\item\label{cx:ilqr} $C_{\rm iLQR} = O\big(\ell(n+m)^3\log(1/\varepsilon)\big)+C_{\rm Jac}$.
			\item\label{cx:sqp} $C_{\rm SQP} = O\big(\ell^3(n+m)^3\log(1/\varepsilon)\big)$.
			\item\label{cx:lpv} $C_{\rm LPV} = O\big(\ell(n+m)^3\big)$.
		\end{enumerate}
	\end{proposition}
	\begin{proof}
		To prove \ref{cx:iscd}, note that each iSCD-MPC iteration performs the rollout \eqref{eq:xkj+1i} and evaluates \eqref{eq:AB_SCDC} at each of the $\ell-1$ stages, which by \ref{C:jac} costs $O(\ell\,C_f)=O(\ell\,n(n+m))$, and then solves \eqref{eq:QP_cost}--\eqref{eq:QP_box}.
		The equality constraints \eqref{eq:QP_dyn} couple only consecutive stages, so the Karush-Kuhn-Tucker matrix is block banded with $\ell$ blocks of order $n+m$, and \ref{C:dense} gives $O(\ell(n+m)^3)$ for the solve.
		Since $n(n+m)\le(n+m)^3$, the solve dominates the coefficient evaluation, and each iteration costs $O(\ell(n+m)^3)$.
		By \eqref{eq:iter_bound} and \ref{C:ipm}, the number of iterations is $O(\log(1/\varepsilon))$, which confirms \ref{cx:iscd}.
		Statement \ref{cx:ilqr} follows from the same accounting, because the backward Riccati pass of iLQR factors the same block-banded structure at cost $O(\ell(n+m)^3)$ per iteration; the Jacobian evaluations and the line search contribute the separate term $C_{\rm Jac}$, which by \ref{C:jac} is $O(\ell(n+m)C_f)$ per iteration when finite differences are used.
		Summing over the $O(\log(1/\varepsilon))$ iterations of \ref{C:ipm} gives the per-step total $C_{\rm Jac}=O(\ell(n+m)C_f\log(1/\varepsilon))$, which for the SCDC form \eqref{eq:fSCDC} is $O(\ell\,n(n+m)^2\log(1/\varepsilon))$; this is the quantity denoted $C_{\rm Jac}$ in \ref{cx:ilqr} and in Table~\ref{tab:complexity}.
		To prove \ref{cx:sqp}, note that general-purpose dense SQP treats the Karush-Kuhn-Tucker matrix in the $N=O(\ell(n+m))$ state and control variables as dense without exploiting the block-banded sparsity, so that \ref{C:dense} gives $O(\ell^3(n+m)^3)$ for each subproblem; condensing the states out instead leaves a dense quadratic program in $O(\ell m)$ variables whose formation and dense solution scale as $O(\ell^3)$ as well.
		Multiplying by the $O(\log(1/\varepsilon))$ outer iterations of \ref{C:ipm} confirms \ref{cx:sqp}.
		The Jacobian and Hessian-approximation costs, $O(\ell(n+m)C_f)$ and $O(\ell^2(n+m)^2)$ per outer iteration, are dominated by the subproblem solve, since $C_f\le(n+m)^2$ for \eqref{eq:fSCDC} and $\ell\ge2$.
		Statement \ref{cx:lpv} follows from the accounting for \ref{cx:iscd} with the iteration count equal to one, since LPV-MPC freezes the coefficients and therefore solves a single quadratic program of the same block-banded structure per step.
	\end{proof}
	
	Two ratios follow immediately from Proposition~\ref{prop:complexity} and are the quantities measured in Section~\ref{sec:scaling}.
	Dividing \ref{cx:sqp} by \ref{cx:iscd} gives
	\begin{equation}
		\frac{C_{\rm SQP}}{C_{\rm iSCD}} = O(\ell^2),\label{eq:ratio_SQP}
	\end{equation}
	so that the per-step cost of iSCD-MPC scales two orders in $\ell$ below that of dense SQP, and dividing \ref{cx:iscd} by \ref{cx:lpv} gives
	\begin{equation}
		\frac{C_{\rm iSCD}}{C_{\rm LPV}} = O(\log(1/\varepsilon)),\label{eq:ratio_LPV}
	\end{equation}
	so that iSCD-MPC pays only a logarithmic factor for iterating the coefficients to tolerance $\varepsilon$.
	Table~\ref{tab:complexity} collects the four costs.
	
	\begin{table}[t]
		\centering
		\caption{Per-step computational cost of the four methods, from Proposition~\ref{prop:complexity}, where $C_{\rm Jac}$ denotes Jacobian-evaluation and line-search overhead and SQP refers to dense (condensed) implementations.}
		\label{tab:complexity}
		\small
		\begin{tabular}{lc}
			\toprule
			Method & Per-step cost\\
			\midrule
			iSCD-MPC & $O(\ell(n+m)^3\log(1/\varepsilon))$\\
			iLQR & $O(\ell(n+m)^3\log(1/\varepsilon))+C_{\rm Jac}$ \\
			SQP (dense) & $O(\ell^3(n+m)^3\log(1/\varepsilon))$\\
			LPV-MPC & $O(\ell(n+m)^3)$\\
			\bottomrule
		\end{tabular}
	\end{table}
	
	Four consequences follow.
	First, \ref{cx:iscd} and \ref{cx:ilqr} share a leading term, so iSCD-MPC attains the scaling of iLQR while evaluating no Jacobians.
	The size of that advantage depends on $C_f$, and it is worth being precise about the regime in which it is asymptotic rather than merely a constant factor.
	Comparing $C_{\rm Jac}=O(\ell(n+m)C_f\log(1/\varepsilon))$ with the leading term $O(\ell(n+m)^3\log(1/\varepsilon))$ shows that the Jacobian overhead dominates if and only if $C_f$ grows faster than $(n+m)^2$.
	For the SCDC form \eqref{eq:fSCDC} itself, $C_f=O(n(n+m))$, so the two are of the same order and the ratio is $n/(n+m)\le1$; on such systems the benefit of dispensing with $J_f$ is a constant factor together with the absence of truncation error, not a change of order.
	The overhead becomes asymptotically dominant when a single evaluation of $f$ is expensive relative to $(n+m)^2$, as occurs when $f$ is defined by an embedded simulation, a table lookup, or an implicit solve, and it becomes decisive when the derivatives of $f$ are unavailable in closed form, in which case iLQR and SQP cannot be applied without finite differencing at all.
	Second, \ref{cx:sqp} is a statement about dense implementations only.
	Structure-exploiting SQP based on sparse Riccati-factored solvers \cite{frison2020hpipm,verschueren2022acados} recovers the $O(\ell(n+m)^3)$ per-quadratic-program scaling, and relative to such implementations the difference is the elimination of Jacobian and sensitivity computation rather than the scaling itself; the comparison \eqref{eq:ratio_SQP}, and the measurements reported in Section~\ref{sec:examples}, refer to general-purpose dense implementations, as exemplified by MATLAB {\tt fmincon}.
	Third, \eqref{eq:ratio_SQP} and \eqref{eq:ratio_LPV} are statements about a fixed system as $\ell$ alone varies, and are tested as such in Section~\ref{sec:scaling}, not by comparing benchmarks of differing dimension against one another.
	Finally, we note that \ref{cx:iscd} presumes the banded structure is retained; condensing the states out of \eqref{eq:QP_cost}--\eqref{eq:QP_box}, as in the explicit map \eqref{eq:Hmap}, forfeits it and yields the cubic scaling of \ref{cx:sqp}, a distinction that Section~\ref{sec:scaling} measures directly.
	
	\section{Numerical Study}
	\label{sec:examples}
	
	This section presents a numerical study designed around the theoretical results of Sections~\ref{sec:theory} and \ref{sec:complexity}.
	The study comprises three benchmarks, each of which isolates a distinct facet of the analysis.
	Example~\ref{ex:quadrotor} considers stabilization of a planar quadrotor with thrust and torque saturation; this benchmark provides the head-to-head comparison of iSCD-MPC with iLQR, SQP, and LPV-MPC and empirically validates the contraction estimate (Proposition~\ref{prop:contraction}), the iteration bound (Proposition~\ref{prop:iteration_bound}), the cost convergence (Proposition~\ref{prop:cost_decrease}), and the boundedness of the adapted terminal penalty (Proposition~\ref{prop:Pk_bounded}).
	Example~\ref{ex:brockett} considers the nonholonomic integrator, which arises as the chained-form kinematics of a differential-drive mobile robot and whose linearization at every equilibrium violates the stabilizability requirement of frozen-coefficient control (i.e., Brockett's obstruction \cite{brockett1983asymptotic}); this benchmark demonstrates a structural failure of LPV-MPC that iSCD-MPC avoids, and validates the horizon monotonicity of the feasible set (Proposition~\ref{prop:doa_expansion}).
	Example~\ref{ex:triple_int} considers output-feedback control of a nonminimum-phase sampled-data plant with output-dependent actuator effectiveness and asymmetric control saturation, validating the BOCF architecture of Section~\ref{sec:OF} (Proposition~\ref{prop:bocf_stability}) and separating iSCD-MPC from LPV-MPC through a state-dependent input matrix.
	All simulation code is implemented in MATLAB and is available from the author.
	
	\subsection{Setup, Saturation Handling, and Comparison Protocol}
	\label{sec:setup}
	
	{\it Sampled-data simulation.}
	All benchmarks are simulated as sampled-data systems: MATLAB {\tt ode45} (relative and absolute tolerances $10^{-8}$) integrates the continuous-time dynamics under the zero-order-hold control $u(t)=u_k$ for $t\in[k\Ts,(k+1)\Ts)$, where $k\in\BBN$, $\Ts$ is the sampling time, and the SCDC model used for prediction is the Euler discretization of the continuous-time dynamics with step $\Ts$.
	For a nonlinear plant, no finite-order discretization reproduces the sampled-data flow exactly, so a mismatch between the prediction model and the simulated plant is intrinsic to sampled-data nonlinear control rather than a choice made here; the reported closed-loop performance therefore includes the effect of that mismatch.
	The QP \eqref{eq:QP_cost}--\eqref{eq:QP_box} is solved by MATLAB {\tt quadprog}.
	
	{\it Saturation as a control-dependent coefficient.}
	Each benchmark includes a componentwise magnitude saturation $\sigma\colon\BBR^m\to\BBR^m$, where, for all $s\in\{1,\ldots,m\}$, $\sigma(u)_{(s)}\triangleq \max\{\underline u_{(s)},\min\{u_{(s)},\overline u_{(s)}\}\}$ and $\underline u_{(s)}<0<\overline u_{(s)}$.
	Rather than enforce the saturation as a hard QP constraint, we exploit the SCDC structure \eqref{eq:fSCDC} by writing
	\begin{equation}
		\sigma(u) = D(u)\,u,\qquad D(u)\triangleq {\rm diag}\!\(\frac{\sigma(u)_{(1)}}{u_{(1)}},\ldots,\frac{\sigma(u)_{(m)}}{u_{(m)}}\), \label{eq:sat_SCDC}
	\end{equation}
	where, for all $s\in\{1,\ldots,m\}$, each diagonal entry is defined to be $1$ at $u_{(s)}=0$, and by absorbing $D(u)$ into the control-dependent input matrix, that is, replacing $B(x,u)$ by $B(x,\sigma(u))D(u)$.
	Since $\sigma$ is the identity in a neighborhood of the origin, the factorization \eqref{eq:sat_SCDC} leaves the linearization $(A_0,B_0)$ unchanged.
	The map $D$ is globally Lipschitz continuous but not differentiable on the saturation boundary; the differentiability requirement of \ref{A:lipschitz} can be met exactly by replacing $\sigma$ with a smooth surrogate (e.g., a scaled hyperbolic tangent), at the cost of conservatism near the limits.
	The formulation \eqref{eq:sat_SCDC} preserves the structure of \eqref{eq:Hessian} and guarantees that the applied control respects the actuator limits regardless of the QP solution. Furthermore, numerical simulations suggest that using \eqref{eq:sat_SCDC} is more reliable than direct enforcement of control bounds when the QP is warm started far from the constrained optimum.
	
	{\it Comparison methods.}
	iSCD-MPC is compared against three methods whose per-step complexities are derived in Section~\ref{sec:complexity}:
	\begin{enumerate}[label=({\it \roman*}),itemindent=1pt]
		\item {\it iLQR} \cite{TodorovLi2005,tomizukaiLQR}, with analytical Jacobians evaluated along the rolled-out trajectory, regularized backward pass, and backtracking line search in the forward pass;
		\item {\it SQP}\cite{nocedal2006numerical,diehl2002real}, implemented by MATLAB {\tt fmincon} using its sequential quadratic programming algorithm in the condensed control variables with analytical cost gradients, representing the dense implementation analyzed in \ref{cx:sqp} of Proposition~\ref{prop:complexity};
		\item {\it LPV-MPC} \cite{cisneros2018lpv,morato2020model}, which freezes $(\bar A_k,\bar B_k)=(A(x_{k,1},u_k),B(x_{k,1},u_k))$ across the horizon and solves a single QP per step.
	\end{enumerate}
	All methods use identical horizon $\ell$, weightings $Q$ and $R$, terminal penalty, sampling period $\Ts$, saturation, and warm-starting shift, so that performance differences are attributable to the algorithms rather than to tuning.
	
	{\it Performance metrics.}
	Three metrics are reported:
	the closed-loop cost
	\begin{equation}
		\SJ_N\triangleq \sum_{k=0}^{N-1}\(x_k^\rmT Q x_k+u_k^\rmT R u_k\),\label{eq:integrated_cost}
	\end{equation}
	accumulated over the $N$-step simulation interval;
	the average per-step computation time $\bar t_{\rm step}$ in milliseconds, which tests the analytical ratios \eqref{eq:ratio_SQP} and \eqref{eq:ratio_LPV};
	and, for iSCD-MPC, the average iteration count $\bar\rho_k$, which tests the bound \eqref{eq:iter_bound}.
	
	{\it Monte Carlo protocol.}
	Each metric is averaged over $M=20$ trials, in which for all $s\in\{1,\ldots,M\}$, the initial condition of trial $s$ is $x_{0,s}=x_0+\zeta_s$, where the components of $\zeta_s\in\BBR^n$ are drawn independently from the uniform distribution on $[-0.05\|x_0\|,0.05\|x_0\|]$.
	Sample means and sample standard deviations are reported.
	Computation times are measured on a 3.4~GHz Intel Core i7 with 64~GB RAM running MATLAB R2023b.
	All quadratic programs are solved by {\tt quadprog} with the interior-point-convex algorithm, optimality and step tolerances $10^{-10}$ and $10^{-12}$, and an iteration limit of 400; SQP is {\tt fmincon} with the {\tt sqp} algorithm and default tolerances.
	Except in Example~\ref{ex:brockett}, where the exact-penalty relaxation of the terminal equality constraint is used, the terminal penalty is $P$ given by \eqref{eq:DARE}, computed as {\tt idare}$(A_0,B_0,2Q,2R)$, and the terminal control law is $\kappa_\rmf(x)=K_\rmf x$ with $K_\rmf$ the corresponding gain.
	Continuous-time plants are integrated between samples by {\tt ode45}, whereas Example~\ref{ex:triple_int} is a discrete-time input-output system and is simulated by its defining recursion.
	Reported per-step times are medians over the simulation interval, which are insensitive to the first-call overhead of the solver.
	
	\subsection{Planar Quadrotor with Thrust and Torque Saturation}
	\label{sec:ex_quadrotor}
	
	\begin{example}
		\label{ex:quadrotor}
		{\it Planar quadrotor.}
		\rm
		Consider the planar quadrotor shown in Figure~\ref{fig:quad_schematic}, whose states are the horizontal position $p_\rmx$, the altitude $p_\rmz$, the pitch angle $\theta$, and the corresponding rates $\dot p_\rmx$, $\dot p_\rmz$, $\dot\theta$.
		The equations of motion are
		\begin{gather}
			m_{\rm q}\ddot p_\rmx = -T_\rmc\sin\theta, \qquad
			m_{\rm q}\ddot p_\rmz = T_\rmc\cos\theta - m_{\rm q}g, \label{eq:quad_trans}\\
			\Iq\ddot\theta = M_\rmp,\label{eq:quad_rot}
		\end{gather}
		where $m_{\rm q}$ is the mass, $\Iq$ is the pitch moment of inertia, $g$ is the gravitational acceleration, $T_\rmc\in[0,T_{\rm c,max}]$ is the total thrust, and $M_\rmp\in[-M_{\rm p,max},M_{\rm p,max}]$ is the pitch moment.
		The parameters $m_{\rm q}=0.033$ kg, $\Iq=1.96\times10^{-5}$ kg\,m$^2$, $T_{\rm c,max}=0.60$ N, and $M_{\rm p,max}=2\times10^{-3}$ N\,m correspond to a palm-sized quadrotor of the Crazyflie class \cite{giernacki2017crazyflie}, for which the hover thrust $m_{\rm q}g\approx0.32$ N is roughly half of $T_{\rm c,max}$.
		
		Define the state $x\triangleq \matl p_\rmx&p_\rmz&\theta&\dot p_\rmx&\dot p_\rmz&\dot\theta\matr^\rmT$ and the shifted control $u\triangleq\matl T_\rmc- m_{\rm q}g&M_\rmp\matr^\rmT$, so that the origin corresponds to hover and the thrust limits become the asymmetric bounds $u_{(1)}\in[-m_{\rm q}g,\,T_{\rm c,max}-m_{\rm q}g]$.
		Applying the Euler discretization with sampling period $\Ts$ and the saturation factorization \eqref{eq:sat_SCDC} to \eqref{eq:quad_trans} and \eqref{eq:quad_rot} yields the SCDC representation \eqref{eq:fSCDC} with
		\begin{align}
			A(x,u) &= I_6 + \Ts\matl 0_{3\times3} & I_3\\ \SG(x_{(3)}) & 0_{3\times3}\matr,\label{eq:Aquad}\\
			\SG(\theta)&\triangleq \matl 0&0&-g\,\dfrac{\sin\theta}{\theta}\\[4pt] 0&0&g\,\dfrac{\cos\theta-1}{\theta}\\[2pt] 0&0&0\matr,\nonumber\\
			B(x,u) &= \Ts\matl \multicolumn{2}{c}{0_{3\times2}}\\[3pt] -\dfrac{\sin x_{(3)}}{m_{\rm q}}&0\\[6pt] \dfrac{\cos x_{(3)}}{m_{\rm q}}&0\\[4pt] 0&\dfrac{1}{\Iq}\matr D(u),\label{eq:Bquad}
		\end{align}
		where $(\sin\theta)/\theta$ and $(\cos\theta-1)/\theta$ are extended continuously by their limits $1$ and $0$ at $\theta=0$, and $D$ is given by \eqref{eq:sat_SCDC}.
		The entries of \eqref{eq:Aquad} and \eqref{eq:Bquad} are bounded functions with bounded Lipschitz constants, so that \ref{A:lipschitz} and \ref{A:bounded} hold globally as discussed in Section~\ref{sec:setup}.
		The linearization at hover is $A_0=I_6+\Ts\matls 0_{3\times3}&I_3\\ \SG(0)&0_{3\times3}\matrs$ with $\SG(0)=\matls 0&0&-g\\0&0&0\\0&0&0\matrs$ and $B_0=\Ts\matls 0_{3\times2}\\ e_2/m_{\rm q}\ \ e_3/\Iq\matrs$, structured as in \eqref{eq:Bquad} with $D=I_2$, which is controllable, so that \ref{A:stabilizable} holds.

		\begin{figure}[t!]
			\centering
			\begin{tikzpicture}[scale=1.15, >=Stealth]
				
				% --- Inertial Reference Frame ---
				\draw[->, thick] (-2.6,-1.3) -- (-1.5,-1.3) node[below] { $p_\rmx$};
				\draw[->, thick] (-2.6,-1.3) -- (-2.6,-0.2) node[left] { $p_\rmz$};
				\fill (-2.6,-1.3) circle (1pt);
				
				% --- Gravity (Inertial Frame) ---
				\draw[->, thick] (0,-0.25) -- (0,-1.15) node[below] { $m_{\rm q}g$};
				
				% --- Inertial Vertical Reference Line ---
				\draw[dashed, thin, black!50] (0,0) -- (0,1.55);
				
				% --- Pitch Angle \theta (Inertial Vertical to Body Vertical) ---
				\draw[->, thick] (0,1.15) arc (90:112:1.15);
				\node at (101:1.35) { $\theta$};
				
				% --- Quadrotor Body Scope (Rotated by \theta = 22^\circ) ---
				\begin{scope}[rotate=22]
					
					% Main carbon-fiber airframe arm
					\draw[very thick, line cap=round] (-1.2,0) -- (1.2,0);
					
					% Fuselage / payload hub
					\draw[thick, fill=black!70, rounded corners=1pt] (-0.22,-0.14) -- (0.22,-0.14) -- (0.08,0.02) -- (-0.08,0.02) -- cycle;
					\fill (0,0) circle (1.5pt);
					
					% Motor masts
					\draw[thick] (-1.2,0) -- (-1.2,0.16);
					\draw[thick] (1.2,0) -- (1.2,0.16);
					
					% Propeller discs (shaded ellipses with center hub caps)
					\draw[thick, fill=black!10] (-1.2,0.18) ellipse (0.42 and 0.06);
					\fill (-1.2,0.18) circle (1pt);
					\draw[thick, fill=black!10] (1.2,0.18) ellipse (0.42 and 0.06);
					\fill (1.2,0.18) circle (1pt);
					
					% --- Thrust vector T_c (body +z axis) ---
					\draw[->, thick, blue!75!black] (0,0) -- (0,1.6) node[above left=-2pt] { $T_\rmc$};
					
					% --- Pitch moment M_p (concentric arc in right sector) ---
					\draw[->, thick, red!75!black] (20:0.65) arc (20:80:0.65) node[midway, above right=-2pt] { $M_\rmp$};
					
				\end{scope}
				
			\end{tikzpicture}
			\caption{Example~\ref{ex:quadrotor}. Planar quadrotor with total thrust $T_\rmc\in[0,T_{\rm c,max}]$ along the body axis and pitch torque $M_\rmp\in[-M_{\rm p,max},M_{\rm p,max}]$.}
			\label{fig:quad_schematic}
		\end{figure}
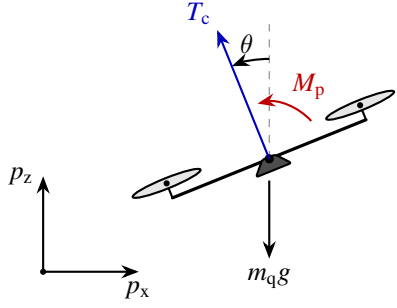
		The simulations use $\Ts=0.02$ s (i.e., 50-Hz update), $\ell=40$ (i.e., 0.8-s prediction window), $\rho=30$, $\varepsilon=10^{-4}$, $Q=\diag(20,20,4,0.5,0.5,0.02)$, $R=\diag(10,8000)$, and the initial condition $x_0=\matl 0.8&-0.5&0.70&0&0&0\matr^\rmT$, $u_0=0$, which corresponds to a $0.94$-m position offset with a $40^\circ$ initial pitch at rest.
		The terminal penalty is adapted online by \eqref{eq:Pk_update} with $P_0=P$ given by \eqref{eq:DARE}, thereby exercising Proposition~\ref{prop:Pk_bounded}, and the simulation interval is $5$ s.
		
		Figure~\ref{fig:quad_traj} shows the closed-loop trajectory and saturated control under iSCD-MPC: the quadrotor levels the initial $40^\circ$ pitch within approximately $0.3$ s using torque-limit and thrust-limit activity, and the position converges to hover in approximately $2.5$ s without constraint violation, illustrating Theorems~\ref{thm:feasibility} and \ref{thm:relaxed_lyapunov}.
		Figure~\ref{fig:contraction} reports $\|U_{k|i}-U_{k|i-1}\|$ versus $i$ at the cold start $k=0$, where $\|x_{k,1}\|=1.18$.
		The increments decay geometrically over the range in which Proposition~\ref{prop:contraction} applies, with a rate $\hat\gamma_0=0.354$ estimated by a least-squares fit to $\log\|U_{k|i}-U_{k|i-1}\|$; the fitted rate is shown as the dashed reference line.
		The decay follows the fitted rate over the entire range until the increments reach the double-precision floor near $10^{-13}$, beyond which they no longer decrease.
		Along the closed-loop trajectory, the warm start \eqref{eq:warm_start} places the initial iterate inside this fast-convergence neighborhood, so that the number of iterations required to meet the tolerance is small after the initial transient.
		Figure~\ref{fig:rho_history} reports the iteration count $\rho_k$ over the closed loop: $\rho_k$ equals $11$ at the first step, where the cold start \eqref{eq:u0_init} is used, and decreases within roughly $1.5$ s to a value of $2$; the mean is $\bar\rho_k=2.44$ over the whole interval and $2.34$ after the fifth step, and the cap $\rho=30$ is never reached, in agreement with the bound \eqref{eq:iter_bound} of Proposition~\ref{prop:iteration_bound}.
		Figure~\ref{fig:cost_convergence} shows $|\bar J_0(U_{0|i})-\bar J_0(U_{0|\rho_0})|$ versus $i$ at $k=0$.
		The magnitude decays at a geometric rate consistent with Proposition~\ref{prop:cost_decrease} over most of the range, with one pronounced dip near $i=7$ at which the iterate cost passes close to its limiting value before the decay resumes.
		Such a dip is admissible: Proposition~\ref{prop:cost_decrease} bounds $|\bar J_0(U_{0|i})-\bar J_0(U_{0,\star})|$ by a decaying envelope and does not assert that the magnitude decreases monotonically in $i$.
		Figure~\ref{fig:Pk_eigs} reports all six eigenvalues of $P_k$ versus the step index $k\in\{0,\ldots,N\}$; after a transient of approximately $50$ steps each eigenvalue is constant, and all of them remain within $[2.79\times10^{-2},\,9.68\times10^{2}]$, consistent with the lower bound $\underline p=\lambda_{\min}(Q)=2\times10^{-2}$ of Proposition~\ref{prop:Pk_bounded}.
		Table~\ref{tab:comparison} compares the four methods.
		On this benchmark the four methods reach the same closed-loop cost to within $0.4\%$: iSCD-MPC, iLQR, and SQP agree to four significant figures, and LPV-MPC is $0.4\%$ higher, since the quadrotor linearization at hover is stabilizable and the frozen coefficients remain accurate over the maneuver.
		The methods are separated instead by computation: iSCD-MPC is $85.1$ times faster than dense SQP per step and within a factor of $4.9$ of the single-QP LPV-MPC, while attaining the accuracy that LPV-MPC sacrifices on the harder benchmarks that follow.
		\exmark
	\end{example}
	
	\begin{figure}[t!]
		\centering
		
		{\centering \includegraphics[width=\columnwidth]{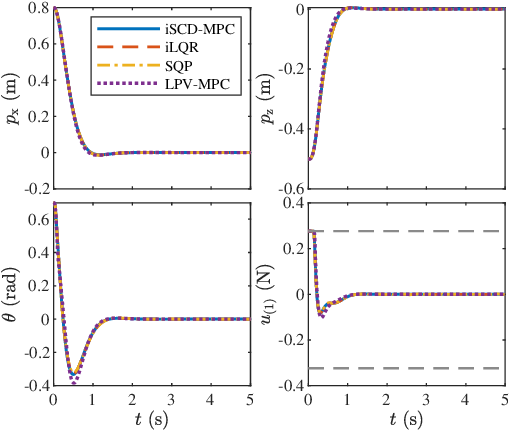}}
		\caption{Example~\ref{ex:quadrotor}. Closed-loop state trajectory and saturated controls $\sigma(u_k)$ for the planar quadrotor under iSCD-MPC with online terminal penalty adaptation \eqref{eq:Pk_update}.}
		\label{fig:quad_traj}
	\end{figure}
	
	\begin{figure}[t!]
		\centering
		
		{\centering \includegraphics[width=\columnwidth]{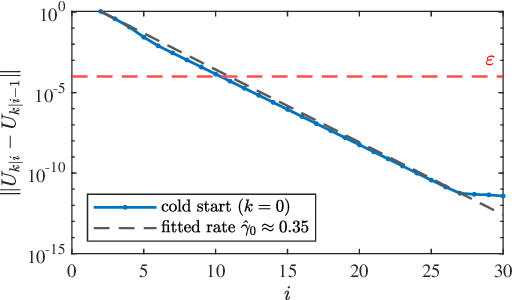}}
		\caption{Example~\ref{ex:quadrotor}. Empirical validation of Proposition~\ref{prop:contraction} at the cold start $k=0$: the increment $\|U_{k|i}-U_{k|i-1}\|$ decays geometrically at the fitted rate $\hat\gamma_0=0.354$ (dashed line), and flattens beyond $i\approx27$ where the increments reach the double-precision floor.}
		\label{fig:contraction}
	\end{figure}
	
	\begin{figure}[t!]
		\centering
		{\centering \includegraphics[width=\columnwidth]{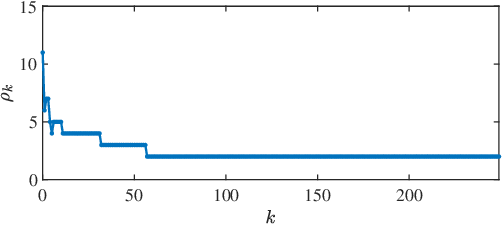}}
		\caption{Example~\ref{ex:quadrotor}. Empirical validation of Proposition~\ref{prop:iteration_bound}: the number of SCDC iterations $\rho_k$ versus the step index $k$. The count equals $11$ at the first step, where the cold start is used, and decreases to $2$ once the warm start \eqref{eq:warm_start} places the initial iterate near the fixed point; the mean is $\bar\rho_k=2.44$ over the interval and $2.34$ after the fifth step, and the cap $\rho=30$ is never reached.}
		\label{fig:rho_history}
	\end{figure}
	
	\begin{figure}[t!]
		\centering
		
		{\centering \includegraphics[width=\columnwidth]{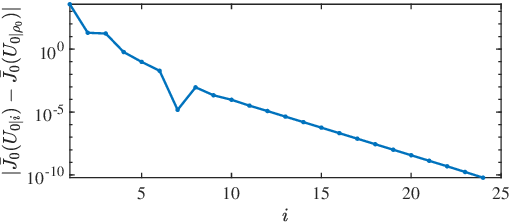}}
		\caption{Example~\ref{ex:quadrotor}. Empirical validation of Proposition~\ref{prop:cost_decrease}: magnitude of the rollout-cost error $|\bar J_0(U_{0|i})-\bar J_0(U_{0|\rho_0})|$ versus iteration $i$. The decay is geometric apart from the dip near $i=7$, at which the iterate cost passes close to its limiting value.}
		\label{fig:cost_convergence}
	\end{figure}
	
	\begin{figure}[t!]
		\centering
		
		{\centering \includegraphics[width=\columnwidth]{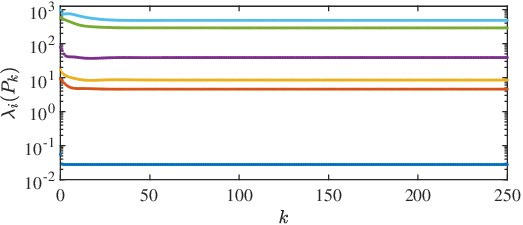}}
		\caption{Example~\ref{ex:quadrotor}. Empirical validation of Proposition~\ref{prop:Pk_bounded}: the eigenvalues of the adapted terminal penalty $P_k$ given by \eqref{eq:Pk_update} versus the step index $k$, which remain within $[\underline p,\bar p]$ and become constant after an initial transient.}
		\label{fig:Pk_eigs}
	\end{figure}
	
	\begin{table*}[t]
		\centering
		\caption{Benchmark comparison: closed-loop cost \eqref{eq:integrated_cost} normalized to iSCD-MPC (lower is better) and average per-step computation time $\bar t_{\rm step}$ in milliseconds, averaged over $M=20$ Monte Carlo trials (mean $\pm$ standard deviation). Failure to converge is denoted by ``---''; a normalized cost is not reported for a method that fails to converge, since the accumulated cost along a nonconvergent trajectory is not comparable with that along a convergent one.}
		\label{tab:comparison}
		\small
		\footnotesize
		\begin{tabular}{@{}lcccccccc@{}}
			\toprule
			& \multicolumn{4}{c}{\textit{Closed-loop cost (normalized to iSCD-MPC)}} & \multicolumn{4}{c}{\textit{$\bar t_{\rm step}$ (\rm{ms})}}\\
			\cmidrule(lr){2-5}\cmidrule(lr){6-9}
			Benchmark & iSCD-MPC & iLQR & SQP & LPV-MPC & iSCD-MPC & iLQR & SQP & LPV-MPC\\
			\midrule
			Planar quadrotor (Ex.~\ref{ex:quadrotor}, $\ell=40$) & $1.000$ & $1.000$ & $1.000$ & $1.004$ & $9.13$ & $25.57$ & $776.31$ & $1.88$\\
			Nonholonomic integrator (Ex.~\ref{ex:brockett}, $\ell=100$) & $1.000$ & $0.609$ & $0.608$ & --- & $44.11$ & $29.66$ & $682.69$ & $7.83$\\
			Output-dep.\ effectiveness (Ex.~\ref{ex:triple_int}, $\ell=40$) & $1.000$ & $1.029$ & $1.000$ & $2.264$ & $4.73$ & $49.53$ & $360.10$ & $2.20$\\
			\bottomrule
		\end{tabular}
	\end{table*}
	
	\begin{table}[t]
		\centering
		\caption{Validation of the analytical ratios \eqref{eq:ratio_SQP} and \eqref{eq:ratio_LPV}: empirical timing ratios versus predictions. The column $\bar t_{\rm iSCD}/\bar t_{\rm LPV}$ compares computation only and is therefore reported on every benchmark, including the nonholonomic integrator, on which LPV-MPC runs at the stated cost per step but does not converge. The iteration count $\bar\rho_k$ on the nonholonomic benchmark is larger than on the quadrotor because the driftless dynamics ($A=I_3$) make the reduced Hessian nearly singular, so more SCDC iterations are taken before the increment falls below the tolerance.}
		\label{tab:complexity_validation}
		\footnotesize
		\begin{tabular}{@{}lccc@{}}
			\toprule
			Benchmark & $\bar t_{\rm SQP}/\bar t_{\rm iSCD}$ & $\bar t_{\rm iSCD}/\bar t_{\rm LPV}$ & $\bar\rho_k$\\
			\midrule
			Planar quadrotor ($\ell=40$) & $85.1$ & $4.9$ & $2.44$\\
			Nonholonomic int.\ ($\ell=100$) & $15.5$ & $5.6$ & $6.08$\\
			Output-dep.\ effect.\ ($\ell=40$) & $76.2$ & $2.1$ & $2.93$\\
			\bottomrule
		\end{tabular}
	\end{table}
	
	\subsection{Nonholonomic Integrator: Failure of Frozen Coefficients}
	\label{sec:ex_brockett}
	
	\begin{example}
		\label{ex:brockett}
		{\it Nonholonomic integrator with control saturation.}
		\rm
		Consider the nonholonomic integrator
		\begin{equation}
			\dot x = \matl 1 & 0\\ 0 & 1\\ -x_{(2)} & x_{(1)}\matr \sigma(u),\label{eq:NH_dyn}
		\end{equation}
		with componentwise saturation $\sigma(u)_{(s)}=\max\{-1,\min\{u_{(s)},1\}\}$, $s\in\{1,2\}$.
		The system \eqref{eq:NH_dyn} is the normal form, under the chained-form transformation of \cite{murray1993nonholonomic}, of the kinematics of the differential-drive mobile robot shown in Figure~\ref{fig:robot_schematic}; consequently, parking the robot at a prescribed pose is equivalent to driving \eqref{eq:NH_dyn} to the origin.
		By Brockett's theorem \cite{brockett1983asymptotic}, no continuous time-invariant feedback stabilizes \eqref{eq:NH_dyn}, and the linearization at every point with $x_{(1)}=x_{(2)}=0$ is not stabilizable in the $x_{(3)}$ direction; stabilization of \eqref{eq:NH_dyn} is a standard stress test for nonlinear control methods \cite{bloch1992control,kolmanovsky1995developments,bloch1996stabilization,hespanha1999stabilization,modin2020makes}.

		\begin{figure}[t!]
			\centering
			\begin{tikzpicture}[scale=1.1, >=Stealth]
				
				% --- Coordinate Axes ---
				\draw[->, thick] (-0.4,0) -- (3.6,0) node[below] { $X$};
				\draw[->, thick] (0,-0.4) -- (0,2.6) node[left] { $Y$};
				
				% --- Robot Center & Projections (drawn below chassis) ---
				\coordinate (R) at (2.0, 1.3);
				\draw[dashed, thin, black!50] (R) -- (2.0, 0) node[below, black] { $X_{\rm r}$};
				\draw[dashed, thin, black!50] (R) -- (0, 1.3) node[left, black] { $Y_{\rm r}$};
				\draw[dashed, thin, black!50] (R) -- (3.3, 1.3);
				
				% --- Robot Scope ---
				\begin{scope}[shift={(R)}, rotate=30]
					% Main chassis (opaque white fill occludes underlying dashed lines)
					\draw[thick, fill=white, rounded corners=3pt] (-0.48,-0.32) rectangle (0.48,0.32);
					
					% Coaxial drive wheels (symmetric about the axle x = 0)
					\draw[thick, fill=black!75, rounded corners=1pt] (-0.22, 0.28) rectangle (0.22, 0.40);
					\draw[thick, fill=black!75, rounded corners=1pt] (-0.22,-0.40) rectangle (0.22,-0.28);
					
					% Center axle and longitudinal reference lines
					\draw[thin, black!25] (-0.35,0) -- (0.35,0);
					\draw[thin, black!25] (0,-0.25) -- (0,0.25);
					\fill[black] (0,0) circle (1.5pt);
					
					% Forward linear velocity v
					\draw[->, thick, blue!75!black] (0,0) -- (1.15,0) node[right=-1pt] { $v$};
					
					% Angular turning rate omega (concentric circular arc)
					\draw[->, thick, red!75!black] (25:0.62) arc (25:115:0.62) node[midway, above right=-2pt] { $\omega$};
				\end{scope}
				
				% --- Heading Angle \psi ---
				\draw[->, thick] (R) ++(0.85,0) arc (0:30:0.85);
				\node at ($(R) + (15:1.08)$) { $\psi$};
				
			\end{tikzpicture}
			\caption{Example~\ref{ex:brockett}. Differential-drive mobile robot with forward speed $v$ and turning rate $\omega$; the chained-form transformation of the pose $(X_{\rm r},Y_{\rm r},\psi)$ yields the nonholonomic integrator \eqref{eq:NH_dyn} \cite{murray1993nonholonomic}.}
			\label{fig:robot_schematic}
		\end{figure}
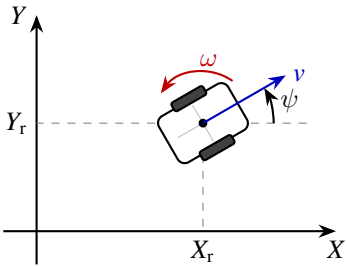
		
		The Euler discretization of \eqref{eq:NH_dyn} combined with \eqref{eq:sat_SCDC} yields the SCDC representation \eqref{eq:fSCDC} with
		\begin{equation*}
			A(x,u)=I_3,\qquad B(x,u) = \Ts\matl 1&0\\0&1\\-x_{(2)}&x_{(1)}\matr D(u).
		\end{equation*}
		The frozen pair $(I_3,B(x,u))$ is not stabilizable at any point with $x_{(1)}=x_{(2)}=0$, since the third row of $B$ vanishes there while $A=I_3$; hence, \ref{A:stabilizable} fails, and the terminal-set construction of Proposition~\ref{prop:terminal_set} is unavailable.
		Accordingly, following Remark~\ref{rem:terminal_equality}, we use the terminal equality constraint $x_{k,\ell}=0$, implemented through its exact-penalty relaxation as the terminal weighting $P=10^4 I_3$, which dominates the stage weightings by three orders of magnitude.
		Along a nonconstant predicted trajectory, however, the time-varying pair $\{(I_3,B_{k,j|i})\}_{j=1}^{\ell-1}$ generated by the rollout \eqref{eq:xkj+1i} is controllable in the $x_{(3)}$ direction, since the third row of $B_{k,j|i}$ is nonzero whenever $x_{k,j|i}$ leaves the set $x_{(1)}=x_{(2)}=0$.
		Freezing the pair at the current operating point removes this dependence on $j$ and therefore removes that controllability, which is the structural difference between the two methods on this benchmark.
		
		The simulations use $\Ts=0.05$ s, $\ell=100$ (i.e., 5-s prediction window), $\rho=30$, $\varepsilon=10^{-4}$, $Q=\diag(1,1,10)$, $R=0.1 I_2$, $u_0=10^{-3}\matl1&1\matr^\rmT$, and the initial condition $x_0=\matl 0.2&0.2&2\matr^\rmT$, for which the dominant error lies in the direction $x_{(3)}$ that is uncontrollable through the frozen coefficients.
		Figure~\ref{fig:NH_traj} shows the closed-loop trajectories: iSCD-MPC drives the state to the origin in approximately $5$ s with saturated bang-like controls, whereas LPV-MPC drives $x_{(1)}$ and $x_{(2)}$ to zero while stalling at $x_{(3)}\approx0.28$, a structural failure predicted by the loss of stabilizability of the frozen pair; iLQR and SQP also converge, at the per-step costs reported in Table~\ref{tab:comparison}.
		On this benchmark, iSCD-MPC does not attain the lowest closed-loop cost: Table~\ref{tab:comparison} shows that iLQR and SQP reach a trajectory whose cost is approximately $39\%$ lower.
		This gap is a consequence of the approximation quantified in Proposition~\ref{prop:kkt_gap} rather than of a failure of the algorithm.
		Since ${\rm OCP}_k$ is nonconvex here, its stationary points are not unique, and the fixed point of the SCDC iteration is a stationary point of the frozen-coefficient problem whose first-order residual, by \eqref{eq:KKTgap}, is $O(\|x_{k,1}\|^2)$; far from the origin, where the parking maneuver is executed, this residual is not negligible, and the derivative-based methods locate a lower-cost stationary point.
		The advantage demonstrated by this benchmark is therefore structural rather than quantitative: iSCD-MPC recovers controllability of $x_{(3)}$ that frozen-coefficient LPV-MPC irrecoverably discards, and it does so without evaluating a single Jacobian, at the price of a suboptimality gap relative to derivative-based methods.
		We turn finally to Proposition~\ref{prop:doa_expansion}.
		Reporting the convergence region at a few fixed horizons is uninformative on this benchmark, because the terminal equality constraint is reached from every gridpoint of any reasonable range once $\ell\ge8$, so that the regions coincide.
		We therefore define instead, for each initial condition, the smallest horizon that stabilizes it, by
		\begin{equation}
			\ell_{\min}(z_0)\triangleq \min\{\ell\in\SL\colon z_0\in\SF_\ell\},
			\label{eq:ellmin}
		\end{equation}
		where $\SL\triangleq\{2,3,4,5,6,8\}$ is the set of tested horizons and membership in $\SF_\ell$ is declared numerically by $\|x_N\|<10^{-2}$ with $N=400$.
		The map \eqref{eq:ellmin} carries strictly more information than the binary regions, since Proposition~\ref{prop:doa_expansion} implies that $\SF_\ell$ is exactly the sublevel set $\{z_0\colon\ell_{\min}(z_0)\le\ell\}$; the monotonicity $\SF_{\ell_1}\subseteq\SF_{\ell_2}$ is therefore visible as the nesting of the level sets of $\ell_{\min}$, and the level sets themselves quantify how much horizon each initial condition requires.
		Figure~\ref{fig:NH_DOA} displays $\ell_{\min}$ over the grid $x_{(1)}(0)=x_{(2)}(0)\in[-3,3]$ and $x_{(3)}(0)\in[-4,4]$, showing the nested level sets $\SF_2$, $\SF_3$, and $\SF_4$.
		Over the $21\times21$ grid, the level sets contain $33$, $343$, $427$, and $441$ gridpoints at $\ell=2$, $3$, $4$, and $5$, respectively, and no further growth occurs beyond $\ell=5$; the nesting is therefore strict over $\ell\in\{2,3,4,5\}$ and the whole grid lies in $\SF_5$.
		The required horizon is largest near $x_{(1)}(0)=x_{(2)}(0)=0$, which is precisely the set on which the third row of $B(x,u)$ vanishes and along which the rollout needs the most steps to recover controllability of $x_{(3)}$.
		This is the same mechanism that defeats LPV-MPC, now resolved in the horizon rather than in time: the shorter the horizon, the less the rollout can vary the coefficients, and the smaller the set of initial conditions from which the terminal constraint is reachable.
		\exmark
	\end{example}
	
	\begin{figure}[t!]
		\centering
		
		{\centering \includegraphics[width=\columnwidth]{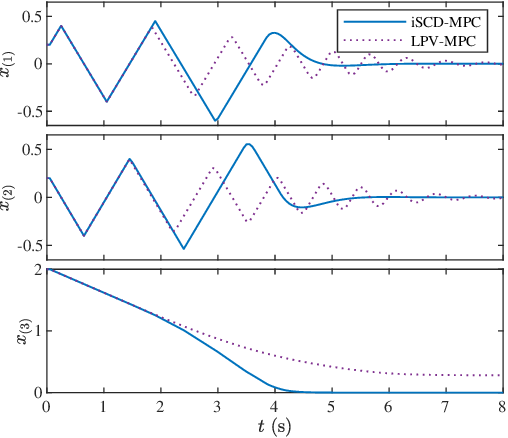}}
		\caption{Example~\ref{ex:brockett}. State trajectories and saturated controls for the nonholonomic integrator: iSCD-MPC (solid) converges to the origin, whereas LPV-MPC (dashed) stalls at $x_{(3)}\approx0.28$ due to the loss of stabilizability of the frozen coefficients.}
		\label{fig:NH_traj}
	\end{figure}
	
	\begin{figure}[t!]
		\centering
		
		\includegraphics[width=\columnwidth]{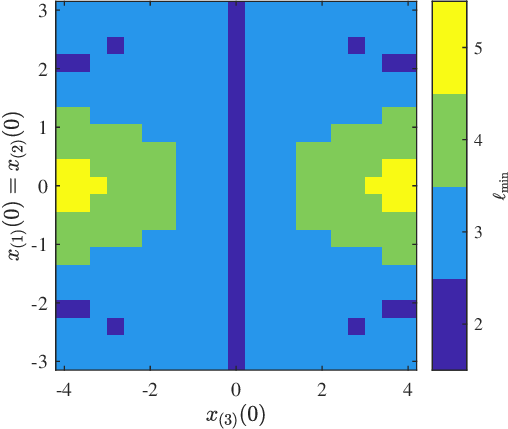}
		\caption{Example~\ref{ex:brockett}. Empirical validation of Proposition~\ref{prop:doa_expansion}: the minimum stabilizing horizon $\ell_{\min}$ of \eqref{eq:ellmin} over the $(x_{(1)}(0),x_{(3)}(0))$ plane, shown on a discrete color scale with one color per attained value of $\ell_{\min}$, so that the color boundaries are exactly the boundaries of the nested level sets $\SF_2$, $\SF_3$, and $\SF_4$. The feasible set at horizon $\ell$ is the sublevel set $\{\ell_{\min}\le\ell\}$, so the nesting of the level sets is the monotonicity of Proposition~\ref{prop:doa_expansion}. The required horizon is largest near $x_{(1)}(0)=x_{(2)}(0)=0$, where the frozen input matrix loses the $x_{(3)}$ direction.}
		\label{fig:NH_DOA}
	\end{figure}
	
	\subsection{Output Feedback of a Nonminimum-Phase Plant with Output-Dependent Actuator Effectiveness}
	\label{sec:ex_triple}
	
	\begin{example}
		\label{ex:triple_int}
		{\it Output-feedback control of a nonminimum-phase plant with output-dependent actuator effectiveness and asymmetric saturation.}
		\rm
		Consider the continuous-time triple integrator
		\begin{equation}
			\dot x = \matl 0&1&0\\0&0&1\\0&0&0\matr x+\matl 0\\0\\1\matr\sigma(u),\qquad y = \matl 1&0&0\matr x,\label{eq:triple_dyn}
		\end{equation}
		with asymmetric saturation levels $u_{\min}=-1$ and $u_{\max}=2$, whose unsaturated transfer function is $G_\rmc(s)=1/s^3$.
		The chain of integrators with bounded controls is a canonical benchmark for saturated stabilization \cite{teel1992global,kamaldar2021dynamic}.
		With zero-order-hold input and sampling period $\Ts$, the discrete-time transfer function is
		\begin{equation}
			G(\bfq) = \frac{\Ts^3}{6}\,\frac{\bfq^2+4\bfq+1}{(\bfq-1)^3},\label{eq:Gqq}
		\end{equation}
		whose numerator is independent of $\Ts$ up to the scalar gain $\Ts^3/6$, so that the sampling zeros are the same for every sampling period.
		Thus \eqref{eq:Gqq} has a pole of multiplicity three on the unit circle and sampling zeros at $\bfq=-2\pm\sqrt 3$, one of which, $-2-\sqrt3\approx-3.73$, lies outside the unit circle, so that \eqref{eq:Gqq} is nonminimum phase \cite{aastrom2013computer}.
		Since the state of \eqref{eq:triple_dyn} is not measured, we apply the output-feedback architecture of Section~\ref{sec:OF}.
		Expanding \eqref{eq:Gqq} and applying \eqref{eq:sat_SCDC} yields \eqref{eq:yyy} with $n=3$, $p=m=1$, the constant coefficients
		\begin{equation}
			F_{1,k}=-3,\qquad F_{2,k}=3,\qquad F_{3,k}=-1,\label{eq:Fcoef}
		\end{equation}
		and the control-dependent coefficients
		\begin{equation}
			G_{\tau,k}=\frac{\Ts^3}{6}\,\frac{\nu_\tau\,\sigma(u_{k-\tau})}{u_{k-\tau}},\quad \tau\in\{1,2,3\},\ k\in\BBN,\label{eq:Gcoef}
		\end{equation}
		where $(\nu_1,\nu_2,\nu_3)=(1,4,1)$ are the coefficients of the numerator of \eqref{eq:Gqq}.
		The BOCF state \eqref{eq:eta1}--\eqref{eq:eta2} is computed exactly from $(y_{k-1},y_{k-2},y_{k-3},u_{k-1},u_{k-2},u_{k-3})$, providing deadbeat reconstruction, and Proposition~\ref{prop:bocf_stability} guarantees that $y_k\to 0$ as $k\to\infty$.
		
		The actuator is driven through a valve whose underlap reduces its effectiveness near null, so that the incremental gain from the commanded input to the delivered input depends on the measured output.
		We model this by the effectiveness function $\chi\colon\BBR\to(0,1]$ defined by
		\begin{equation}
			\chi(y)\triangleq \chi_\infty + (1-\chi_\infty)\frac{(y/y_{\rm s})^2}{1+(y/y_{\rm s})^2},
			\label{eq:chi_def}
		\end{equation}
		where $\chi_\infty\in(0,1]$ is the effectiveness at the origin and $y_{\rm s}\in(0,\infty)$ is the output scale over which full effectiveness is recovered, and by replacing \eqref{eq:Gcoef} with the state- and control-dependent coefficients
		\begin{equation}
			G_{\tau,k}=\frac{\Ts^3}{6}\,\nu_\tau\,\chi(y_{k-1})\,\frac{\sigma(u_{k-\tau})}{u_{k-\tau}},\quad \tau\in\{1,2,3\},\ k\in\BBN.
			\label{eq:Gcoef_eff}
		\end{equation}
		Three properties of \eqref{eq:chi_def}--\eqref{eq:Gcoef_eff} are worth recording.
		First, $\chi$ is smooth and takes values in $[\chi_\infty,1]$, so that $B$ is bounded and globally Lipschitz and \ref{A:lipschitz} and \ref{A:bounded} hold on all of $\BBR^n\times\BBR^m$.
		Second, since $\chi$ depends on $y_{k-1}$ alone, \eqref{eq:Gcoef_eff} is of the form \eqref{eq:Gk_def}, and $G_{\tau,k+1}$ is an explicit function of $y_k=Cx_k$; hence $B(x_k,u_k)$ in \eqref{eq:BC_BOCF} is an explicit function of the BOCF state and the control, as \eqref{eq:fSCDC} requires.
		Third, the coefficients \eqref{eq:Fcoef} are unchanged, so $A$ is constant, the denominator of \eqref{eq:Gqq} is unchanged, and the sampling zeros remain at $\bfq=-2\pm\sqrt3$; the plant is therefore nonminimum phase for every $\chi_\infty$, and $\chi\equiv1$ recovers the triple integrator \eqref{eq:triple_dyn} exactly.
		Since $A$ is constant and $B$ varies with the output, this benchmark isolates the effect of refreezing the coefficients along the predicted trajectory: LPV-MPC commits to the effectiveness at the current output over the entire horizon, whereas iSCD-MPC updates it along the rollout.
		
		The simulations use $y(0)=300$, $\Ts=0.5$ s, $\ell=40$ (i.e., a $20$-s prediction window), $\rho=30$, $\varepsilon=10^{-3}$, $Q=10^{3}I_3$, $R=1$, $\chi_\infty=0.15$, $y_{\rm s}=100$, and a simulation interval of $150$ s.
		The output is held at $y(0)$ with $u=0$ before the first step, so that the initial BOCF state \eqref{eq:eta1}--\eqref{eq:eta2} is $\matl y(0)&-2y(0)&y(0)\matr^\rmT$.
		Unlike Examples~\ref{ex:quadrotor} and \ref{ex:brockett}, which are sampled-data benchmarks, the plant here is the discrete-time input-output system \eqref{eq:yyy} itself, which is the class for which Section~\ref{sec:OF} is stated; the BOCF reconstruction \eqref{eq:eta1}--\eqref{eq:eta2} is therefore exact and Proposition~\ref{prop:bocf_stability} applies without a reconstruction-error term.
		The interval is longer than in the other examples because $\chi_\infty=0.15$ leaves only $15\%$ of the nominal actuator effectiveness near the setpoint, so the final approach is necessarily slow for every method.
		The sampling period and horizon require comment.
		A prediction window of roughly $20$ s is necessary on this benchmark, because the control authority is bounded and the plant is a triple integrator: with a window of a few seconds the controller cannot see far enough ahead to begin braking, and the closed loop overshoots and diverges.
		Realizing that window at $\Ts=0.1$ s requires $\ell=200$, for which the prediction of a triple integrator spans roughly seven orders of magnitude in $\|x_{k,j}\|$ over $k\in\BBN$ and $j\in\{1,\ldots,\ell\}$ and the reduced Hessian \eqref{eq:Hessian} has condition number of order $10^{13}$; the resulting QP is not solvable in double precision, and both dense and sparse solvers fail on it.
		Since the numerator of \eqref{eq:Gqq} is independent of $\Ts$, the same $20$-s window is obtained at $\Ts=0.5$ s with $\ell=40$ on a plant with identical sampling zeros, for which the QP is well conditioned.
		Since the BOCF realization \eqref{eq:xk1_OF} is of the SCDC form \eqref{eq:fSCDC}, all four methods of Section~\ref{sec:setup} are applied to the same realization and the same reconstructed state \eqref{eq:eta1}--\eqref{eq:eta2}, so that the comparison isolates the optimization method rather than the state-reconstruction architecture.
		Figure~\ref{fig:triple_traj} shows the closed-loop output and control.
		iSCD-MPC drives the output into a $2\%$ band in $43.0$ s and SQP does so in $42.5$ s, at essentially identical closed-loop cost; iLQR is slower, at $58.0$ s and $2.9\%$ higher cost.
		LPV-MPC does not settle within the $150$-s interval and incurs a cost $126\%$ higher than iSCD-MPC.
		The reason is visible in the figure: the effectiveness frozen at the current output is not representative of the effectiveness encountered over a $20$-s window along which $y_k$ traverses the range on which $\chi$ varies, and the resulting mismatch sustains an oscillation that the trajectory-dependent coefficients of iSCD-MPC avoid.
		The comparison with SQP separates accuracy from cost on this benchmark: the two methods reach the same closed-loop cost to four significant figures, and the measured per-step time of iSCD-MPC is smaller by a factor of $76.2$.
		All methods respect the asymmetric saturation throughout, and the transient undershoot in $y_k$ is consistent with the fundamental limitation imposed by the exterior sampling zero of \eqref{eq:Gqq}.
		The closed-loop costs and per-step computation times are reported in Tables~\ref{tab:comparison} and \ref{tab:complexity_validation}; this benchmark carries the longest horizon of the three, so it provides the sharpest test of the horizon scaling \eqref{eq:ratio_SQP}.
		Setting $\chi_\infty=1$ in \eqref{eq:chi_def} reduces the plant to the triple integrator, for which $A$ and $B$ are both constant on the feasible set of \eqref{eq:QP_box}, ${\rm OCP}_k$ is a convex QP with a unique minimizer, and all four methods return the same control sequence.
		That degenerate case is the reason the nonlinearity \eqref{eq:chi_def} is introduced at all: without it the benchmark exercises the output-feedback architecture but says nothing about the iteration, since a single QP already delivers the exact minimizer.
		\exmark
	\end{example}
	
	\begin{figure}[t!]
		\centering
		
		{\centering \includegraphics[width=\columnwidth]{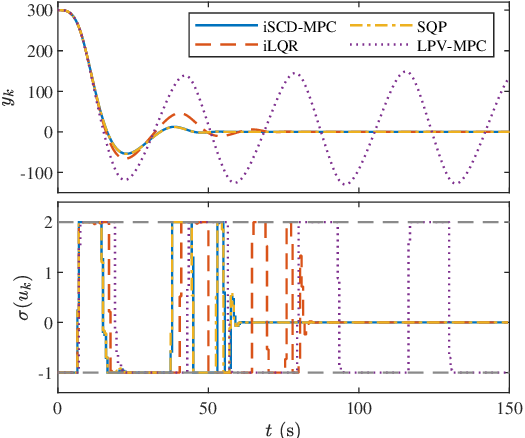}}
		\caption{Example~\ref{ex:triple_int}. Output-feedback stabilization of the nonminimum-phase plant with output-dependent actuator effectiveness \eqref{eq:chi_def} under BOCF reconstruction \eqref{eq:eta1}--\eqref{eq:eta2}, illustrating Proposition~\ref{prop:bocf_stability}. All four methods are applied to the same BOCF realization. LPV-MPC commits to the effectiveness at the current output over the entire horizon, whereas iSCD-MPC updates it along the rollout.}
		\label{fig:triple_traj}
	\end{figure}
	
	\subsection{Horizon Scaling}
	\label{sec:scaling}
	
	The three benchmarks differ in $n$, $m$, and in the nonlinearity of $f$, so the per-benchmark timings of Tables~\ref{tab:comparison} and \ref{tab:complexity_validation} cannot isolate the dependence on $\ell$ that \eqref{eq:ratio_SQP} and \eqref{eq:ratio_LPV} predict.
	This subsection fixes the quadrotor of Example~\ref{ex:quadrotor} and varies only the horizon, taking $\ell\in\{10,20,40,80,160\}$ and reporting the median per-step computation time over the first $15$ control steps.
	Assuming $\bar t_{\rm step}$ is proportional to $\ell^{\,p}$, the exponent $p$ is estimated by least squares on $\log\bar t_{\rm step}$ against $\log\ell$.
	
	One implementation detail governs what this experiment can show.
	Statement \ref{cx:iscd} of Proposition~\ref{prop:complexity} rests on the block-banded structure of the prediction, which survives only if the predicted states are retained as decision variables and the sparsity is passed to the solver.
	Condensing the states out, as in \eqref{eq:Hmap}, produces a dense reduced Hessian of order $m(\ell-1)$ and a per-step cost cubic in $\ell$, so a condensed implementation cannot attain \ref{cx:iscd} however the coefficients are frozen.
	We therefore time both variants: the condensed formulation used in Examples~\ref{ex:quadrotor}--\ref{ex:triple_int}, for which we expect $p\approx3$, and the structured formulation, for which \ref{cx:iscd} predicts $p\approx1$.
	Figure~\ref{fig:horizon_scaling} reports the measurements.
	The fitted exponents are $p_{\rm iSCD,str}=0.83$, $p_{\rm iSCD,cond}=1.90$, $p_{\rm iLQR}=0.93$, $p_{\rm SQP}=1.04$, and $p_{\rm LPV}=1.74$, and they must be read together with the regime in which each cost model applies.
	
	The structured implementation is the one to which \ref{cx:iscd} applies, and its exponent $0.83$ agrees with the predicted value of unity; the local exponents over successive doublings are $0.64$, $0.64$, $1.01$, and $1.04$, so the agreement is uniform over the range rather than an artifact of the fit.
	The exponent $0.93$ for iLQR likewise matches the leading term of \ref{cx:ilqr}, as it must, since the backward Riccati pass factors the same block-banded structure.
	These exponents also confirm that the iteration count does not grow with the horizon.
	Since \ref{cx:iscd} is the product of the per-iteration cost with $\rho_k$, an exponent near unity is attainable only if $\bar\rho_k$ is bounded independently of $\ell$; the measured mean iteration counts are flat across $\ell\in\{10,20,40,80,160\}$, so the horizon enters the per-step cost through the linear algebra alone.
	This is consistent with Proposition~\ref{prop:iteration_bound}, in which $\ell$ appears only through $\|U_{k|2}-U_{k|1}\|$, and with the fact that the warm start \eqref{eq:warm_start} appends the terminal control law, which is near-optimal on the added stages once the predicted trajectory has entered $\SX_\rmf$.
	The condensed implementation departs from \ref{cx:iscd} once the dense factorization dominates: its local exponents are $0.67$, $1.20$, and $1.60$ over the first three doublings and $4.65$ over the last, and its per-step time rises from $29.08$ ms at $\ell=80$ to $730.39$ ms at $\ell=160$, against $71.36$ ms for the structured variant at the same horizon.
	The two implementations differ only in whether the block-banded structure of \eqref{eq:Hessian} is retained, so this tenfold separation measures directly the distinction drawn at the end of Section~\ref{sec:complexity}.
	
	Two of the measurements do not test what \eqref{eq:ratio_SQP} and \eqref{eq:ratio_LPV} assert, and we record why rather than read them as refutations.
	The SQP timings are not monotone in $\ell$, falling from $1443.40$ ms at $\ell=40$ to $446.09$ ms at $\ell=80$ before rising again, which gives the local exponents $2.46$, $2.26$, $-1.69$, and $1.89$.
	A negative exponent cannot be a per-iteration cost, and its appearance shows that the measured quantity is the product of the per-iteration cost with an outer-iteration count that varies with $\ell$; convention \ref{C:ipm}, under which that count is independent of $\ell$, does not hold for {\tt fmincon}, whose termination is governed by the conditioning of the subproblem.
	Isolating the $O(\ell^2)$ ratio of \eqref{eq:ratio_SQP} would require an implementation performing a fixed number of outer iterations, such as the real-time iteration scheme \cite{diehl2002real}.
	The LPV-MPC point at $\ell=160$ is likewise not a cost measurement, since the single condensed quadratic program reached the solver iteration limit at that horizon, so that the reported $97.65$ ms records a solver failure.
	Over $\ell\in\{10,20,40,80\}$, where the LPV-MPC solve is uncontaminated, the fitted exponent is $1.13$, consistent with \ref{cx:lpv}, and the ratio $\bar t_{\rm iSCD}/\bar t_{\rm LPV}$ takes the values $11.78$, $11.04$, $8.92$, and $5.26$, decreasing by a factor of $2.2$ over the range.
	Equation \eqref{eq:ratio_LPV} predicts a ratio that does not grow with $\ell$, with which these values are consistent; they do not, however, establish that the ratio is constant, and the decrease is larger than the run-to-run scatter discussed in Section~\ref{sec:discussion}.
	
	\begin{figure}[t!]
		\centering
		{\centering \includegraphics[width=\columnwidth]{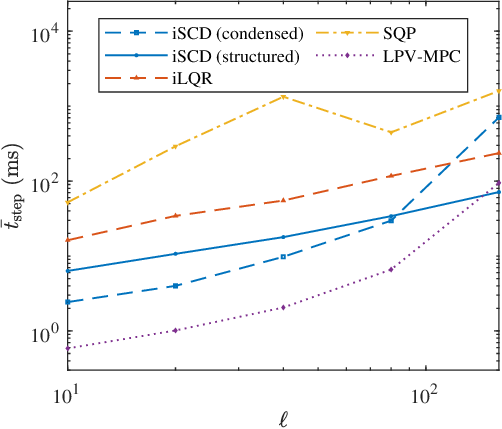}}
		\caption{Horizon scaling on the quadrotor of Example~\ref{ex:quadrotor}: median per-step computation time versus $\ell$ on logarithmic axes. The condensed and structured iSCD-MPC implementations differ only in whether the block-banded structure of the prediction is exploited; statement \ref{cx:iscd} of Proposition~\ref{prop:complexity} applies to the structured variant.}
		\label{fig:horizon_scaling}
	\end{figure}
	
	\subsection{Discussion}
	\label{sec:discussion}
	
	The three benchmarks collectively exercise every result of Sections~\ref{sec:theory} and \ref{sec:complexity}.
	The quadrotor benchmark validates the contraction estimate, the iteration bound, the geometric cost convergence, and the boundedness of the adapted terminal penalty, and quantifies the closed-loop-cost and computation-time comparison with iLQR, SQP, and LPV-MPC.
	The nonholonomic-integrator benchmark isolates the structural advantage of trajectory-dependent coefficients: LPV-MPC fails in the direction rendered uncontrollable by the frozen coefficients, whereas iSCD-MPC converges, and the empirical convergence regions are consistent with the horizon monotonicity of Proposition~\ref{prop:doa_expansion}.
	The output-feedback benchmark validates the BOCF architecture on a nonminimum-phase sampled-data plant with asymmetric saturation and output-dependent actuator effectiveness; carrying the longest horizon of the three, it provides the sharpest test of the horizon scaling \eqref{eq:ratio_SQP}, and it separates iSCD-MPC from LPV-MPC through the state dependence of the input matrix rather than of the dynamics matrix.
	The timings in Tables~\ref{tab:comparison} and \ref{tab:complexity_validation} show iSCD-MPC to be one to two orders of magnitude faster per step than dense SQP on every benchmark, but they should not be read across benchmarks as a test of the horizon scaling \eqref{eq:ratio_SQP}: the measured ratio $\bar t_{\rm SQP}/\bar t_{\rm iSCD}$ is $85.1$ at $\ell=40$ on the quadrotor and $15.5$ at $\ell=100$ on the nonholonomic integrator, whereas \eqref{eq:ratio_SQP} would predict growth by a factor of $(100/40)^2$.
	The comparison is confounded because $n$, $m$, $f$, and the number of outer iterations taken by {\tt fmincon} all change from one benchmark to the next, while \eqref{eq:ratio_SQP} holds for a fixed system as $\ell$ alone varies.
	Section~\ref{sec:scaling} therefore performs the controlled experiment.
	The overhead relative to LPV-MPC is of the order of the average iteration count, as \eqref{eq:ratio_LPV} predicts, but the two are not in a fixed relation: $\bar t_{\rm iSCD}/\bar t_{\rm LPV}=4.9$ against $\bar\rho_k=2.44$ on the quadrotor and $5.6$ against $6.08$ on the nonholonomic integrator.
	Two effects work in opposite directions.
	Each iSCD-MPC iteration performs work that \eqref{eq:ratio_LPV} does not count, namely the rollout \eqref{eq:xkj+1i}, the re-evaluation of the SCDC matrices \eqref{eq:AB_SCDC}, and the rebuilding of \eqref{eq:Phidef}--\eqref{eq:Hessian}, which inflates the ratio; conversely, the QP solved at iteration $i\ge3$ is warm started by an iterate close to its solution and terminates in fewer interior-point steps than the single cold LPV-MPC solve, which deflates it.
	Absolute per-step times also vary appreciably between runs on a general-purpose operating system, so we regard the agreement between $\bar t_{\rm iSCD}/\bar t_{\rm LPV}$ and $\bar\rho_k$ as order-of-magnitude rather than exact.
	The discrepancy is expected: \eqref{eq:ratio_LPV} counts QP solves alone, whereas each iSCD-MPC iteration also rolls out the nonlinear prediction \eqref{eq:xkj+1i}, re-evaluates the SCDC matrices \eqref{eq:AB_SCDC}, and rebuilds the block matrices \eqref{eq:Phidef}--\eqref{eq:Hessian}, none of which LPV-MPC repeats.
	The ratio \eqref{eq:ratio_LPV} is therefore a lower bound on the measured overhead.
	Summarizing the four benchmarks, iSCD-MPC attains closed-loop costs within $0.4\%$ of iLQR and SQP on the quadrotor and equal to SQP on the output-feedback plant, is $39\%$ above them on the nonholonomic integrator, and is lower than LPV-MPC on every benchmark on which LPV-MPC converges.
	Its measured per-step time lies between that of LPV-MPC and that of iLQR on all three benchmarks.
	These outcomes are consistent with the analysis: the optimality gap of Proposition~\ref{prop:kkt_gap} accounts for the shortfall relative to derivative-based methods far from the origin, and the absence of Jacobian evaluations accounts for the reduction in per-step cost relative to iLQR and SQP.
	
	\section{Conclusions}
	\label{sec:conclusions}
	
	This paper introduced iterative state- and control-dependent model predictive control (iSCD-MPC). By refreezing the pseudo-linear coefficient matrices along the previously predicted trajectory, the algorithm solves a sequence of constrained linear-quadratic programs at each control update without evaluating plant Jacobians.
	
	We proved the iteration contracts locally to a unique fixed point, bounded the required iteration count explicitly, and established the geometric convergence of the predicted cost. We quantified the first-order optimality gap of the fixed point, demonstrating that the residual vanishes quadratically near the origin. The proposed online terminal-penalty adaptation remains uniformly bounded, providing the mathematical foundation to guarantee recursive feasibility and asymptotic stability under both optimal and early-terminated suboptimal control sequences. We extended these guarantees to encompass terminal equality constraints, bounded disturbances, and output feedback implemented through the block-observable canonical form.
	
	The operation count places the per-step cost of iSCD-MPC at $O(\ell)$ in the horizon length, which is the scaling of the iterative linear quadratic regulator and is two orders in $\ell$ below that of dense sequential quadratic programming, and it does so without evaluating derivatives of the dynamics.
	The numerical study confirms the $O(\ell)$ scaling for the structured implementation and identifies the conditions under which the remaining predicted ratios are not observed.
	Refreezing the coefficients along the predicted trajectory also allows the horizon to be used on systems for which the frozen pair loses stabilizability at the operating point, as the nonholonomic integrator illustrates.
	
	Future research will expand the framework to address stochastic environments and broader classes of nonlinear systems. To handle measurement maps lacking exact deadbeat reconstruction, we will integrate the controller with state- and control-dependent moving-horizon estimation \cite{kamaldarSCDMHE}. We will also develop tube- and scenario-based formulations to robustify the algorithm against stochastic disturbances \cite{mayne2005robust,mayne2011tube,rakovic2005invariant}. On the theoretical front, quantifying the domain of attraction as a function of the horizon length remains an open priority \cite{damm2014exponential}. Finally, treating the nonunique SCDC factorization as an offline design variable will allow us to minimize the Lipschitz constants that govern both the contraction rate \eqref{eq:xk1_small} and the optimality gap \eqref{eq:KKTgap}.
	
	\bibliographystyle{unsrt}
	\bibliography{Ref}

\begin{thebibliography}{10}

\bibitem{keerthi1988optimal}
S.~S. Keerthi and E.~G. Gilbert.
\newblock Optimal infinite-horizon feedback laws for a general class of
  constrained discrete-time systems: Stability and moving-horizon
  approximations.
\newblock {\em J. Optim. Theory Appl.}, 57(2):265--293, 1988.

\bibitem{kwon2006receding}
W.~Kwon and S.~Han.
\newblock {\em Receding Horizon Control: Model Predictive Control for State
  Models}.
\newblock Springer, 2006.

\bibitem{camacho2013model}
E.~F. Camacho and C.~Bordons.
\newblock {\em Model Predictive Control}.
\newblock Springer, 2nd edition, 2007.

\bibitem{mayne2000constrained}
David~Q. Mayne, James~B. Rawlings, Christopher~V. Rao, and Pierre O.~M.
  Scokaert.
\newblock Constrained model predictive control: Stability and optimality.
\newblock {\em Automatica}, 36(6):789--814, 2000.

\bibitem{rawlings2017model}
James~B. Rawlings, David~Q. Mayne, and Moritz Diehl.
\newblock {\em Model Predictive Control: Theory, Computation, and Design}.
\newblock Nob Hill Publishing, Madison, WI, 2nd edition, 2017.

\bibitem{eren2017model}
U.~Eren, A.~Prach, B.~B. Ko{\c{c}}er, S.~V. Rakovi{\'c}, E.~Kayacan, and
  B.~A{\c{c}}{\i}kme{\c{s}}e.
\newblock Model predictive control in aerospace systems: Current state and
  opportunities.
\newblock {\em J. Guid. Contr. Dyn.}, 40(7):1541--1566, 2017.

\bibitem{cairano2007model}
S~Di Cairano, Alberto Bemporad, Ilya~V Kolmanovsky, and Davor Hrovat.
\newblock Model predictive control of magnetically actuated mass spring dampers
  for automotive applications.
\newblock {\em Int. J. Contr.}, 80(11):1701--1716, 2007.

\bibitem{qin2003survey}
S.~Joe Qin and Thomas~A. Badgwell.
\newblock A survey of industrial model predictive control technology.
\newblock {\em Contr. Eng. Pract.}, 11(7):733--764, 2003.

\bibitem{kuhne2004model}
F.~K{\"u}hne, W.~F. Lages, and J.~M.~G. da~Silva~Jr.
\newblock Model predictive control of a mobile robot using linearization.
\newblock In {\em Proc. Mechatronics Robot.}, pages 525--530, 2004.

\bibitem{mayne2014model}
David~Q. Mayne.
\newblock Model predictive control: Recent developments and future promise.
\newblock {\em Automatica}, 50(12):2967--2986, 2014.

\bibitem{allgower2012nonlinear}
Frank Allg{\"o}wer and Alex Zheng.
\newblock {\em Nonlinear Model Predictive Control}.
\newblock Birkh{\"a}user, Basel, 2012.

\bibitem{nocedal2006numerical}
Jorge Nocedal and Stephen~J. Wright.
\newblock {\em Numerical Optimization}.
\newblock Springer, New York, 2nd edition, 2006.

\bibitem{diehl2002real}
Moritz Diehl, Hans~Georg Bock, Johannes~P. Schl{\"o}der, Rolf Findeisen, Zoltan
  Nagy, and Frank Allg{\"o}wer.
\newblock Real-time optimization and nonlinear model predictive control of
  processes governed by differential-algebraic equations.
\newblock {\em J. Process Contr.}, 12(4):577--585, 2002.

\bibitem{wachter2006implementation}
Andreas W{\"a}chter and Lorenz~T. Biegler.
\newblock On the implementation of an interior-point filter line-search
  algorithm for large-scale nonlinear programming.
\newblock {\em Math. Program.}, 106(1):25--57, 2006.

\bibitem{diehl2005nominal}
Moritz Diehl, Rolf Findeisen, Frank Allg{\"o}wer, Hans~Georg Bock, and
  Johannes~P. Schl{\"o}der.
\newblock Nominal stability of real-time iteration scheme for nonlinear model
  predictive control.
\newblock {\em IEE Proc. Contr. Theory Appl.}, 152(3):296--308, 2005.

\bibitem{frison2020hpipm}
G.~Frison and M.~Diehl.
\newblock {HPIPM}: A high-performance quadratic programming framework for model
  predictive control.
\newblock {\em IFAC-PapersOnLine}, 53(2):6563--6569, 2020.

\bibitem{verschueren2022acados}
R.~Verschueren, G.~Frison, D.~Kouzoupis, J.~Frey, N.~{van Duijkeren},
  A.~Zanelli, B.~Novoselnik, T.~Albin, R.~Quirynen, and M.~Diehl.
\newblock acados---a modular open-source framework for fast embedded optimal
  control.
\newblock {\em Math. Program. Comput.}, 14(1):147--183, 2022.

\bibitem{ohtsuka2004continuation}
Toshiyuki Ohtsuka.
\newblock A continuation/{GMRES} method for fast computation of nonlinear
  receding horizon control.
\newblock {\em Automatica}, 40(4):563--574, 2004.

\bibitem{jacobson1970differential}
David~H. Jacobson and David~Q. Mayne.
\newblock {\em Differential Dynamic Programming}.
\newblock Elsevier, New York, 1970.

\bibitem{LiTodorov2004}
W.~Li and E.~Todorov.
\newblock Iterative linear quadratic regulator design for nonlinear biological
  movement systems.
\newblock In {\em ICINCO}, pages 222--229, 2004.

\bibitem{TodorovLi2005}
E.~Todorov and W.~Li.
\newblock A generalized iterative {LQG} method for locally-optimal feedback
  control of constrained nonlinear stochastic systems.
\newblock In {\em Proc. Amer. Contr. Conf.}, pages 300--306, 2005.

\bibitem{tomizukaiLQR}
J.~Chen, W.~Zhan, and M.~Tomizuka.
\newblock Constrained iterative {LQR} for on-road autonomous driving motion
  planning.
\newblock In {\em IEEE Int. Conf. Intell. Trans. Sys.}, pages 1--7, 2017.

\bibitem{cisneros2018lpv}
Pablo S.~G. Cisneros and Herbert Werner.
\newblock A dissipativity formulation for stability analysis of nonlinear and
  parameter dependent {MPC}.
\newblock {\em Automatica}, 101:246--253, 2018.

\bibitem{morato2020model}
Marcelo~M. Morato, Julio~E. Normey-Rico, and Olivier Sename.
\newblock Model predictive control design for linear parameter varying systems:
  A survey.
\newblock {\em Ann. Rev. Contr.}, 49:64--80, 2020.

\bibitem{cloutier1996nonlinear}
James~R. Cloutier, Christopher~N. D'Souza, and Curtis~P. Mracek.
\newblock Nonlinear regulation and nonlinear ${H_\infty}$ control via the
  state-dependent {Riccati} equation technique: Part 1, theory.
\newblock In {\em Proc. Int. Conf. Nonlinear Probl. Aviation Aerospace}, pages
  117--130, 1996.

\bibitem{tsiotras1996counterexample}
P.~Tsiotras, M.~Corless, and M.~Rotea.
\newblock Counterexample to a recent result on the stability of nonlinear
  systems.
\newblock {\em IMA J. Math. Contr. Info.}, 13(2):129--130, 1996.

\bibitem{mracek1998control}
C.~P. Mracek and J.~R. Cloutier.
\newblock Control designs for the nonlinear benchmark problem via the
  state-dependent {Riccati} equation method.
\newblock {\em Int. J. Robust Nonlinear Contr.}, 8(4-5):401--433, 1998.

\bibitem{findeisen2003state}
R.~Findeisen, L.~Imsland, F.~Allgower, and B.~A. Foss.
\newblock State and output feedback nonlinear model predictive control: An
  overview.
\newblock {\em Eur. J. Contr.}, 9(2-3):190--206, 2003.

\bibitem{erdem2004design}
E.~B. Erdem and A.~G. Alleyne.
\newblock Design of a class of nonlinear controllers via state dependent
  {Riccati} equations.
\newblock {\em IEEE Trans. Contr. Sys. Tech.}, 12(1):133--137, 2004.

\bibitem{ccimen2010systematic}
T.~{\c{C}}imen.
\newblock Systematic and effective design of nonlinear feedback controllers via
  the state-dependent {Riccati} equation ({SDRE}) method.
\newblock {\em Ann. Rev. Contr.}, 34(1):32--51, 2010.

\bibitem{weiss2012forward}
A.~Weiss, I.~Kolmanovsky, M.~Baldwin, R.~S. Erwin, and D.~S. Bernstein.
\newblock Forward-integration {Riccati}-based feedback control for spacecraft
  rendezvous maneuvers on elliptic orbits.
\newblock In {\em Proc. Conf. Dec. Contr.}, pages 1752--1757, 2012.

\bibitem{prach2018output}
A.~Prach, O.~Tekinalp, and D.~S. Bernstein.
\newblock Output-feedback control of linear time-varying and nonlinear systems
  using the forward propagating {Riccati} equation.
\newblock {\em J. Vib. Contr.}, 24(7):1239--1263, 2018.

\bibitem{chang2013constrained}
L.~Chang and J.~Bentsman.
\newblock Constrained discrete-time state-dependent {Riccati} equation
  technique: A model predictive control approach.
\newblock In {\em Proc. Conf. Dec. Contr.}, pages 5125--5130, 2013.

\bibitem{ru2017nonlinear}
P.~Ru and K.~Subbarao.
\newblock Nonlinear model predictive control for unmanned aerial vehicles.
\newblock {\em Aerospace}, 4(2):31, 2017.

\bibitem{kamaldarSCDMHE}
Mohammadreza Kamaldar.
\newblock Nonlinear moving-horizon estimation using state- and
  control-dependent models.
\newblock In {\em Proc. IEEE Conf. Decis. Control}, 2026.

\bibitem{brockett1983asymptotic}
R.~W. Brockett.
\newblock Asymptotic stability and feedback stabilization.
\newblock In R.~W. Brockett, R.~S. Millman, and H.~J. Sussmann, editors, {\em
  Differential Geometric Control Theory}, pages 181--191. Birkh\"{a}user, 1983.

\bibitem{Rao2003}
Christopher~V. Rao, James~B. Rawlings, and David~Q. Mayne.
\newblock Constrained state estimation for nonlinear discrete-time systems:
  stability and moving horizon approximations.
\newblock {\em IEEE Trans. Autom. Contr.}, 48(2):246--258, 2003.

\bibitem{Alessandri2008}
Angelo Alessandri, Marco Baglietto, and Giorgio Battistelli.
\newblock Moving-horizon state estimation for nonlinear discrete-time systems:
  new stability results and approximation schemes.
\newblock {\em Automatica}, 44(7):1753--1765, 2008.

\bibitem{scokaert1999suboptimal}
P.~O.~M. Scokaert, D.~Q. Mayne, and J.~B. Rawlings.
\newblock Suboptimal model predictive control (feasibility implies stability).
\newblock {\em IEEE Trans. Autom. Contr.}, 44(3):648--654, 1999.

\bibitem{polderman1989state}
J.~W. Polderman.
\newblock A state space approach to the problem of adaptive pole assignment.
\newblock {\em Math. Contr. Sig. Sys.}, 2:71--94, 1989.

\bibitem{islamPCAC}
T.~W. Nguyen, S.~A.~U. Islam, D.~S. Bernstein, and I.~V. Kolmanovsky.
\newblock Predictive cost adaptive control: A numerical investigation of
  persistency, consistency, and exigency.
\newblock {\em IEEE Contr. Sys. Mag.}, 41:64--96, 2021.

\bibitem{kamaldar2023arxiv}
Mohammadreza Kamaldar and Dennis~S. Bernstein.
\newblock Output-feedback nonlinear model predictive control with iterative
  state- and control-dependent coefficients.
\newblock arXiv:2309.11589, 2023.

\bibitem{alhazmi2025nonlinear}
Rami Alhazmi, Syed Aseem~Ul Islam, and Dennis~S. Bernstein.
\newblock Nonlinear model predictive guidance for {6DOF} pursuer/evader
  engagements.
\newblock In {\em Proc. AIAA SCITECH Forum}, page 0102, 2025.

\bibitem{alhazmi2025root}
Rami~Abdulelah Alhazmi, Syed Aseem~Ul Islam, and Dennis~S. Bernstein.
\newblock Application of root-finding methods to iterative model predictive
  control of pseudo-linear systems.
\newblock In {\em Proc. Amer. Contr. Conf.}, 2025.

\bibitem{rao1998application}
Christopher~V. Rao, Stephen~J. Wright, and James~B. Rawlings.
\newblock Application of interior-point methods to model predictive control.
\newblock {\em J. Optim. Theory Appl.}, 99(3):723--757, 1998.

\bibitem{chen1998quasi}
H.~Chen and F.~Allg{\"o}wer.
\newblock A quasi-infinite horizon nonlinear model predictive control scheme
  with guaranteed stability.
\newblock {\em Automatica}, 34(10):1205--1217, 1998.

\bibitem{Anderson1979}
Brian D.~O. Anderson and John~B. Moore.
\newblock {\em Optimal Filtering}.
\newblock Prentice-Hall, Englewood Cliffs, NJ, 1979.

\bibitem{khalil2002nonlinear}
Hassan~K. Khalil.
\newblock {\em Nonlinear Systems}.
\newblock Prentice Hall, Upper Saddle River, NJ, 3rd edition, 2002.

\bibitem{anderson1981detectability}
B.~D.~O. Anderson and J.~B. Moore.
\newblock Detectability and stabilizability of time-varying discrete-time
  linear systems.
\newblock {\em SIAM J. Contr. Optim.}, 19(1):20--32, 1981.

\bibitem{bernstein2018scalar}
Dennis~S. Bernstein.
\newblock {\em Scalar, Vector, and Matrix Mathematics: Theory, Facts, and
  Formulas}.
\newblock Princeton University Press, 2018.

\bibitem{mayne2005robust}
David~Q. Mayne, Mar{\'\i}a~M. Seron, and Sa{\v{s}}a~V. Rakovi{\'c}.
\newblock Robust model predictive control of constrained linear systems with
  bounded disturbances.
\newblock {\em Automatica}, 41(2):219--224, 2005.

\bibitem{mayne2011tube}
D.~Q. Mayne, E.~C. Kerrigan, E.~J. {van Wyk}, and P.~Falugi.
\newblock Tube-based robust nonlinear model predictive control.
\newblock {\em Int. J. Robust Nonlinear Contr.}, 21(11):1341--1353, 2011.

\bibitem{rakovic2005invariant}
Sa{\v{s}}a~V. Rakovi{\'c}, Eric~C. Kerrigan, Konstantinos~I. Kouramas, and
  David~Q. Mayne.
\newblock Invariant approximations of the minimal robust positively invariant
  set.
\newblock {\em IEEE Trans. Autom. Contr.}, 50(3):406--410, 2005.

\bibitem{giernacki2017crazyflie}
W.~Giernacki, M.~Skwierczy{\'n}ski, W.~Witwicki, P.~Wro{\'n}ski, and
  P.~Kozierski.
\newblock Crazyflie 2.0 quadrotor as a platform for research and education in
  robotics and control engineering.
\newblock In {\em Proc. Int. Conf. Methods Models Autom. Robot.}, pages 37--42,
  2017.

\bibitem{murray1993nonholonomic}
Richard~M Murray and S~Shankar Sastry.
\newblock Nonholonomic motion planning: Steering using sinusoids.
\newblock {\em IEEE Trans. Autom. Contr.}, 38(5):700--716, 1993.

\bibitem{bloch1992control}
AM~Bloch, M~Reyhanoglu, and NH~McClamroch.
\newblock Control and stabilization of nonholonomic dynamic systems.
\newblock {\em IEEE Trans. Autom. Contr.}, 37(11):1746--1757, 1992.

\bibitem{kolmanovsky1995developments}
Ilya Kolmanovsky and N~Harris McClamroch.
\newblock Developments in nonholonomic control problems.
\newblock {\em IEEE Contr. Sys. Mag.}, 15(6):20--36, 1995.

\bibitem{bloch1996stabilization}
Anthony Bloch and Sergey Drakunov.
\newblock Stabilization and tracking in the nonholonomic integrator via sliding
  modes.
\newblock {\em Sys. Contr. Lett.}, 29(2):91--99, 1996.

\bibitem{hespanha1999stabilization}
Joao~P Hespanha and A~Stephen Morse.
\newblock Stabilization of nonholonomic integrators via logic-based switching.
\newblock {\em Automatica}, 35(3):385--393, 1999.

\bibitem{modin2020makes}
K~Modin and O~Verdier.
\newblock What makes nonholonomic integrators work?
\newblock {\em Numer. Math.}, 145(2):405--435, 2020.

\bibitem{teel1992global}
Andrew~R Teel.
\newblock Global stabilization and restricted tracking for multiple integrators
  with bounded controls.
\newblock {\em Sys. Contr. Lett.}, 18(3):165--171, 1992.

\bibitem{kamaldar2021dynamic}
Mohammadreza Kamaldar and Dennis~S Bernstein.
\newblock Dynamic output-feedback control of a chain of discrete-time
  integrators with arbitrary zeros and asymmetric input saturation.
\newblock {\em Automatica}, 125:109387, 2021.

\bibitem{aastrom2013computer}
Karl~J {\AA}str{\"o}m and Bj{\"o}rn Wittenmark.
\newblock {\em Computer-Controlled Systems: Theory and Design}.
\newblock Courier Corporation, 2013.

\bibitem{damm2014exponential}
Tobias Damm, Lars Grune, Marleen Stieler, and Karl Worthmann.
\newblock An exponential turnpike theorem for dissipative discrete time optimal
  control problems.
\newblock {\em SIAM J. Contr. Optim.}, 52(3):1935--1957, 2014.

\end{thebibliography}
	
\end{document}